\documentclass[11pt]{article}
\usepackage{color}
\usepackage[notref,notcite,final]{showkeys}
\usepackage{amsfonts}
\usepackage{amsmath,amsthm,amssymb}
\usepackage{enumerate}
\numberwithin{equation}{section}
\newtheorem{theorem}[equation]{Theorem}
\newtheorem{lemma}[equation]{Lemma}
\newtheorem{proposition}[equation]{Proposition}
\newtheorem{definition}[equation]{Definition}
\newtheorem{corollary}[equation]{Corollary}

\theoremstyle{remark}
\newtheorem{rem}[equation]{Remark}
\newtheorem{ex}[equation]{Example}

\newtheorem{notation}[equation]{Notation}
\newtheorem{terminology}[equation]{Terminology}
\newtheorem{notation/terminology}[equation]{Notation/Terminology}
\input xy
\xyoption{all}
\xyoption{2cell}

\newcommand{\Map}{\textrm{Map}}

\newcommand{\tocong}{\xymatrix@C5mm{\ar[r]^-{\cong}&}}
\newcommand{\toocong}{\xymatrix@C5mm{\ar@{=>}[r]^-{\cong}&}}
\newcommand{\Bone}{\textrm{B}(\ul{i\Omega}_M^1)}
\newcommand{\BiR}{\textrm{B}(\ul{i\R}_M)}

\newcommand{\BT}{\textrm{B}(\ul{\T}_M)}

\newcommand{\bp}{\boxplus}
\renewcommand{\L}{\mathcal{L}}

\newcommand{\R}{\mathbb{R}}
\newcommand{\cx}{\check{\xi}}
\newcommand{\ce}{\check{\eta}}

\newcommand{\Z}{\mathbb{Z}}
\newcommand{\Aut}{\textrm{Aut}}

\newcommand{\T}{\mathbb{T}}
\newcommand{\Hom}{\textrm{Hom}}

\newcommand{\A}{\mathcal{A}}

\newcommand{\g}{\mathfrak{g}}

\newcommand{\F}{\mathcal{F}}

\newcommand{\tx}{\hat{\xi}}
\newcommand{\te}{\hat{\eta}}

\newcommand{\CC}{\mathcal{G}}

\newcommand{\tor}{\textbf{Tor}}
\newcommand{\ctor}{\textbf{\v{T}or}}
\newcommand{\ul}{\underline}
\renewcommand{\F}{\mathcal{F}}

\title{Infinitesimal Symmetries of Dixmier-Douady Gerbes}
\author{Braxton Collier}
\begin{document}

\maketitle 
\abstract{We introduce the infinitesimal symmetries of Dixmier-Douady gerbes over a manifold $M$, both with and without connective structures and curvings.  We explore the algebraic structure possessed by these symmetries, and relate them to equivariant gerbes via a ``differentiation functor".  In the case that a gerbe $\CC$ is equipped with a connective structure $\A$, we give a new construction of the Courant algebroid associated to $(\CC,\A)$ in terms of the infinitesimal symmetries of $(\CC,\A)$.}

 \section{\label{intro}Introduction}
Dixmier-Douady gerbes over a manifold $M$ are geometric representatives of elements in the third integral cohomology of $M$: equivalence classes of gerbes are in one-to-one correspondence with elements of $H^3(M,\Z)$ \cite{Br1}.  One advantage of working with the gerbes themselves, as opposed to their equivalence classes, is that geometric objects have symmetries.  In this paper we explore the infinitesimal symmetries of gerbes.  We investigate the structures that the collection of these symmetries possesses, and how they relate to the geometry of $M$.

Various models of symmetries of gerbes exist in the literature.  On the one hand, various models of equivariant gerbes (i.e. gerbes with families of non-infinitesimal symmetries) have been developed by different authors \cite{Br1}\cite{Br2}\cite{G1}\cite{Mein}.  On the other hand, as explained in \cite{H}, starting from a gerbe $\CC$ with connective structure $\A$ over $M$, one can construct an exact Courant algebroid $E_{(\CC,\A)}$ over $M$.   As pointed out by a number of authors \cite{BCh}\cite{Gu}\cite{R2}, this construction is analogous to the construction of the Atiyah algebroid associated to a principal circle bundle $P\to M$.  As discussed below, sections of the Atiyah algebroid are infinitesimal symmetries of $P$, so by analogy it has been suggested that $E_{(\CC,\A)}$ should encode the infinitesimal symmetries of the gerbe with connective structure. Furthermore, it follows from the work of \cite{RW} that any Courant algebroid  canonically determines a 2-term $L_{\infty}$-algebra \cite{R1}; this is the type of algebraic structure one might expect to describe the symmetries of a gerbe.  The authors of \cite{BCh} point out that this analogy between the Atiyah algebroid and the Courant algebroid $E_{(\CC,\A)}$ is somewhat formal, however, and they remark that it would be desirable to have a more precise understanding of the sense in which the Courant algebroid is related to the symmetries of the gerbe.  In this paper we establish such a direct connection, giving a geometric construction of the Courant algebroid in terms of the infinitesimal symmetries of the gerbe with connective structure.
 
To understand the types of symmetries in which we are interested, consider by analogy a principal circle bundle $P$ over $M$; such a bundle represents an element of $H^2(M,\Z)$.  A symmetry of $P$ is a diffeomorphism $\Phi:P\to P$ commuting with the action of the circle group $\T$ on $P$.  Such a diffeomorphism necessarily covers a diffeomorphism $\varphi$ of $M$.   The group of such symmetries fits into a group extension 
\begin{equation}\label{group extension} \xymatrix{ 1 \ar[r] & C^{\infty}_M(\T) \ar[r] & \textrm{Diff}^{\T}(P) \ar[r] & \textrm{Diff}_{P}(M)\ar[r] & 1}, \end{equation}
where $\textrm{Diff}_{P}(M)$ is the subgroup of the diffeomorphisms of $M$ which preserve the isomorphism class of $P$. The infinitesimal version of this extension  
\begin{equation} \label{inf extension} \xymatrix{ 0 \ar[r] & C^{\infty}_M(i\R) \ar[r] & C^{\infty}(TP)^{\T} \ar[r] & C^{\infty}(TM) \ar[r] & 0}\end{equation} is obtained by replacing $\textrm{Diff}_P(M)$ and $\textrm{Diff}^{\T}(P)$ by the Lie algebras of vector fields on $M$ and $\T$-invariant vector fields on $P$, respectively; here we identify $C^{\infty}_M(i\R)$ with the set of $\T$-invariant vertical vector fields on $P$ (with trivial Lie algebra structure).    Note also that vector fields can be constructed in terms of local data, and the extension (\ref{inf extension}) can be refined to an extension of sheaves over $M$.  These sheaves may be viewed as sections of the Atiyah sequence of vector bundles \begin{equation}\label{atiyah algebroid} \xymatrix{ 0 \ar[r] & M\times i\R\ \ar[r] & TP/\T\ar[r] &  TM \ar[r] & 0.}\end{equation}  In particular, for each (global) vector field $\xi$ on $M$, we have a principal $i\R$-bundle whose sections are local ``lifts" of $\xi$ to $P$.

Similarly, given a gerbe $\CC$ over $M$, in this paper we consider symmetries of $\CC$ that cover symmetries of $M$.  For example,  given a diffeomorphism $\varphi$ of $M$, a symmetry of $\CC$ lifting $\varphi$ is an isomorphism of gerbes 
\begin{equation} \hat{\varphi}:\CC\tocong \varphi^*\CC,\end{equation} where $\varphi^*\CC$ is the pull-back of $\CC$ via $\varphi$.  The collection of all such lifts comprise the objects of a category $\L_{\CC}(\varphi)$.  For example, the category $\L_{\CC}(\textrm{id}_M)$ of symmetries of $\CC$ covering the identity on $M$ is naturally isomorphic to the category $\textrm{Bund}_M(\T)$ of principal circle bundles over $M$. 
 If we let $\textrm{Sym}(\CC)$ denote the disjoint union of the categories $\L_{\CC}(\varphi)$ for $\varphi$ ranging over all $\varphi\in \textrm{Diff}(M)$, then both $\textrm{Sym}(\CC)$ and $\textrm{Bund}_M(\T)$ have the structure of coherent 2-groups \cite{BL}; for example multiplication in $\textrm{Bund}_M(\T)$ is given by tensor product, with the trivial bundle $M\times \T$ acting as the identity.  Moreover, regarding $\textrm{Diff}(M)$ as a category with only identity morphisms, it too has the structure of a coherent 2-group, and we obtain a sequence of homomorphisms of 2-groups \begin{equation} \label{category seq}\xymatrix{1\ar[r] &  \textrm{Bund}_M(\T)\ar[r] & \textrm{Sym}(\CC) \ar[r] & \textrm{Diff}(M) \ar[r]& 1,} \end{equation} which is ``exact" in an appropriate sense, here $1$ denotes the category with one object and one morphism.

The first new structure introduced in this paper is an infinitesimal version of the category $\textrm{Sym}(\CC)$, which we call $\L_{\CC}$.   More generally, we define a sheaf of groupoids $\ul{\L}_{\CC}$ over $M$ of infinitesimal symmetries of $\CC$ whose category of global sections is $\L_{\CC}$.   This sheaf, which is analogous to the vector bundle $TP/\T$ in the circle bundle case, is equipped with a surjective projection functor $ \pi:\ul{\L}_{\CC}\to \ul{TM}$, where the sheaf $\ul{TM}$ is considered as a sheaf of groupoids with only identity morphisms.  Furthermore, for each vector field $
\xi$ on $M$, the subsheaf $\ul{\L}_{\CC}(\xi)$ of symmetries of $\CC$ projecting to (restrictions of) $\xi$ is a gerbe with band $i\R$;  in the case $\xi=0$ this gerbe is naturally isomorphic to the trivial gerbe $\textrm{B}(\ul{i\R}_M)$ of principal $i\R$-bundles over $M$.  Although we do not develop this definition in detail, both $\L_{\CC}$ and $\textrm{B}(\ul{i\R}_M)$ should be regarded as ``sheaves of coherent Lie 2-algebras," and we have a sequence
\begin{equation} \xymatrix{\textrm{B}(\ul{i\R}_M) \ar[r] & \ul{\L}_{\CC} \ar[r] & \ul{TM}.}\end{equation}   

In the case of a circle bundle $P\to M$, one way of relating non-infinitesimal symmetries to infinitesimal ones is via differentiation.  Given a 1-parameter subgroup $t\mapsto \varphi_t$ of $\textrm{Diff}(M)$, we may differentiate at $t=0$ to obtain a vector field $\xi$ on $M$.  If $\{\hat{\varphi}_t\}$ is a 1-parameter subgroup of $\textrm{Diff}^{\T}(P)$ covering $\{\varphi_t\}$, then differentiating at $t=0$ gives a $\T$-invariant vector field $\hat{\xi}$ on $P$ covering $\xi$.  In the case of a gerbe $\CC\to M$, there is category $\L_{\CC}(\{\varphi_t\})$ of lifts of $\{\varphi_t\}$ to $\CC$.  We construct a ``differentiation" functor 
\begin{equation} D:\L_{\CC}(\{\varphi_t\})\to \L_{\CC}(\xi),\end{equation} where $\L_{\CC}(\xi)$ is the category of global sections of $\ul{\L}_{\CC}$ lifting $\xi$. Furthermore, we prove that,  restricted to a suitably small open set on $M$ and suitably small $t$,  $D$ gives an equivalence of categories.  This is analogous to the statement in differential geometry that every vector field on a manifold can locally be integrated to a unique flow.  It justifies our assertion that $\L_{\CC}$ encodes the infinitesimal symmetries of $\CC$.



  One advantage of our approach is that it elucidates the role played by connective structures and curvings on $\CC$, which together generalize the notion of a connection on a principal $\T$-bundle.  Note in particular that the construction of the Courant algebroid $E_{(\CC,\A)}$ depends upon the choice of a connective structure, whereas the construction of the Atiyah algebroid (\ref{atiyah algebroid}) makes no reference to a connection on $P$.  Connective structures and curvings enter our picture in two distinct, but related, ways.  First, recall that a connection $\Theta$ on $P$ is equivalent to a splitting of the Atiyah sequence (\ref{atiyah algebroid}) of vector bundles. Furthermore, at the level of global sections, each term in the extension (\ref{inf extension}) has the structure of a Lie algebra with respect to the bracket operation (vertical vector fields commute so that bracket on $C^{\infty}_M(i\R)$ is trivial).  $\Theta$  produces a \emph{linear} splitting of the exact sequence (\ref{inf extension}), and the curvature of $\Theta$ measures the failure of this splitting to be a splitting of Lie algebras.  
  
  Similarly, in \S\ref{operations} we construct  operations on $\L_{\CC}$ analogous to those possessed by a Lie algebra.  This structure is most concretely described in a Cech-type picture relative to a collection of local trivializations for $\CC$, and we therefore introduce a Cech version $\L_{g_{ijk}}$ of the category of infinitesimal symmetries.   The category $\L_{g_{ijk}}$ then obtains the structure of a Lie 2-algebra, or alternatively a 2-term $L_{\infty}$-algebra \cite{BC}.  We can also regard $C^{\infty}(TM)$ as a Lie 2-algebra, and from this point of view we have a strict extension of Lie 2-algebras (as defined in \cite{R2})
\begin{equation}\label{Lie 2-alg ext} \xymatrix{ \L_{g_{ijk}}(0) \ar[r] & \L_{g_{ijk}} \ar[r] & C^{\infty}(TM),}\end{equation} where $\L_{g_{ijk}}(0)$ is the set of lifts of the zero vector field in the Cech picture. As we explain in \S\ref{conn str and curving}, a connective structure on $\CC$ determines a linear splitting  of the extension (\ref{Lie 2-alg ext}), sending each vector field $\xi$ on $M$ to its ``horizontal lift" $\tx^h$.  Given a pair of vector fields $\xi,\eta\in C^{\infty}(TM)$, without additional structure there is no natural way to compare $[\tx^h,\te^h]$ to $[\tx,\te]^h$.  We then observe that such data is essentially equivalent to a curving $K$ for $\A$.  The curvature 3-form of $K$ (or more precisely its de Rham cohomology class) acts as an obstruction to splitting the sequence (\ref{Lie 2-alg ext}) on the level of Lie 2-algebras.

One can also consider the collection of gerbes with connective structure over $M$ as a 2-category in its own right.  In particular, given a gerbe $\CC$ with connective structure $\A$ we next study the infinitesimal symmetries of $\CC$ which take $\A$ into account.  These form a sheaf of categories $\ul{\L}_{(\CC,\A)}$, which may also be considered an analogue of $TP/\T$ in the circle bundle case.  If we denote by $\check{\textrm{B}}(\ul{i\R}_M)$ the (stack) of principal $i\R$ bundles over $M$ with connection, we have a sequence
\begin{equation} \xymatrix{ \check{\textrm{B}}(\ul{i\R}_M) \ar[r] & \ul{\L}_{(\CC,\A)} \ar[r] & \ul{TM}. }\end{equation} We denote the category of global sections by $\L_{(\CC,\A)}$ and refer to its objects as \emph{connective symmetries}.     There is a forgetful functor 
\begin{equation} \L_{(\CC,\A)}\to \L_{\CC},\end{equation} and we prove that the set of connective symmetries extending a given element of $\L_{\CC}$ is a torsor for global 1-forms on $M$.  We extend most of the results and constructions given earlier to the connective case.  In particular there is a Cech model $\L_{(g_{ijk},A_{ij})}$ which is a Lie 2-algebra.
     
In the last section we give a construction of the Courant algebroid $E_{(\CC,\A)}$ in terms of the infinitesimal connective symmetries of $(\CC,\A)$, and compare this construction to that given in \cite{H}. The construction (at the level of global sections) is roughly as follows.  As mentioned above, there is a forgetful functor $\pi:\L_{(\CC,\A)}\to \L_{\CC}$ such that, for each $\tx\in \L_{\CC}$ the set of objects $\pi^{-1}(\tx)\subset \L_{(\CC,\A)}$ exending $\tx$ is a torsor for 1-forms on $M$.  On the other hand, for each vector field $\xi$ on $M$ the connective structure $\A$ determines a horizontal lift 
\begin{equation} \tx^h\in \L_{\CC},\end{equation}  and thus for each vector field $\xi$ we may define 
\begin{equation} E_{\tx}=\pi^{-1}(\tx^h) \subset \L_{(\CC,\A)}(\xi). \end{equation} By the above, the set  $E_{\tx}$ form a torsor for 1-forms on $M$. Taking the disjoint union of the sets $E_{\xi}$ as $\xi$ ranges over all vector fields on $M$,  we obtain a $C^{\infty}_M$-module\footnote{we could easily generalize our methods to construct a sheaf of modules} $E_{(\CC,\A)}$ fitting into an exact sequence of $C^{\infty}_M$-modules
\begin{equation} \xymatrix{ 0 \ar[r] & C^{\infty}(T^*M) \ar[r] & E_{(\CC,\A)} \ar[r] & C^{\infty}(TM) \ar[r] & 0.}\end{equation}

Given a pair of vector fields $\xi,\eta$, and a pair of connective lifts $\cx,\ce\in E$ extending $\tx^h,\te^h$, their bracket $[\cx,\ce]$ in the category $\L_{(\CC,\A)}$ is \emph{not} itself a connective extension of $[\xi,\eta]^h$, i.e. is not an element of $E_{[\xi,\eta]}$.  On the other hand, $[\cx,\ce]$ is \emph{naturally isomorphic} to an element $[\cx,\ce]_{E}\in E_{[\xi,\eta]}$, where the bracket $[\cdot,\cdot]_{E}$ corresponds to the Courant bracket.  We also construct a non-degenerate symmetric bilinear pairing 
\begin{equation} \langle\cdot,\cdot\rangle: E\times E \to C^{\infty}_M \end{equation} which is suitably compatible with the bracket $[\cdot,\cdot]_E$ and the projection to $C^{\infty}(TM)$.  

The construction of the Courant algebroid in \cite{H} is given in terms of Cech data $\{g_{ijk},A_{ij}\}$ for $(\CC,\A)$.  We explain how this Courant algebroid, which we denote $E_{(g_{ijk},A_{ij})}$, is related to the $C^{\infty}_M$-module $E_{(\CC,\A)}$ described above. One consequence of this analysis is that we obtain an isomorphism of  Lie 2-algebras \begin{equation}\label{Lie 2-algebra iso} L_{E_{(g_{ijk},A_{ij})}}  \tocong\L_{(g_{ijk},A_{ij})},  \end{equation} where $L_{E_{(g_{ijk},A_{ij)}}}$ is the 2-term $L_{\infty}$-algebra constructed from $E_{(g_{ijk},A_{ij})}$ using the results of \cite{RW} and \cite{R1}.
Since the proof that $\L_{(g_{ijk},A_{ij})}$ is an $L_{\infty}$-algebra follows almost immediately from its construction, the isomorphism (\ref{Lie 2-algebra iso}) gives some intuition as to the  origin of the $L_{\infty}$-structure on $L_{E_{(\CC,\A)}}$.

The outline of the paper is as follows.  In \S\ref{VF} we review the geometry of vector fields on principal $\T$-bundles; this material serves as a useful point of reference for the remainder of the paper.  In \S\ref{torsors} we explain some constructions involving torsors for sheaves of abelian groups.  In \S\ref{gerbes} we review basic definitions and results involving gerbes, in particular emphasizing maps between gerbes with different band, as well as the the 2-categorical structure possessed by the collection of gerbes over a fixed manifold $M$.  In \S\ref{section gerbe symmetries} we introduce our definition of the category of infinitesimal symmetries of a gerbe $\CC$ over $M$, motivated by the circle bundle case and some ideas from algebraic geometry.   In \S\ref{Cech}  we explain how infinitesimal symmetries appear in the Cech picture, taking pains to compare this approach to our initial definition.  In \S\ref{local flows} we explain the local relationship between 1-parameter symmetries of $\CC$ and infinitesimal symmetries.   In \S\ref{operations} we examine the algebraic structure of the category of infinitesimal gerbe symetries, in particular introducing operations of addition, scalar multiplication and the Lie bracket.  We use these operations to give $\L_{g_{ijk}}$ the structure of a 2-term $L_{\infty}$-algebra.  In \S\ref{conn str and curving} we introduce connective structures and curvings, and in \S\ref{conn lifts} introduce infinitesimal connective symmetries. \S\ref{local flows conn} generalizes the discussion in \S\ref{local flows} to the connective case.  Finally, in \S\ref{courant} we discuss the relationship between connective symmetries and Courant algebroids. There are two appendices which contain some of the more technical arguments.  Appendix A goes into more detail about the relationship between 1-parameter families of gerbe symmetries and infinitesimal symmetries.  In appendix B we recall the precise definitions related to the 2-category of gerbes.  We use this material to prove some results from the main text.

\begin{rem} There exist in the literature various models of gerbes.  In this paper we adopt the perspective of Brylinski \cite{Br1}, where a Dixmier-Douady gerbe is regarded as a sheaf of categories equipped with some extra structure.  We have found it to be easier to understand certain conceptual questions about the symmetries of gerbes using this model.  On the other hand, because the categorical model keeps track of so much information (isomorphic objects are never identified), constructions in this picture can become quite intricate.  As a result, we have found it useful to recast many of the main definitions and constructions in a more concrete Cech-type model for gerbes, in which a gerbe is represented by a $\T$-valued Cech 2-cocycle on $M$.  Most constructions can be done equally well in either approach, although our construction of the Lie 2-algebra structures on $\L_{\CC}$ and $\L_{(\CC,\A)}$ are done only in the Cech picture.  One could similarly phrase our ideas in terms of other models of gerbes, for example bundle gerbes or presentations of gerbes as groupoids.

\end{rem}

\section{\label{VF} Infinitesimal Symmetries of Circle Bundles}
We begin by reviewing the basic geometry of vector fields on circle bundles and their relationship to 1-parameter groups of symmetries.  In particular, we formulate the basic definitions and structures in  sheaf-theoretic terms.   The discussion in this section lays the groundwork for our discussion of infinitesimal symmetries of Dixmier-Douady gerbes, and we will continually refer back to the circle bundle case throughout the rest of the paper as a point of reference.

Let $\pi:E\to M$ be a principal $\T$-bundle over a smooth manifold $M$.  Thus $E$ is a smooth manifold with a free action of the circle group $\T$ whose orbits are the fibers of the projection map $\pi$.  For every $x\in E$, the kernel of $\pi_*:T_xE\to T_{\pi(x)}M$ is the one dimensional subspace consisting of \emph{vertical vectors}.  The $\T$-action induces a canonical isomorphism of this subspace with the Lie algebra $\textrm{Lie}(\T)$, which we identify with the $\R$-vector space $i\R$ consisting of purely imaginary complex numbers.

A smooth 1-parameter group of diffeomorphisms of $M$ is a collection of diffeomorphisms $\{\varphi_t:M\to M\}$ such that $\varphi_t'\circ\varphi_t=\varphi_{t+t'}$ for each $t,t'\in \R$, and such that the corresponding map 
\begin{equation} \Phi:M\times \R\to M\end{equation} is smooth.
\begin{definition} \label{def bundle 1-par lift}  Let $\Phi$ be a be a smooth 1-parameter group of diffeomorphisms of $M$.  A 1-parameter group of symmetries of $E$ lifting $\Phi$ is a 1-parameter group of diffeomorphism 
\begin{equation} \hat{\Phi}:E\times \R\to E\end{equation} such that for each $t\in\R$, the map $\Phi(\cdot,t)=\varphi_t$ commutes with the $\T$-action and satisfies $\varphi_t\pi=\pi\hat{\varphi_t}$.
\end{definition}
  Given such a pair $(\Phi,\hat{\Phi})$, we obtain vector fields $\xi\in C^{\infty}(TM)$ and $\hat{\xi}\in C^{\infty}(TE)$ by differentiating at $t=0$.  Definition (\ref{def bundle 1-par lift}) implies that $\tx$ is $\T$-invariant and projects to $\xi$; we call such a vector field a \emph{lift} of $\xi$ to $E$.  

Let us describe a lift $\tx$ in terms of the sheaf of sections $\ul{E}$ of $E\to M$.  Let $\sigma$ be a local section of $E$. Then $\tx-\sigma_*\xi$ is a vertical vector at each point in the image of $\sigma$, and therefore the section $\sigma$ determines a local $i\R$-valued function $f_{\sigma}$ according to the formula
\begin{equation}\label{lift sheaf eq} \tx_{\sigma(x)}=\sigma_*\xi_x+f_{\sigma}(x).\end{equation}  Because $\tx$ is $\T$-invariant, it is completely determined locally by $f_{\sigma}$.  Moreover, it is easy to see that the assignment $\sigma\mapsto f_{\sigma}$ is compatible with restrictions to smaller open sets, so that we obtain a sheaf homomorphism
\begin{align} F_{\tx}: \ul{E}  \to & \ul{i\R}_M \\ \sigma  \mapsto & f_{\sigma},\notag\end{align} where $\ul{i\R}_M$ denotes the sheaf of smooth $i\R$-valued functions on $M$.  

The sheaf of sections $\ul{E}$ is naturally a torsor for the sheaf $\ul{\T}_M$ of smooth $\T$-valued functions on $M$, as discussed in detail in the next section.  Given a local section $\sigma$ of $E$ and a local $\T$-valued function $g$, equation (\ref{lift sheaf eq}) implies that
\begin{equation}\label{f dependence} F_{\tx}(\sigma\cdot g)=F_{\tx}(\sigma)-\iota_{\xi}d\log(g),\end{equation} where $d\log(g)$ denotes $g^{-1}dg$.  Conversely, given a sheaf homomorphism $\ul{E}\to \ul{i\R}_M$ satisfying (\ref{f dependence}) we obtain a unique lift of $\xi$ to $E$.

To set the stage for our discussion in \S\ref{local flows} and appendix A, let us derive (\ref{f dependence}) by considering the relationship between the vector field $\tx$ and the 1-parameter family $\hat\Phi$.  
\begin{notation} We will denote $M\times \R$ by $M^1$.  We will denote the projection
\begin{equation} \pi_1:M^1\to M\end{equation} by $p_1$ and 
\begin{equation} \Phi:M^1\to M\end{equation} by $p_0$.  Furthermore, for each open set $U$ of $M$, we define the open set  
\begin{equation} \nu(U)=\{(x,t)\in U\times\R:\varphi_t(x)\in U\}\subset M^1.\end{equation}  Note that there is a natural inclusion $U\hookrightarrow \nu^1(U)$ given by $x\mapsto (x,0)$. 
\end{notation}  For each section $\sigma$ of $E$ over $U$ define a function $g_{\sigma}:\nu(U)\to \T$ by 
\begin{equation}\label{def g lift} \hat{\varphi}_t(\sigma(x))=\sigma(\varphi_t(x))g_{\sigma}(x,t).\end{equation}  

\noindent The function $f_{\sigma}:U\to i\R$ is then related to $g_{\sigma}$ by   
\begin{equation}\label{diff g} f_{\sigma}(x)=\iota_{\frac{d}{dt}}d\log(g_{\sigma}(x,t))|_U.\end{equation}  Given another local section $\sigma'$, there is a unique function $h:U\to \T$ such that $\sigma'=\sigma\cdot h$.  Equation (\ref{def g lift}) then implies that 
\begin{equation} \label{g cocycle 1} g_{\sigma'}(x,t)=g_{\sigma}(x,t)h(x)h^{-1}(\varphi_t(x)).\end{equation}
Combining this with (\ref{diff g}) we obtain (\ref{f dependence}). 

Next, suppose $E$ has a connection $\Theta$.  We view $\Theta$ as a map $A_{\Theta}$ from the sheaf of sections of $E$ to the sheaf $\underline{i\Omega}_M^1$ of imaginary 1-forms on $M$ by defining \begin{equation}\label{A sheaf} A_{\Theta}(\sigma)= \sigma^*\Theta.\end{equation}  The behavior of $A_{\Theta}$ under gauge transformations is
\begin{equation}\label{gauge}A_{\Theta}(\sigma\cdot g)=A_{\Theta}(\sigma)+g^{-1}dg.\end{equation}
Given a vector field $\xi$, there is a unique \emph{horizontal lift} $\tx^h$ of $\xi$ to $P$, characterized by the equation \begin{equation}\label{horlift}f_{\tx^h}(\sigma)=-\iota_{\xi}A_{\Theta}(\sigma).\end{equation}  On the other hand, if we define $\tx^h$ by equation (\ref{horlift}), it is easy to see from equation (\ref{gauge})  that  $f_{\tx^h}$ satisfies equation (\ref{f dependence}).  


Suppose that $\hat{\xi}$ and $\hat{\eta}$ are a pair of vector fields on $E$ lifting vector fields $\xi$ and $\eta$ on $M$.  Then the bracket $[\hat{\xi},\hat{\eta}]$ is a lift of $[\xi,\eta]$ satisfying
\begin{equation}f_{[\tx,\te]}=\xi(f_{\te})-\eta(f_{\tx}).\end{equation}  In particular, if $\tx=\tx^h$ and $\te=\te^h$ are the horizontal lifts, then
\begin{align} \label{cocycle1} & f_{[\xi,\eta]^h}-f_{[\xi^h,\eta^h]}  
\\ &= \iota_{[\xi,\eta]}A_{\Theta}-\iota_{\xi}d\iota_{\eta}A_{\Theta}+\iota_{\eta}d\iota_{\xi}A_{\Theta} \notag
\\  \label{cocycle2}&= -\iota_{\eta}\iota_{\xi}K(\Theta),\notag
\end{align} where $K(\Theta)\in i\Omega^2(M)$ is  the curvature form of $\Theta$ and satisfies $K(\Theta)=dA_{\Theta}(\sigma)$ for any local section $\sigma$.   

\section{\label{torsors}Torsors}

Let $\textrm{Bund}_M(\T)$ denote the category of principal $\T$-bundles over a manifold $M$.  The starting point for this section is the observation that $\textrm{Bund}_M(\T)$ possesses many structures analogous to those of an abelian group, i.e. it is a Picard category.  For example, given $E,E'\in \textrm{Bund}_M(\T)$ we can form their tensor product $E\otimes E'$; this operation is commutative and associative up to natural isomorphism, and the trivial $\T$-bundle $M\times \T$ acts as a multiplicative unit.  Furthermore, for any bundle $E$ the dual bundle $E^{\vee}$ plays the role of an inverse for $E$.  As explained below, there are also operations corresponding to taking Lie derivatives and the exterior derivative.   These examples can  be conveniently described as special cases of the \emph{associated torsor} construction.  After introducing this construction and examining its properties, we use it to rephrase some of the discussion in the last section in a way which will easily generalize to the gerbe setting.

\begin{definition}  Let $A$ be an abelian group.  
\begin{enumerate} \item An \emph{$A$-torsor} is a set $S$ with a simply transitive action of $A$ on $S$:\begin{equation} (s,a)\mapsto s+a, \mbox{ for $s\in S$ and $a\in A$.}\end{equation}
\item Given $A$-torsors $S$ and $T$, a \emph{homomorphism of $A$-torsors} is a map $f:S\to T$ such that $f(s+a)=f(s)+a$ for all $s\in S$ and $a\in A$.
\end{enumerate}
\end{definition}
\begin{notation} We will denote the category of all $A$-torsors by $\tor_A$.  We write the action of $A$ on a torsor additively unless multiplication in the group $A$ is conventionally written multiplicatively (e.g. $\T$).  Note that because $A$ is abelian we do not distinguish between left and right $A$-torsors.
\end{notation}

We now introduce the associated torsor construction.

\begin{definition} \label{def ass tor} Let $\varphi:A\to B$ be a homomorphism of abelian groups.  Given $T\in \tor_A$ let $\sim$ denote the equivalence relation on $T\times B$ given by
\begin{equation}(t +a,b)\sim (t,b+\varphi(a)) \mbox{ for all $a\in A$}.\end{equation} The \emph{associated $B$-torsor} $\varphi[T]$ is the quotient
\begin{equation} T\times B/\sim,  
 \end{equation}  with elements of $B$ acting on the second factor. 
 \end{definition}
 \begin{ex} \label{ass tor ex} Let $A=\T\times\T$, $B=\T$; because $\T$ is abelian group multiplication defines a homomorphism $\mu:\T\times\T\to \T$.  Given $S,T\in\tor_{\T}$, $S\times T$ is naturally a $\T\times \T$-torsor, and we will denote the associated torsor $\mu[S\times T]$ by $S\otimes T$.  In the case where $A$ is written addititively (e.g. $i\R$), if $\varphi:A\times A\to A$ is the addition homomorphism we will use the notation $\varphi[S\times T]:=S\bp T$. 
 
Again, because $A$ is abelian, the inverse map $a\mapsto -a$ defines a group homomorphism, and in this case the associated torsor construction takes a torsor $T$ to its dual $T^{\vee}$.  Given $S,T\in \tor_A$, we will sometimes use the notation $T\boxminus S$ to denote the $A$-torsor $\Hom(S,T)$; this can be identified with the torsor associated to $S\times T$ via the homomorphism from $A\times A\to A$ taking $(a,a')\mapsto a'-a$.

In the case that $A=i\R$ (or more generally any vector space), given $T\in \tor_{i\R}$ and $\lambda\in \R$ regarded as a homomorphism $i\R\to i\R$, we will denote the associated torsor $\lambda[T\times i\R]$ by $\lambda\odot T$.
\end{ex}

 The associated torsor construction satisfies a universal property, the proof of which is omitted.  
 \begin{lemma}\label{ass tor universal} Let $\varphi:A\to B$ be a homomorphism of abelian groups.  For each $T\in\tor_A$ there is a canonical map $\varphi[\cdot]:T \to \varphi[T]$ such that 
 \begin{equation} \varphi[t+a]=\varphi[t]+\varphi(a)\end{equation} for all $t\in T$ and $a\in A$. Furthermore, if $R$ is any $B$-torsor and $\psi:T\to R$ is any map satisfying $\psi(t+a)=\psi(t)+\varphi(a)$ for all $a\in A$, $t\in T$, then there exists a unique isomorphism of $B$-torsors $\tilde{\psi}:\varphi[T]\to R$ such that $\tilde{\psi}(\varphi[t])=\psi(t)$ for all $t\in T$.
 \end{lemma}
 \begin{terminology}  Given $T\in \tor_A$, $R\in \tor_B$ and $\varphi:A\to B$, we say a map $\psi:T\to R$ \emph{intertwines} $\varphi$ if $\psi(t+a)=\psi(t)+\varphi(a)$ for all $t\in T$ and $a\in A$.  Lemma (\ref{ass tor universal}) states that there is a canonical map from $T$ to $\varphi[T]$ intertwining $\varphi$, and that given any $R\in \tor_B$, a map from $T$ to $R$ intertwining $\varphi$ is the same thing as a homomorphism of $B$-torsors $\varphi[T]\to R$.
 \end{terminology}
 The universal property (\ref{ass tor universal}) can be used to give a formal proof of the following.
 \begin{corollary}\label{ass tor fun}
 \begin{enumerate}\item Let $\varphi:A\to B$ be a homomorphism of abelian groups.  Then there is a functor
\begin{equation} \varphi[\cdot]:\tor_A\to \tor_B\end{equation} sending $T\to \varphi[T]$, and a homomorphism $\psi:T\to T'$ of $A$-torsors to the unique homomorphism $\varphi[\psi]:\varphi[T]\to\varphi[T']$ satisfying $\varphi[\psi](\varphi[t])=\varphi[\psi(t)]$ for all $t\in T.$
\item There is a canonical isomorphism 
\begin{equation} \varphi[A]\cong B,\end{equation} where we consider $A$ as an $A$-torsor with the obvious action on itself. 
\item Given another abelian group $C$ and a homomorphism $\psi:B\to C$, there is a natural isomorphism 
\begin{equation} (\alpha_{\psi,\varphi})_T:\psi[\varphi[T]]\cong (\psi\varphi)[T]. \end{equation}
\end{enumerate}
\end{corollary}

We now generalize to the context of  sheaves over a manifold $M$.  For example, given a principal $A$-bundle $E$ for $A$ a Lie group, the sheaf of sections of $E$ carries an action of the sheaf of smooth $A$-valued functions.  


\begin{definition}  Let $M$ be a manifold, and let $A$ be a sheaf of abelian groups over $M$.  Then a \emph{sheaf of $A$-torsors} over $M$ is a sheaf $T$ of sets over $M$ together with a homomorphism of sheaves $\alpha:T\times A\to T$ such that
\begin{enumerate} \item For every open subset $U\subset M$ such that $T(U)$ is non-empty, $(T(U),\alpha(U))$ is an $A(U)$-torsor,
\item Every $x\in M$ is contained in a neighborhood $U\subset M$ such that $T(U)$ is non-empty.
\end{enumerate}
\end{definition} 
\begin{rem} \label{cat tor sheaf}Given a sheaf of abelian groups $A$, one defines a morphism of $A$-torsors in the obvious way.  We will denote the resulting category by $\tor_A$.
\end{rem}

\begin{ex} As mentioned above, given any principal $A$-bundle $E$, its sheaf of sections $\ul{E}$ is naturally an $\ul{A}_M$-torsor, where $\ul{A}_M$ denotes the sheaf of smooth $A$-valued functions on $M$.  Moreover, the assignment $E\mapsto \ul{E}$ is functorial and induces an equivalence of categories.
\end{ex}
\begin{ex}  We will also be interested in sheaves of groups which are not of the form $\ul{A}_M$ for some fixed Lie group $A$.  For example, given a principal circle bundle $E$ over $M$ the set of connections for $E$ is torsor for the group of 1-forms on $M$.  More generally, we have a sheaf of  connections which is a torsor for the sheaf of (imaginary) 1-forms on $M$.
\end{ex}  

\begin{definition} Given a homomorphism $\varphi:A\to B$ of sheaves of groups over $M$ and torsors $S\in \tor_A$, $T\in \tor_B$, a sheaf homomorphism $\psi:S\to T$ \emph{intertwines} $\varphi$ if for each open set $U\subset M$ such that $S(U)$ is non-empty, $T(U)$ is also non-empty and the map of torsors
\begin{equation} \psi_U:S(U)\to T(U) \end{equation} intertwines 
\begin{equation} \varphi_U:A(U)\to B(U).\end{equation} \end{definition}  We may extend definition (\ref{def ass tor}) of the associated torsor construction to the sheaf setting.  For brevity we give a definition in terms of a universal property, and then sketch a construction.
\begin{definition}\label{ass tor sheaf} Let $\varphi:A\to B$ be a homomorphism of sheaves of groups over $M$, and let $S$ be an $A$-torsor.  Than a \emph{torsor associated to $S$ via $\varphi$} is a $B$-torsor $\varphi[S]$ together with a homomorphism of sheaves
\begin{equation} \varphi[\cdot]:S\to \varphi[S]\end{equation} intertwining $\varphi$.
\end{definition}

One easily shows from the definition that the associated torsor $\varphi[S]$ is unique up to unique isomorphism, so that for will speak of \emph{the} associated torsor.  For a specific construction, we may first take 
\begin{equation} \varphi[S](U)=\varphi_U[S(U)].\end{equation}  Given an inclusion of open sets $i:V\hookrightarrow U$, the restriction map is then characterized by the equation 
\begin{equation} i^*(\varphi_U[s])=\varphi_V[i^*s]\end{equation} for each $s\in S(U)$.  In general this procedure only defines a presheaf, and we must take its sheafification to finish the construction of $\varphi[S]$; for example, if $A=\ul{\T}_M$ and $B=\ul{i\R}_M$ then $S$ may have no global sections, whereas any $\ul{i\R}_M$-torsor does have a global section.  It is then easily checked that the appropriate generalizes of the second part of lemma (\ref{ass tor universal}) as well as corollary (\ref{ass tor fun}) hold in the sheaf setting.

\begin{rem}  Let us see how the associated torsor construction looks in terms of local trivializations and transition functions.  Given a sheaf $A$ of abelian groups over $M$, let $E\in\tor_A$.  Given an open cover $\{U_i\}$ of $M$ and local sections $\sigma_i\in E(U_i)$, we obtain ``transition functions" $a_{ij}\in A(U_{ij}=U_i\cap U_j)$ given by 
\begin{equation} \sigma_j|_{U_{ij}}=\sigma_i|_{U_{ij}}+a_{ij}.\end{equation} It is easily verified that on triple overlaps $U_i\cap U_j\cap U_k:=U_{ijk}$ we have 
\begin{equation}\label{transition function cocycle} a_{jk}-a_{ik}+a_{ij}=0,\end{equation} where each term in equation (\ref{transition function cocycle}) is implicitly restricted to $U_{ijk}$.  Let $B$ be another sheaf of abelian groups over $M$ and $\varphi:A\to B$ a sheaf homomorphism.  By the universal property of the associated torsor construction, we obtain local sections $\varphi[\sigma_i]$ of $\varphi[E]$ over $U_i$.  Since the map taking $\sigma_i\mapsto \varphi[\sigma_i]$ intertwines $\varphi$, it follows that the transition functions for $\varphi[E]$ with respect to the local sections $\varphi[\sigma_i]$ are given by $\varphi(a_{ij})$.  
\end{rem}

 In the next example we introduce an operation  on $\T$-bundles analogous to taking the Lie derivative with respect to a vector field on $M$. 
  \begin{ex} \label{ex tor T sheaf}  Given a manifold $M$, note that a vector field $\xi\in C^{\infty}(TM)$ determines a sheaf homomorphism $\ul{\T}_M\to\ul{i\R}_M$ given by
\begin{equation} \iota_{\xi}\textrm{dlog}: g\mapsto \iota_{\xi}g^{-1}dg. \end{equation}  
  Let $E$ is a principal $\T$-bundle over $M$ and $\ul{E}$ its sheaf of sections.
By the discussion in section \S\ref{VF}, a lift of $\xi$ to $E$ is equivalent to a sheaf homomorphism $f_{\tx}:\ul{E}\to\ul{i\R}$ intertwining $-\iota_{\xi}d\log$.  We may therefore identify the sheaf of such lifts with\footnote{Given sheaves $\A$ and $B$, we will denote by $\ul{\Hom}(A,B)$ the \emph{sheaf} of homomorphisms from $A$ to $B$}   \begin{equation} \ul{\Hom}(-\iota_{\xi}d\log[\ul{E}],\ul{i\R}_M)\cong \ul{\Hom}(\ul{i\R}_M,\iota_{\xi}d\log[\ul{E}])\cong \iota_{\xi}d\log[\ul{E}].\end{equation}  To see this isomorphism concretely on the level of global sections, suppose that we are given a global section $S$ of $\iota_{\xi}d\log[\ul{E}]$.  Given a local section $\sigma\in\ul{E}(U)$, there is a unique function $f_S(\sigma):U\to i\R$ such that 
\begin{equation} \label{local form} S|_U=\iota_{\xi}d\log[\sigma]+f_S(\sigma).\end{equation}  A short calculation shows that given $g:U\to \T$ we must have $f_S(\sigma\cdot g)=f_S(\sigma)-\iota_{\xi}d\log(g)$, so that $\sigma\mapsto f_S(\sigma)$ is a homomorphism from $\ul{E}$ to $\ul{i\R}_M$ intertwining $-\iota_{\xi}d\log$.  Conversely, given such a homomorphism $f_{\tx}$, we can define a global section $\iota_{\xi}d\log[\ul{E}]$ which is given locally by the formula (\ref{local form}).


\end{ex}
\begin{ex}\label{connection torsor} Let $\Theta$ be a connection on the principal $\T$-bundle $E$.  Then equation (\ref{gauge}) says that the sheaf homomorphism $A_{\Theta}:\ul{E}\to\ul{\Omega}^1_M(i\R)$ defined by equation (\ref{A sheaf}) is consistent with the homomorphism 
\begin{align} \textrm{dlog}: \ul{\T}_M  \to & \ul{\Omega}^1_M \\   g  \mapsto & g^{-1}dg.\notag \end{align}

Proceeding as in the last example, we see that the $i\Omega^1_M$-torsor of connections on $E$ can be identified with the global sections of $-d\log[E]$.  Explicitly, given a connection viewed as a sheaf homomorphism $A_{\Theta}:\ul{E}\to \ul{i\Omega}^1_M$, the corresponding section $\Theta$ of $-d\log[E]$ is given locally by 
\begin{equation}\Theta|_U=-d\log[\sigma]+A_{\Theta}(\sigma) \end{equation} for $\sigma$ a local section of $E$.  We remark also that, given such a section $\Theta$, we obtain a section \begin{equation}\label{horiz section} \iota_{-\xi}[\Theta]\end{equation} of \begin{equation}\iota_{-\xi}[-d\log[E]]\cong \iota_{\xi}d\log[E].\end{equation}  As explained in the last example, global sections of $\iota_{\xi}d\log[E]$ can be identified with lifts of $\xi$ to $E$; in this case, the section (\ref{horiz section}) is the horizontal lift determined by $\Theta$ as described in \S\ref{VF}.

\end{ex}
\section{\label{gerbes} Gerbes} In this section we briefly recall some of the definitions  and constructions related to gerbes which we will need in the remainder of the paper.  More detailed definitions can be found in appendix B.  The standard reference for most of this material is \cite{Br1}, and much of our notation and terminology follows this source.   In addition to this standard material, we introduce some terminology which generalizes the material in \S\ref{torsors} to the context of gerbes.  We will also find it convenient to describe the family of all gerbes over a given manifold using the language of (strict) two-categories.  Our reference for 2-categories is \cite{Bo}.

Following the approach in \cite{Br1}, we use the model of a gerbe $\CC$ over a manifold $M$ as a sheaf of groupoids equipped with extra structure.  For the precise definition of a presheaf of groupoids, see definition (\ref{def groupoid presheaf}) in  appendix B.  Briefly, a presheaf of groupoids $\CC$ consists of the following data: for each open set $U$ of $M$ we have  a groupoid $\CC(U)$, and for each inclusion of open sets $i:V\hookrightarrow U$ we a restriction functor $i^*:\CC(U)\to \CC(V)$.  Given another inclusion of open sets $j:W\hookrightarrow V$, there is a specified natural isomorphism $\alpha_{i,j}:j^*i^*\tocong (ij)^*$.  These natural isomorphisms  are themselves required to satisfy a coherence condition with respect to chains of inclusions of the form 
\begin{equation} T\hookrightarrow W\hookrightarrow V \hookrightarrow U.\end{equation}  

It is a consequence of definition (\ref{def groupoid presheaf}), that for every open set $U\subset M$ and every pair of objects $P,Q\in \CC(U)$ there is a presheaf $\ul{\Hom}(P,Q)$ over $U$.   $\CC$ is called a \emph{prestack} if $\ul{\Hom}(P,Q)$ is actually a sheaf for each open set $U$ and each pair $P,Q$.  A prestack $\CC$ is a \emph{stack} (or sheaf of groupoids) if in addition it satisfies a gluing property, as outlined in the discussion following (\ref{def groupoid presheaf}) in appendix B.

A gerbe is a stack equipped with additional structure analogous to that possessed by a principal bundle.

\begin{definition}\label{def gerbe}  Let $A$ be a sheaf of abelian groups over $M$.  A \emph{gerbe with band $A$} over $M$ is a stack $\CC$ equipped with a family of isomorphisms 
\begin{equation} \alpha_P:\ul{Aut}(P)\tocong A|_U\end{equation} for each open set $U\subset M$ and each object $P\in \CC(U)$, such that for each object $Q\in \CC(U)$ and each isomorphism $\psi:P\to Q$, $\alpha_P$ and $\alpha_Q$ are compatible with the induced isomorphism $\ul{\Aut}(P)\to \ul{\Aut}(Q)$.  In addition, we require that 
\begin{enumerate} \item Every point $x\in M$ is contained in a neighborhood $U\subset M$ such that $\CC(U)$ is non-empty.
\item Any two $P,Q\in \CC(U)$ are \emph{locally isomorphic}; that is, for each $x\in U$ there exists a neighborhood $V$ of $x$ in $U$ such that the restrictions of $P$ and $Q$ to $V$ are isomorphic.
\end{enumerate}
\end{definition}
\begin{rem} \label{morphism torsor} Note that for any two local sections $P,Q\in \CC(U)$, the sheaf $\ul{\Hom}(P,Q)$ has the structure of a $\ul{\T}_U$-torsor: given $\psi\in \ul{\Hom}(P,Q)(V)$ for $V\subset U$ and $g:V\to \T$, we define 
\begin{equation} \psi\cdot g=\alpha_P(g)\circ \psi.\end{equation}  Conversely, as explained in \cite{Br1}, given any object $P\in \CC(U)$ and any $\ul{\T}_U$-torsor $E$ there exists an object $P\otimes E$ such that there is an isomorphism of $\ul{\T}_U$-torsors $\ul{\Hom}(P,P\otimes E)$.  More precisely, there is a functor \begin{equation}
\CC(U)^{op}\times \CC(U)\to \CC(U)^{op}\times \textbf{Tor}_{\ul{\T}_U}\end{equation} taking 
\begin{equation} (P,Q) \mapsto (P, \ul{\Hom}(P,Q)) \end{equation} which is an equivalence of categories.  In this sense, the category $\CC(U)$ is a torsor for the category of $\ul{\T}_{U}$-torsors.
\end{rem}

\begin{ex}\label{trivial gerbe}  Let $A$ be a sheaf of abelian groups over $M$.  Then the \emph{trivial gerbe} $\textrm{B}(A)$ with band $A$  associates to each open set $U\subset M$ the category of $A|_U$-torsors.    
\end{ex}

Next, let us discuss morphisms and 2-morphisms between presheaves of groupoids;  for precise definitions see definitions (\ref{stack morphism}) and  (\ref{2-morphism}) in the appendix.    Given presheaves of groupoids $\CC$ and $\CC'$, a morphism of $\Phi:\CC\to \CC'$  consists first of all of a collection of functors 
\begin{equation} \Phi_U:\CC(U)\to \CC'(U).\end{equation}  In addition, for each inclusion $i:V\hookrightarrow U$  the data of $\Phi$ includes coherent natural transformations $\Phi_i:i^*_{\CC'}\Phi_U\Rightarrow \Phi_Vi^*_{\CC}$.  In the terminology of \cite{Bo}, this is called a \emph{pseudo-natural transformation}.    Given a pair of morphisms $\Phi,\Psi:\CC\to \CC'$, a 2-morphism $\tau:\Phi\Rightarrow \Psi$ consists of natural transformations $\tau_U:\Phi_U\Rightarrow \Psi_U$ which are suitably compatible with the structure of $\Phi$ and $\Psi$.  In \cite{Bo}, this is called a \emph{modification}.

The following is a consequence of proposition (7.5.4) in \cite{Bo}.
\begin{proposition} \label{presheaf 2-category}There is a 2-category whose objects are presheaves of groupoids over $M$, whose 1-morphisms are pseudo-natural transformations, and whose 2-morphisms are modifications.
\end{proposition}

In particular,  we have a well-defined composition of 1-morphisms, as well as vertical and horizontal composition of 2-morphisms.

In the case that $\CC$ and $\CC'$ are gerbes, we will be interested in morphisms which interact well with the additional structure specified in definition (\ref{def gerbe}).

\begin{definition}  Let $\varphi:A\to B$ be a homomorphism of sheaves of abelian groups over $M$.  Given a gerbe $\CC$ over $M$ with band $A$ and a gerbe  $\CC'$ with band $B$, a 1-morphism $\Phi:\CC\to \CC'$ between their underlying sheaves of groupoids is said to \emph{intertwine $\varphi$} if for each object $P\in \CC(U)$, the diagram
\begin{equation} \label{gerbe intertwine}\xymatrix{ \underline{\Aut}(P) \ar[d]_{\alpha_P} \ar[r]^-{\Phi_U} & \underline{Aut}(\Phi_U(P)) \ar[d]^{\alpha'_{\Phi_U(P)}} \\ A|_U \ar[r]_{\varphi|_{U}} & B|_U}\end{equation}
commutes.  \end{definition}  
\begin{rem} In terms of the notation (\ref{morphism torsor}), for each $P,Q\in \CC(U)$, each local section $\psi\in \ul{\Hom}(P,Q)(V)$ and each $g:V\to \T$, the diagram (\ref{gerbe intertwine}) implies that 
\begin{equation} \Phi(\psi\cdot g)=\Phi(\psi)\cdot\varphi(g).\end{equation}
\end{rem}
\begin{rem} By proposition (\ref{presheaf 2-category}), given a gerbe $\CC$ with band $A$ and $\CC'$ with band $B$, there is a category of 1-morphisms between presheaves underlying $\CC$ and $\CC'$.  Given $\varphi:A\to B$, we then have a subcategory of 1-morphisms from $\CC$ to $\CC'$ that intertwine $\varphi$, which we will denote 
\begin{equation} \Hom_{\varphi}(\CC,\CC').\end{equation}
\end{rem}

The following proposition says that the associated torsor construction can be extended to a 1-morphism of gerbes.  The proof, which is omitted, is a straightforward application of the definitions.
\begin{proposition} \label{ass gerbe coh} \begin{enumerate} \item Let $\varphi:A\to C$ be a homomorphism of sheaves of abelian groups over $M$.  Then there is a 1-morphism $\varphi[\cdot]:\textrm{B}(A)\to\textrm{B}(C)$ intertwining $\varphi$ defined on each open set $U\subset M$ by the functor $\varphi|_U[\cdot]:\tor_{A|_U}\to \tor_{B|_U}$ given in definition (\ref{ass tor sheaf}). 
\item Given another sheaf of abelian groups $D$ and a homomorphism $\psi:C\to D$, there exists a 2-morphism
\begin{equation} \psi[\varphi[\cdot]]\Rightarrow (\psi\varphi)[\cdot]. \end{equation} 
\end{enumerate}
\end{proposition}

We now recall several operations on gerbes which we will need.  First, given a gerbe $\CC$ over $M$ with band $A$, we may construct in an obvious way the opposite gerbe $\CC^{op}$, which assigns to each open set $U\subset M$ the category $\CC(U)^{op}$.  Since $A$ is abelian, $\CC^{op}$ canonically has the structure of a gerbe with band $A$.  There is a 1-morphism of gerbes 
\begin{equation} \CC\to \CC^{op} \end{equation} intertwining the homomorphism 
\begin{equation} g\mapsto g^{-1}.\end{equation}  This morphism is a bijection on the level of both objects and morphisms: on the set of objects it is the identity, and on the level of morphisms it maps 
\begin{equation} \psi:P\to Q\end{equation} to 
\begin{equation} \psi^{-1}:Q\to P.\end{equation}  The next construction involves a homomorphism $\varphi:A\to B$ of sheaves of abelian groups over $M$.  On page 199 of \cite{Br1} Brylinski describes the construction of an associated gerbe with band $B$ (Brylinski calls this gerbe $\CC\times^A B$, but following our earlier notation for torsors we will call it $\varphi[\CC]$).  Furthermore, from the construction of $\varphi[\CC]$ one obtains a 1-morphism 
\begin{equation} \varphi[\cdot]:\CC\to \varphi[\CC]\end{equation} intertwining $\varphi$.  As discussed in example (\ref{ass tor ex}) for torsors, one can use the associated gerbe construction to construct the tensor product of gerbes, as well as the dual of a gerbe.  

Next, let $f:M\to M'$ be a smooth map of manifolds.  Given a sheaf of groupoids $\CC$ over $M$ with band $A$, one constructs the direct image $f_*\CC$ in the obvious way: for example, for each open set $U\subset N$ we have 
\begin{equation} f_*(\CC)(U)=\CC(f^{-1}(U)).\end{equation}  Generically, if $\CC$ has the structure of a gerbe with band $A$, $f_*\CC$ will not necessarily have the structure of a gerbe with band $f_*A$; on the other hand, proposition (5.2.7) from \cite{Br1} shows that this will be under certain assumptions about the map $f$ and the sheaf of groups $A$.  We also point out that, given another smooth map $g:M'\to M''$, we have a natural identification of $g_*(f_*\CC)$ with $(gf)_*\CC$.

Given a gerbe $\CC$ over $M'$ with band $A$, we may also form the inverse image $f^*\CC$, which is a gerbe with band $f^*A$ over $M$.  If $\CC$ is a Dixmier-Douady gerbe (i.e. $A=\ul{\T}_{M'}$), then the inverse image sheaf $f^*(\ul{\T}_{M'})$ will generally not be equal to $\ul{\T}_M$, so that $f^*\CC$ is not a DD gerbe.  On the other hand, we have a natural sheaf homomorphism 
\begin{equation} \varphi:f^*(\ul{\T}_{M'})\to \ul{\T}_M,\end{equation} and we can form the associated gerbe $\varphi[f^*\CC]$, which is a DD gerbe over $M$; from now on we will use the notation $f^*\CC$ denote this gerbe over $M$.  We remark that, as in the case of the associated gerbe, the inverse image  $f^*\CC$ of a DD gerbe is characterized by the following universal property.   There exists a morphism of gerbes 
\begin{equation} f^*:\CC\to f_*(f^*\CC) \end{equation} intertwining the canonical homomorphism 
\begin{equation} f^*:\ul{\T}_{M'}\to f_*\ul{\T}_M.\end{equation}

\section{\label{section gerbe symmetries} Infinitesimal Symmetries of Gerbes} In this section, we explain how to generalize the discussion of \S\ref{VF} to define the infinitesimal symmetries of a Dixmier Douady gerbe $\CC$ over a manifold $M$.  Since $\CC$ is not itself a manifold, we cannot directly import structures from differential geometry such as vector fields or flows.  We  instead proceed by analogy to the circle bundle case.  To make the analogy clearer, we introduce a concept from algebraic geometry. The basic idea, inspired by the discussion in \cite{Ra}, is to use a ringed space $I_1$, sometimes called the ``dual numbers", which can be thought of as a first order formal curve.  In \cite{Ra}, the authors study maps from $I_1$ into a stack $X$ to motivate the definition of the tangent stack to $X$.  Similarly, given a gerbe $\CC$ over $M$, we will consider the pullback of $\CC$ to $I_1$ via a map $I_1\to M$.  This somewhat informal discussion motivates our definition (\ref{def symmetry gerbe}).  In \S\ref{local flows} (and in appendix A), we provide another perspective on  $\L_{\CC}$ by relating infinitesimal symmetries to families of (non-infinitesimal) symmetries of $\CC$ through a process analogous to differentiation.  
\subsection{Gerbes over the formal interval}
  Consider the ring 
\begin{equation} R=\R[\epsilon]/(\epsilon)^2.\end{equation}  We can form a ringed space 
\begin{equation} I_1=\textrm{Spec}(R),\end{equation} which has a single underlying point, and whose ring of functions is  
\begin{equation} \mathcal{O}(\{*\})=R.\end{equation}  Let $X$ be a smooth manifold, considered as a ringed space with its sheaf of smooth real valued-functions.  Then the set of maps from $I_1$ to $X$ in the category of ringed spaces is identified with the set of ring homomorphisms 
\begin{equation} \Hom(C^{\infty}(X),R).\end{equation}  Note in particular that we have an inclusion 
\begin{equation} \iota:\{*\}\hookrightarrow I_1 \end{equation} corresponding to the homomorphism from $R\to \R$ sending $a+b\epsilon\mapsto a$, as well as a retraction 
\begin{equation} r:I_1\to \{*\}\end{equation} corresponding to the inclusion $\R\hookrightarrow R$.  Given a map $\phi$ from $I_1$ to a smooth manifold $X$, let 
$\phi^*:  C^{\infty}(X)    \to R $ be the corresponding ring homomorphism, and write 
\begin{equation} \phi^*f=a(f)+b(f)\epsilon \end{equation} for each $f\in C^{\infty}(X)$.  Let $x=\phi\circ \iota(*)\in M$ be the image the underlying point of $I_1$ in $M$, then 
\begin{equation} a(f)=(\phi\circ\iota)^*f=f(x),\end{equation} i.e. $a:C^{\infty}(X)\to \R$ is the homomorphism which evaluates each function at the point $x$.  The condition that $\phi^*$ be a ring homomorphism then implies that for every pair of functions $f,g\in C^{\infty}(X)$ we have 
\begin{equation} b(fg)=f(x)b(g)+g(x)b(f),\end{equation} so that there is a unique tangent vector $\xi\in T_xX$ such that 
\begin{equation} b(f)=\xi(f).\end{equation} Conversely, any tangent vector $\xi\in TX$ determines a map 
\begin{equation} \phi_{\xi}:I_1\to M.\end{equation} Thus $\Map(I_1,X)$ is naturally identified with the tangent space of $X$.  Heuristically, we think of $I_1$ as a ``formal interval" and the map $\phi_{\xi}$ as an infinitesimal curve in $X$ in the direction $\xi$.

 Next, let $E$ be a principal $\T$-bundle over $M$, and let $\ul{E}$ be its sheaf of sections.  The restriction of $\ul{E}$ to the point $x$ is a skyscraper sheaf with stalk a torsor for the set of germs of smooth $\T$-valued functions at $x$.  On the other hand, there is a homomorphsim 
 \begin{equation} \rho_x:St_x(\ul{\T}_M)\to \ul{\T}_{\{x\}}\end{equation} which evaluates each germ at the point $x$ .  We define the restriction $(\ul{E})_x$ of $\ul{E}$ to $x$ (in the category of $\ul{\T}$-torsors) to be the associated torsor; note that we have a natural identification 
\begin{equation} (\ul{E})_x\cong \ul{(E_x)}.\end{equation}

Given a  Lie group $G$, the set of maps $\textrm{\Map}(I_1,G)$ also has the structure of a group which we identify with the product $G\times \g\cong TG$.  In particular, we may define a principal $\T$-bundle over $I_1$ to be a $\T\times i\R$-torsor, and we can similarly define a DD gerbe over $I_1$.   Given a tangent vector $\xi\in T_xM$ thought of as a map $\phi_{\xi}:I_1\to M$, we can then construct the inverse image torsor $\phi_{\xi}^*\ul{E}$ over $I_1$.  There is a restriction map 
\begin{equation} \ul{E}\to (\phi_{\xi})_*(\phi_{\xi})^*\ul{E}\end{equation} intertwining the map
\begin{equation} g\mapsto (g(x),\iota_{\xi_x}d\log(g)).\end{equation}  The maps $\iota:\{*\}\to I_1$ and $r:I_1\to \{*\}$ yield maps
 \begin{equation} \label{map 1}\iota^*:\phi_{\xi}^*\ul{E}\to \ul{E}_x\end{equation} and 
 \begin{equation} r^*:\ul{E}_x\to \phi_{\xi}^*\ul{E}\end{equation} such that $\iota^*\circ r^*=\textrm{id}_{\ul{E}_x}.$  The kernel of (\ref{map 1}) is isomorphic to $\iota_{\xi}d\log[\ul{E}]_x$,  and we therefore obtain an isomorphism
 \begin{equation} \phi_{\xi}^*\ul{E}\cong \ul{E}_x\times \iota_{\xi}d\log[\ul{E}]_x. \end{equation} Recall from example (\ref{ex tor T sheaf}) that the second factor can be identified with the set of $\T$-invariant lifts of $\xi_x$ to $E$.  Thus, we can encode the infinitesimal symmetries of $E$ point-wise by pulling back $\ul{E}$ to $I_1$; informally this amounts to restricted $E$ to an infinitesimal curve in $M$.
 
Recall from the previous section that the inverse image and associated torsor construction also make sense for gerbes.  Thus, given a tangent vector $\xi\in T_xM$ and a DD gerbe $\CC$ over $M$ we can form a gerbe $\phi_{\xi}^*\CC$ over $I_1$ with band $\T\times i\R$.  This gerbe  is naturally isomorphic to a product
\begin{equation} \CC_x\times \iota_{\xi}dlog[\CC]_x,\end{equation} where the first factor should be thought of as the fiber of $\CC$ at $x$ and the second term encodes the infinitesimal symmetries of $\CC$ near $x$ lifting $\xi$.  It is therefore natural to define the infinitesimal symmetries of $\CC$ lifting $\xi$ as the associated gerbe $\iota_{\xi}d\log[\CC]$.

  Rather than working with the associated gerbe $\iota_{\xi}d\log[\CC]$, we will instead  use of an alternative definition of the infinitesimal symmetries of $\CC$.  Recall from section \S\ref{gerbes} that  there is a morphism of gerbes $\CC\to \iota_{\xi}d\log[\CC]$ intertwining $\iota_{\xi}d\log$.  In particular, for each local section $Q$ of $\CC$ we obtain a local section $\iota_{\xi}d\log[Q]$ of $\iota_{\xi}d\log[\CC]$. Suppose we are given a global section $S$ of $\iota_{\xi}d\log[\CC]$.  Then for each open set $U\subset M$ and each $Q\in \CC(U)$, we can form the $\ul{i\R}_U$-torsor $\ul{\Hom}(\iota_{\xi}d\log[Q],S|_U)$.  More formally, the section $S$ defines a morphism of gerbes from  $\CC^{op}$  to $\textrm{B}(\ul{i\R}_M)$ intertwining  $\iota_{\xi}d\log$.  Moreover one can check that the assignment 
  \begin{equation} S\mapsto \ul{\Hom}(\iota_{\xi}d\log[\cdot],S) \end{equation} defines an isomorphism from the category  of global sections of $\iota_{\xi}d\log[\CC]$ to the category 
  \begin{equation} \Hom_{\iota_{\xi}d\log}(\CC^{op},\textrm{B}(\ul{i\R}_M)) \end{equation} of 1-morphisms from $\CC^{op}$ to $\textrm{B}(\ul{i\R}_M)$ intertwining $\iota_{\xi}d\log$.  Using the canonical 1-morphism from $\CC$ to $\CC^{op}$ intertwining $g\mapsto g^{-1}$, we have a canonical equivalence (which is actually a bijection on the level of objects and morphisms)
  \begin{equation}\Hom_{\iota_{\xi}d\log}(\CC^{op},\textrm{B}(\ul{i\R}_M)) \tocong \Hom_{-\iota_{\xi}d\log}(\CC,\textrm{B}(\ul{i\R}_M)).\end{equation}
  
    \begin{definition}\label{def symmetry gerbe}  Let $\CC$ be a DD gerbe over a manifold $M$, and let $\xi$ be a vector field on $M$.  Then the category of infinitesimal symmetries of $\CC$ lifting $\xi$ is 
 \begin{equation} \L_{\CC}(\xi)=\Hom_{-\iota_{\xi}dlog}(\CC,\textrm{B}(i\R_M)).\end{equation} 
 \end{definition}
 \begin{notation} We will often  use the notation $\tx$ to denote an element of $\L_{\CC}(\xi)$; we call $\tx$ a \emph{lift} of $\xi$ to $\CC$.  We also introduce the following notational convention, which in the present context is vacuous but will prove notationally useful when we consider connective lifts in \S\ref{conn lifts}: given a lift $\tx\in \L_{\CC}(\xi)$, we will say $\tx$ \emph{determines} a 1-morphism 
 \begin{equation} F_{\tx}:\CC\to \textrm{B}(\ul{i\R}_M).\end{equation}  Of course, by definition an element of $\L_{\CC}(\xi)$ is a 1-morphisms, so in the present case $\xi=F_{\tx}$.   
 \end{notation}
\begin{rem} Although for simplicity we have defined a \emph{category} of symmetries of $\CC$ lifting $\xi$, we clearly could have defined a gerbe of symmetries $\ul{\Hom}_{-\iota_{\xi}dlog}(\CC,\textrm{B}(\ul{i\R}_M))$ assigning to each open set $U\subset M$ the category of morphisms  from  $\CC|_U$ to $\textrm{B}(\ul{i\R}_U)$ intertwining  $-\iota_{\xi|_U}d\log$.  In other contexts (for example in a holomorphic setting) it may be that the gerbe analogous to $\ul{\Hom}_{-\iota_{\xi}dlog}(\CC,\textrm{B}(i\R_M))$ has no global sections, in which case it would be crucial to work with the gerbe of symmetries itself. 
\end{rem}
 
\begin{ex} \label{trivial inf lift}   Suppose that $\CC=\textrm{B}(\ul{\T}_M)$ is the trivial Dixmier-Douady gerbe over $M$.  
\begin{definition}\label{trivial lift} The \emph{trivial lift} of a vector field $\xi$ to $\textrm{B}(\ul{\T}_M)$ is the 1-morphism 
\begin{equation} \tx_0=-\iota_{\xi}d\log[\cdot]:\textrm{B}(\ul{\T}_M)\to \textrm{B}(\ul{i\R}_M)\end{equation} sending each $\ul{\T}_U$-torsor to the associated $\ul{i\R}_U$-torsor described in proposition (\ref{ass gerbe coh}).
\end{definition}    To understand the geometric origin of the trivial lift, consider by analogy the trivial $\ul{\T}_M$-torsor $\ul{\T}_M$.  A section of this torsor is simply a function $f$, and since we can compare the value of $f$ at distinct points on $M$ we can differentiate $f$.  Put differently, given a diffeomorphism $\varphi:M\to M$, we can pull-back a principal bundle $P$ to obtain a different principal bundle, but if $P$ is the trivial bundle then we can pull-back \emph{sections} of $P$ via $\varphi$ to obtain sections of the same bundle.  Similarly, given any gerbe $\CC$ over $M$, we can pull back via $\varphi$ to obtain another gerbe over $M$, but if $\CC$ is the trivial gerbe then we can pull-back sections of $\CC$, i.e. principal bundles, and $-\iota_{\xi}d\log[\cdot]$ is (minus) the corresponding directional derivative operation.  

This analogy can be made more precise by considering the relationship between 1-parameter families of symmetries of gerbes (in particular, $\R$-equivariant gerbes) and infinitesimal symmetries, as discussed in appendix $A$.  Given any 1-parameter family $\Phi$ of diffeomorphisms of $M$, there is a canonical lift of $\Phi$ to the trivial gerbe over $M$.  Applying the differentiation functor defined in (\ref{def diff fun}) one obtains a lift of the vector field generating $\Phi$ which is naturally isomorphic to the one defined in definition (\ref{trivial lift}). 

More generally, let $G$ be a Lie group acting on $M$.  If $\CC$ is given the structure of a $G$-equivariant gerbe, then for every 1-parameter subgroup of $G$ we obtain an infinitesimal symmetry of $\CC$ by applying the differentiation functor. 
\end{ex}

\section{\label{Cech}The Cech picture}
 
In this section we introduce a more concrete ``Cech" version of the category $\L_{\CC}$.  In this picture one works not with sheaves of categories but with (locally defined) differential forms.  We could similarly develop a version of $\L_{\CC}$ in terms of other models of gerbes, such as bundle gerbes.

\begin{definition} \label{local triv}  Let $\CC$ be a DD gerbe over a manifold $M$.  A \emph{collection of local trivializations} of $\CC$ consists a triple $\{\{U_i\},\{Q_i\},\{s_{ij}\}\}$ where 
\begin{enumerate} \item $\{U_i\}$ is an open cover of $M$,
\item $Q_i$ is a section of $\CC(U_i)$, and 
\item $s_{ij}:Q_i|_{U_{ij}}\to Q_{j}|_{U_{ij}}$, where $U_{ij}=U_i\cap U_j$.
\end{enumerate}  
\end{definition}
\begin{rem}  Such a collection of local trivializations gives rise to \emph{Cech data} for $\CC$, which is a collection of functions $\{g_{ijk}\}$ from triple intersections $U_{ijk}:=U_i\cap U_j\cap U_k$ to  $\T$.  In terms of the discussion of the descent property for gerbes in appendix B, the functions $g_{ijk}$ are are a cocycle whose cohomology class is the obstruction to gluing the local data $\{\{Q_i\},\{s_{ij}\}\}$ to form a global section of $\CC$.  Explicitly, we define  
\begin{equation}\label{cocycle definition} s_{jk}\circ s_{ij}=s_{ik}\cdot g_{ijk}. \end{equation}  Because of the associativity of the composition of morphisms, it follows that on 4-fold overlaps $U_{ijkl}$ the cocycle condition
\begin{equation} \label{g cocycle} g_{jkl}g_{ikl}^{-1}g_{ijl}g_{ijk}^{-1}=1\end{equation} is satisfied.

\end{rem}

\begin{notation} If $\mathcal{F}$ is any sheaf over an open set $U\subset M$, we will use the notation $x\in \mathcal{F}$ to mean that $x$ is a global section of $\mathcal{F}$ over $U$.
\end{notation}

\begin{definition}\label{def local triv}  Let $\tx$ be a lift of a vector field $\xi$ to $\CC$.  A \emph{collection of local trivializations} for $\tx$ relative to $\{\{U_i\},\{Q_i\},\{s_{ij}\}\}$ is a choice of sections $r_i\in F_{\tx}(Q_i)(U_i)$ for each $i$.
\end{definition}

Given such a collection of local trivializations of $\tx$, we can define $i\R$-valued functions $f_{ij}$ on double intersections according to the formula
\begin{equation} \label{lift Cech}r_j=F_{\tx}(s_{ij})(r_i)+f_{ij}. \end{equation}  We will refer to these functions as \emph{Cech data} for the lift $\tx$.  Definition (\ref{def symmetry gerbe}) together with equation (\ref{cocycle definition}) imply that on triple intersections we have 
\begin{equation}\label{lift delta} f_{jk}-f_{ik}+f_{ij}=\iota_{\xi}d\log(g_{ijk}), \end{equation} where $\{g_{ijk}\}$ is the Cech data corresponding to the local trivializations of $\CC$.  

Given another lift $\tx'$ of $\xi$ to $\CC$ with sections $r_i'$ of $F_{\tx'}(Q_i)$, and given an equivalence $T:F_{\tx}\to F_{\tx'}$, define functions $u_i:U_i\to i\R$ by 
\begin{equation} r_i'=T_{Q_i}(r_i)+u_i.\end{equation}  The naturality of $T$ implies that on double intersections we have 
\begin{equation}\label{cech morphism} f_{ij}-f'_{ij}=u_j-u_i.\end{equation}  

These considerations motivate the following definition.
\begin{definition} \label{Cech lift data} Let $\{g_{ijk}:U_{ijk}\to \T\}$ be a Cech cocycle on a manifold $M$ with respect to a good open cover $\{U_i\}$.  The category $\L_{g_{ijk}}$ has as objects the set of pairs  $(\xi,\{f_{ij}\})$, where $\xi$ is a vector field on $M$ and $\{f_{ij}:U_{ij}\to i\R\}$ is a collection of functions satisfying equation (\ref{lift delta}) on triple intersections.  A morphism from $(\xi,\{f_{ij}\})$ to $(\xi,\{f'_{ij}\})$ is a collection of functions $\{u_i:U_i\to i\R\}$ satisfying equation (\ref{cech morphism}).
\end{definition}

\begin{notation}Given a fixed vector field $\xi$, let $\L_{g_{ijk}}(\xi)$ denote the subcategory of $\L_{g_{ijk}}$ with objects of the form $(\xi,\{f_{ij}\})$ for some collection of functions $\{f_{ij}\}$.  
\end{notation}
\begin{theorem}\label{cech lift class} For each vector field $\xi$ on $M$ 
\begin{enumerate}[(1)] \item The set of isomorphism classes of objects $\pi_0(\L_{g_{ijk}})$ has exactly one element.
\item For any object $\tx=(\xi,\{f_{ij}\})$, the map which to a smooth function $h:M\to i\R$ associates the automorphism of $\tx$ given by $\{u_i=h|_{U_i}\}$ determines an isomorphism of groups 
\begin{equation} C_M^{\infty}(i\R)\tocong \Aut(\tx).\end{equation}
\end{enumerate}
\end{theorem}
\begin{proof}  
Consider $h_{ijk}=d\log(g_{ijk}):=g^{-1}_{ijk}dg_{ijk}$; this is a Cech cocycle with values in the sheaf $\ul{i\R}_M$.  By definition an element of $\L_{g_{ijk}}(\xi)$ is a trivialization of this cocycle, i.e. a Cech cochain $\{f_{ij}\}$ such that 
\begin{equation} (\delta f)_{ijk}=h_{ijk}.\end{equation}  Since  $\ul{i\R}_M$ is a fine sheaf (i.e. admits partitions of unity), the cohomology class $[h_{ijk}]\in H^3(M,\ul{i\R})$ is necessarily zero, and since $\{U_i\}$ is a good open cover we can find such a trivialization $\{f_{ij}\}$.  Thus, for each vector field $\xi$ $L_{g_{ijk}}(\xi)$ is non-empty.  On the other hand, given another lift $\{f'_{ij}\}$, it follows that $\{f_{ij}-f'_{ij}\}$ is closed, and therefore also exact.  Thus we can find a Cech cochain $\{u_i\}$ such that for each $i,j$ we have 
\begin{equation} f_{ij}-f'_{ij}=u_j-u_i.\end{equation} Thus any two lifts are isomorphic.  

Given an element $\tx=(\xi,\{f_{ij}\})\in \L(\xi)$, if $\{u_i\}$ is an automorphism of $\tx$ than by equation (\ref{cech morphism}) we must have $u_j=u_i$ on $U_i\cap U_j$, and therefore there is a global function $f$ with $u_i=f|_{U_i}$; conversely any such function gives an automorphism of $\tx$.
\end{proof}

We saw above that every element of $\L_{\CC}$ (non-canonically) gives rise to an element of $\L_{g_{ijk}}$.  This is completely analogous to the description of a principal $\T$-bundle in terms of local sections and transition functions.  On the other hand, given a $\T$-valued Cech cocycle $g_{ij}$ we can produce a $\T$-bundle by gluing trivial bundles.  There is a similar procedure in our situation for producing lifts out of Cech data. 
\begin{proposition}  \label{cech equiv} Let $\CC$ be a DD gerbe over a manifold $M$, and let $\{\{U_i\},\{Q_i\},\{s_{ij}\}\}$ be a collection of local trivializations for $\CC$ with associated Cech cocycle $g_{ijk}$.  Then there is an equivalence of categories $\L_{g_{ijk}}\to \L_{\CC}$.
\end{proposition}
 We defer the proof to appendix $B$.  To see a similar argument to that used there, we direct the reader to the proof of proposition (5.3.2) in \cite{Br1}.

\begin{corollary}\label{lift class} For each vector field $\xi$ on $M$
\begin{enumerate} \item The set of isomorphism classes of objects $\pi_0(L_{\CC})$ has exactly one element.
\item For any object $\tx\in \L_{\CC}(\xi)$, the group of automorphisms of $\tx$ is isomorphic to $C^{\infty}_M(i\R)$.  Given $f\in C^{\infty}_M(i\R)$, the associated automorphism $\tau$ of $\tx$ is given by 
\begin{equation} (\tau)_Q=\alpha^{-1}_Q(f|_U):F_{\tx}(Q)\to F_{\tx}(Q) \end{equation} for each $Q\in \CC(U)$, where $\alpha_Q:\ul{\T}_U\to \ul{\textrm{Aut}}(Q)$ is as in definition (\ref{def gerbe}) (see also definition (\ref{2-morphism}).
\end{enumerate}
\end{corollary} 

\section{\label{local flows} Local flows and infinitesimal symmetries}

Let $\Phi$ be a 1-parameter group of diffeomorphisms of $M$ generated by a vector field $\xi$.  In appendix $A$ we describe the category $\L_{\CC}(\Phi)$ of lifts of $\Phi$ to $\CC$ and construct a differentiation functor 
\begin{equation} \L_{\CC}(\Phi)\to \L_{\CC}(\xi).\end{equation}
Roughly speaking, an element of $\L_{\CC}(\Phi)$ may be thought of as a flow on $\CC$, and the corresponding infinitesimal symmetry as a vector field.  In the case of manifolds, every vector field can be integrated locally to obtain a unique flow, and this establishes a bijective correspondence. Similarly, in this section we will establish an equivalence between local versions of the categories $\L_{\CC}(\Phi)$ and $\L_{\CC}(\xi)$.  
  
  To describe the local version of the category $\L_{\CC}(\Phi)$, it will be convenient to use employ the language of simplicial manifolds.  Given a vector field $\xi$ on $M$, for each $x$ in $M$ there exists an open set $U$ containing $x$, a positive number $\tilde{\epsilon}$, and a smooth map 
\begin{align} \Phi:U\times (-\tilde{\epsilon},\tilde{\epsilon}) & \to M \\
x, t & \mapsto \varphi_t(x),\notag\end{align}  such that for each $t$, $\varphi_t$ is a diffeomorphism from $U$ to $\varphi_t(U)$, and such that, for each $(x,t)\in U\times (-\tilde{\epsilon},\tilde{\epsilon})$, we have 
\begin{equation} \frac{d}{ds}|_{s=t}\varphi_s(x)=\xi_{\varphi_t(x)}.\end{equation} 
Next, we can choose a smaller open set $V\subset U$ and a smaller positive number $\epsilon<\tilde{\epsilon}$ such that, for each $t\in I=(-\epsilon,\epsilon)$ and each $x\in V$ we have $\varphi_t(x)\in U$.  For each $k\geq 0$, consider the subset of $V\times \R^k$ given by 
\begin{equation} U^k=\{(\varphi_{t_0}(x),t_1,\cdots, t_k): x \in V, \sum_{i=0}^k|t_i|<\epsilon \} .\end{equation}  Note that $U^k$ is open since it is a union of open sets of the form $\varphi_t(V)\times W_t$, where for each $t\in I$ we define 
\begin{equation} W_t=\{(t_1,\cdots t_k):\sum_{i=1}^k|t_i|< \epsilon-|t|\}.\end{equation}  The open sets $U^{\bullet}$ fit together into a simplicial manifold.  The boundary maps $p_j:U^k\to U^{k-1}$ for $j=0,1\cdots k$ are given by 
\begin{equation} p_0(x,t_1,\cdots,t_k)=(\varphi_{t_1}(x),t_2,\cdots,t_k), \end{equation} 
\begin{equation} p_i(x,t_1,\cdots,t_k)=(x,\cdots,t_i+t_{i+1},\cdots,t_k) \end{equation} for $i=1,\cdots k-1$, and 
\begin{equation} p_k(x,t_1,\cdots,t_k)=(x,t_1\cdots,t_{k-1}).\end{equation}  The degeneracy maps $s_i:U^k\to U^{k+1}$ for $i=0,\cdots k$ are given by 
\begin{equation} s_0(x,t_1,\cdots,t_k)=(x,0,t_1,\cdots,t_k)\end{equation} and 
\begin{equation} s_i(x,t_1,\cdots,t_k)=(x,t_1,\cdots,t_i,0,\cdots, t_k)\end{equation} for $i=1,\cdots k$.

It is easily checked that these satisfy the correct relations to define a simplicial manifold.

\begin{notation} Given a $\ul{\T}_{U^k}$-torsor $S$, we define a $\ul{\T}_{U^{k+1}}$-torsor\footnote{In this formula and the next we assume $k+1$ is even.  If $k+1$ is odd, the last term is instead $p_{k+1}^*P^{\vee}$, respectively $p_{k+1}^*g^{-1}$.} 
\begin{equation} \delta S=p_0^*S\otimes p_1^*S^{\vee}\otimes\cdots\otimes p_{k+1}^*S.\end{equation}  Similarly, given a function $g:U^k\to \T$, define $\delta g:U^{k+1}\to \T$ by 
\begin{equation} (p_0^*g)(p_1^*g^{-1})\cdots (p_{k+1}^*g).\end{equation}
It them follows from (\ref{simp relations}) that we have natural isomorphisms 
\begin{equation} s_0^*\delta S\cong p_0^*s_0^*S \end{equation} and 
\begin{equation} s_1^*\delta S \cong p_1^*s_0^*S \end{equation} over $U^1$.

\end{notation}

Any DD gerbe is locally isomorphic to the trivial gerbe with band $\ul{\T}$, and therefore without loss of generality we will restrict ourselves to the case that $\CC$ is the trivial gerbe in the following.  The following definition is a local version of an equivariant gerbe (for the group $\R$), see for example \cite{Mein}, \cite{G2}, \cite{Br2}.
\begin{definition} \label{local lift cat} Let $U^{\bullet}$ be the simplicial manifold described above.  A local lift of $\Phi$ to the trivial gerbe consists of a $\ul{\T}_{U^1}$-torsor $S$ over $U^1$, together with a section $e$ of $s_0^*S$ over $U^0$ and a section $\sigma$ of $\delta E$ over $U^2$ satisfying the following conditions: 
\begin{enumerate}[(i)] \item $s_0^*\sigma=p_0^*e$ and $s_1^*\sigma=p_1^*e$,
\item $\delta \sigma$ is equal to the canonical section of $\delta^2S$ over $U^3$.
\end{enumerate}  Given a pair of local lifts $\hat{\Phi}=(S,e,\sigma)$ and $\hat{\Phi}'=(S',e',\sigma')$, an isomorphism from $\hat{\Phi}$ to $\hat{\Phi}'$ is an isomorphism of torsors $\Psi:S\tocong S'$ compatible with the sections $e,e',\sigma,\sigma'$.  We will denote the corresponding category $\L(\Phi)$.
\end{definition}   
\begin{proposition}\label{local lift class}  $\L(\Phi)$ has a single isomorphism class of objects. 
 \end{proposition}

 \begin{proof}  Note that there is a distinguished element $\hat{\Phi}_0\in \L(\Phi)$, which we call the \emph{trivial lift}.  Namely, we take $S$ to be the trivial $\ul{\T}_{U^1}$-torsor, and $e$ and $\sigma$ to be the trivial sections.  We will show that an arbitrary element $(S,e,\sigma)\in \L(\Phi)$ is isomorphic to $\hat{\Phi}_0$.  Note that $s_0^*S$ can be identified with the restriction of $S$ to $U^0\times \{0\}\subset U^1$.  We therefore begin by extending the section $e$ to a smooth section $\tau$ of $S$.  We can then define a smooth function $g:U^2\to \T$ by the formula 
 \begin{equation} \sigma =\delta\tau \cdot g.\end{equation}  Actually, it will be more convenient to work with logarithm of $g$; this is well-defined since $U^2$ deformation retracts onto $U^0\times \{(0,0)\}$ and since by condition (i) in definition (\ref{local lift cat}) we have $g(x,0,0)=1$.  Thus, we let $f:U^2\to i\R$ be the unique function such that $f(x,0,0)=0$ and $g=e^f$.  Condition (i) in definition (\ref{local lift cat}) is equivalent to the condition that, for all $(x,t)\in U^1$ we have 
 \begin{equation} f(x,t,0)=f(x,0,t)=0, \end{equation} and by condition (ii) we have $\delta f(x,t,t',t'')=$
 \begin{equation} \label{f cocy}f(\varphi_t(x),t',t'')-f(x,t+t',t'')+f(x,t,t'+t'')-f(x,t,t')=0.\end{equation}  Furthermore, it is easy to see that specifying an isomorphism from $\hat{\Phi}_0$ to $(S,e,\sigma)$ is equivalent to giving a smooth function $h:U^1\to i\R$ such that $\delta h=f$, i.e. for each $(x,t,t')\in U^2$, $h$ satisfies
 \begin{equation} h(\varphi_t(x),t')-h(x,t+t')+h(x,t)=f(x,t,t'),\end{equation} and such that $h(x,0)=0$ for all $x$.  Define a function $k:U^1\to i\R$ by the equation 
 \begin{equation} k(x,t)=\frac{\partial}{\partial u}|_{u=0}f(x,t,u).\end{equation}  Differentiating equation (\ref{f cocy}) with respect to $t''$ at $t''=0$ we obtain the relation
 \begin{equation} k(x,t+t')-k(\varphi_t(x),t')=\frac{\partial}{\partial u}|_{u=t'}f(x,t,u).\end{equation}  If we then define 
 \begin{equation} h(x,t)=-\int_0^t k(x,s)ds, \end{equation} we have 
 \begin{align} & h(\varphi_t(x),t')-h(x,t+t')+h(x,t) \\ 
  = &-\int_0^{t'}k(\varphi_t(x),s)ds+\int_0^{t+t'}k(x,s)ds-\int_0^tk(x,s)ds \\
  = &-\int_0^{t'}k(\varphi_t(x),s)ds+\int_t^{t+t'}k(x,s)ds \\ 
  = &\int_0^{t'}[k(x,t+s)-k(\varphi_t(x),s)]ds \\
  = &\int_0^{t'}\frac{\partial}{\partial u}|_{u=s}f(x,t,u)ds \\
  = & f(x,t,t')-f(x,t,0)=f(x,t,t').\end{align}
 \end{proof}
 
   We now define a concrete local version of the differentiation functor (\ref{def diff fun}) considered in appendix A.  Given a vector field $\xi$ on $U^0$, let $\L(\xi)$ denote $\L_{\textrm{B}(\ul{\T}_{U^0})}(\xi)$.  Then we will construct a functor 
   \begin{equation} D:\L(\Phi)\to \L(\xi).\end{equation}

 It follows  from the results of \S\ref{Cech} that every lift $\tx$ of $\xi$ to $\textrm{B}(\ul{\T}_{U^0})$ is determined up to canonical isomorphism by its action on the trivial $\ul{\T}_{U^0}$-torsor.  Put differently, there is a functor
 \begin{align} \L(\xi) & \to \tor_{\ul{i\R}_{U^0}} \\ \tx & \mapsto F_{\tx}(\ul{\T}_{U^0}),\notag \end{align}  and using proposition (\ref{lift class}) one can show this is an equivalence of categories.  For example, under this isomorphism the trivial lift discussed in example (\ref{trivial inf lift}) is sent to $\iota_{\xi}d\log[\ul{\T}_{U^0}]$, which is canonically isomorphic to the trivial $\ul{i\R}_{U^0}$-torsor.  For simplicity we will therefore take the codomain of the differentiation functor to be the category $\tor_{\ul{i\R}_{U^0}}$.

\begin{definition} \label{local diff fun}\begin{equation} D:\L(\Phi) \to \tor_{\ul{i\R}_{U^0}}\end{equation} is the functor taking $(S,e,\sigma)$ to the $\ul{i\R}_{U^0}$-torsor
 \begin{equation} s_0^*(\iota_{\frac{d}{dt}}d\log[S]).\end{equation} \end{definition}

  \begin{theorem}\label{diff equiv} $D$ is an equivalence of categories.
 \end{theorem}
 \begin{proof}  We proceed to by showing that $D$ is both essentially surjective and fully faithful. Since $\tor_{\ul{i\R}_{U^0}}$ has a single isomorphism class of objects, $D$ is trivially essentially surjective.  Furthermore, since $\L(\Phi)$ has a single isomorphism class of objects, to check that $D$ is fully faithful it is sufficient to check that $D$ induces an isomorphism of groups 
 \begin{equation} \Aut(\hat{\Phi}_0)\to \Aut(D(\hat{\Phi}_0)),\end{equation}  where $\hat{\Phi}_0$ is the trivial lift.

 Define the set \begin{equation}\label{ T coc}  Z_{\T}:=\{g:U^1\to \T:\delta g=1, s_0^*g=1\}.\end{equation}  There is an isomorphism 
 \begin{equation} Z_{\T}\tocong \Aut(\hat{\Phi}_0) \end{equation} sending each function $g$ to the bundle automorphism of $\ul{\T}_{U^1}$ given by right multiplication.  Define 
 \begin{equation} D_{\T}:Z_{\T}\to C^{\infty}(U^0;i\R)\end{equation} by 
 \begin{equation} g\mapsto \iota_{\frac{d}{dt}}d\log(g)|_{U^0}.\end{equation}  Then the definition of the functor $D$ implies that there is a commutative diagram
 \begin{equation} \xymatrix{ Z \ar[d]_{D_{\T}} \ar[r]^-{\cong} & \Aut(\hat{\Phi}_0) \ar[d]^D \\ C^{\infty}(U^0;i\R) \ar[r]^-{\cong} & \Aut(D(\hat{\Phi}_0))}\end{equation}  Since both horizontal arrows are isomorphisms, to finish the proof we must show that $D_{\T}$ is an isomorphism.  To see this, note that for any $g\in Z$, the condition $s^*g=1$ can be written 
 \begin{equation} g(x,0)=1\end{equation} for every $x\in U^0$.  Therefore there exists a unique $f:U^1 \to i\R$ such that 
 \begin{equation} g(x,t)=e^{f(x,t)} \end{equation} and 
 \begin{equation} f(x,0)=0.\end{equation}  The condition $\delta g=1$ then implies that for each $(x,t,t')\in U^1$ we have 
 \begin{equation} f(\varphi_t(x),t')-f(x,t+t')+f(x,t)=0.\end{equation}  Differentiating with respect to $t'$ we obtain 
 \begin{equation} \frac{\partial f}{\partial s}|_{s=0}f(\varphi_t(x),s)=\frac{\partial f}{\partial s}|_{s=t}f(x,s). \end{equation}  Define 
 \begin{equation}\label{diff inverse} h(x)= D_{\T}(g)(x)=\frac{\partial f}{\partial s}|_{s=0}f(\varphi_t(x),s).\end{equation}  By the fundamental theorem of calculus we have 
 \begin{equation} f(x,t)=\int_0^th(\varphi_s(x))ds, \end{equation} so that $g$ is completely determined by $D_{\T}(g)=h$.  Conversely, given an arbitrary function $h:U^0\to i\R$, the function on $U^1$ defined by 
 \begin{equation} g(x,t)=e^{\int_0^th(\varphi_s(x))ds}\end{equation} is in $Z_{\T}$.

\end{proof}

\section{\label{operations}Operations on lifts}

Given diffeomorphisms $\varphi,\psi:M\to M$, suppose that $\tilde{\varphi}$ $\tilde{\psi}$ are lifts of these symmetries to $\CC$.  Then we can compose $\tilde{\varphi}$ and $\tilde{\psi}$ to obtain a symmetry of $\CC$ covering $\varphi\psi:M\to M$.  In this way the category of symmetries of $\CC$ obtains a structure analogous to that of a group.  Infinitesimally, this group structure gives rise to a bracket operation on the infinitesimal symmetries of $\CC$.  In this section we give a direct definition of this bracket, together with other structures on $\L_{\CC}$ analogous to  those possessed by a Lie algebra.  

\begin{proposition} \label{Lie algebra structure}\begin{enumerate}[(I)]\item There exists a functor $\bp:\L_{\CC}\times \L_{\CC}\to \L_{\CC}$ such that, if $\tx,\te$ are lifts of vector fields $\xi,\eta$ to $\CC$, then $\tx\bp\te$ is a lift of $\xi+\eta$ satisfying 
\begin{equation} F_{\tx\bp\te}(Q)=F_{\tx}(Q)\bp F_{\te}(Q)\end{equation} for each object $Q\in \CC(U)$.
\item For each real number $\lambda\in \R$ there is a functor $\lambda\odot:\L_{\CC}\to \L_{\CC}$ such that, if $\tx$ is a lift of $\xi$, then $\lambda[\xi]$ is a lift of $\lambda\xi$ satisfying 
\begin{equation} F_{\lambda\odot\tx}(Q)=\lambda\odot (F_{\tx}(Q)),\end{equation}.
\item There is a functor $[\cdot,\cdot]:\L_{\CC}\times\L_{\CC}\to \L_{\CC}$ such that, if $\tx,\te$ are lifts of $\xi,\eta$, then $[\tx,\te]$ is a lift of $[\xi,\eta]$ satisfying 
\begin{equation} F_{[\tx,\te]}(Q)=\xi[F_{\te}(Q)]\boxminus\eta[F_{\tx}(Q)].\end{equation} 
\item There exists a canonical lift $\zeta$ of the zero vector field to $\CC$ such that for each open set $U\subset M$, $F_{\zeta,U}:\CC(U)\to \BiR(U)$ is the constant functor sending every object to the sheaf of groups $\ul{i\R}_U$.  We call this the \emph{zero lift}.
\end{enumerate}
\end{proposition}
\begin{proof} Let $+:\ul{i\R}_M\times\ul{i\R}_M\to\ul{i\R}_M$ denote the addition homomorphism, and recall from proposition (\ref{ass gerbe coh}) that there exists a canonical 1-morphism of gerbes $\cdot\boxplus\cdot:\BiR\times\BiR \to \BiR$ intertwining $+$ and sending $R,S\in \tor_{\ul{i\R}(U)}$ to $R\bp S$ for each open set $U\subset M$.  Given lifts $\tx,\te$ of vector fields $\xi,\eta$ to $\CC$, we then define $\tx\bp\te$ to be the composition 
\begin{equation} \xymatrix{\CC \ar[r]^-{\tx\times\te} & \BiR\times \BiR \ar[r]^-{\cdot\bp\cdot} & \BiR.}\end{equation}  If $\psi:\tx\Rightarrow\tx'$ and $\varphi:\te\Rightarrow\te'$ are equivalences of lifts, we define $\psi\bp\varphi:\tx\bp\te\to \tx'\bp\te'$ to be the horizontal composition of the identity 2-morphism $\cdot\boxplus\cdot\Rightarrow \cdot\boxplus\cdot $ with the 2-morphism $\psi\times\varphi: \tx\times\te\Rightarrow\tx'\times\te'$.  The functors $\lambda[\cdot]$ and $[\cdot,\cdot]$ are defined similarly.

We omit the proof of (IV), which is straightforward.
\end{proof}

 One can further show that the operations above satisfy the axioms of a Lie algebra up to natural isomorphism.  For example, given lifts $\tx,\te,$ and $\check{\tau}$ there is a natural isomorphism 
\begin{equation} (\tx\bp\te)\bp\hat{\tau}\tocong \tx\bp(\te\bp\hat{\tau}).\end{equation}  Furthermore, these natural isomorphisms themselves satisfy various coherence conditions, similar to (but much more elaborate than) those satisfied by the associator in a monoidal category.  Roughly speaking, the category $\L_{\CC}$ has the structure of a ``Lie algebra object in the 2-category of categories."  Rather than give a precise definition of such an algebraic structure (which so far as we know does not exist in the literature), we will instead work in the Cech picture, where the relevant algebraic structure can be described using $L_{\infty}$-algebras.  

Thus, let us choose a collection of local trivializations $\{\{U_i\},\{Q_i\},\{s_{ij}\}\}$ for $\CC$ with corresponding Cech data $\{g_{ijk}\}$. Given lifts $\tx,\te$ of vector fields $\xi,\eta$ to $\CC$, let $\{r^{\tx}_i\},\{r^{\te}_i\}$ be local sections for these lifts, and let $\{f^{\tx}_{ij}\},\{f^{\te}_{ij}\}$ be the corresponding Cech data defined in equation (\ref{lift Cech}).  Then we obtain local sections \begin{equation} r_i^{\lambda\odot\tx}=\lambda\odot r_i^{\tx},\end{equation} \begin{equation}r_i^{\tx\bp\te}=r_i^{\tx}\bp r_i^{\te},\end{equation} and  \begin{equation}r_i^{[\tx,\te]}=\xi[r^{\te}_i]\boxminus\eta[r^{\tx}]\end{equation} of the lifts $\lambda\odot\tx$, $\tx\bp\te$, and $[\tx,\te]$, respectively.  The corresponding Cech data is given by 
\begin{equation} f_{ij}^{\lambda\odot\tx}=\lambda f^{\tx}_{ij},\end{equation} \begin{equation} f_{ij}^{\tx\bp\te}=f_{ij}^{\tx}+f_{ij}^{\te},\end{equation} and 
\begin{equation} \label{cech bracket} f_{ij}^{[\tx,\te]}=\xi(f_{ij}^{\te})-\eta(f_{ij}^{\tx}).\end{equation} 

We now show that these operations give the category $\L_{g_{ijk}}$ the structure of a (strict) Lie 2-algebra, or equivalently of a 2-term $L_{\infty}$-algebra. For a discussion of Lie 2-algebras and 2-term $L_{\infty}$-algebras,  we direct the reader to \cite{BC}.
\begin{definition} A \emph{2-term $L_{\infty}$ algebra} is a 2-term chain complex of vector spaces $\xymatrix@C5mm{ V_1 \ar[r]^{d} & V_0}$ equipped with: 
\begin{enumerate} 
\item an antisymmetric chain map $[\cdot,\cdot]:V\otimes V\to V$, 
\item an antisymmetric chain homotopy $J :V\otimes V\otimes V\to V$ from the chain map 
\begin{align} V\otimes V\otimes V &\to V \\ x\otimes y\otimes z &\mapsto [x,[y,z]] \notag\end{align} to the chain map \begin{align} V\otimes V\otimes V &\to V \\ x\otimes y\otimes z &\mapsto [[x,y],z]+[y,[x,z]], \notag\end{align} such that the following equation holds for each $x,y,z,w\in V$: 
\begin{align} \label{jacobiator identity}[x,J(y,z,w)]+J(x,[y,z],w)+J(x,z,[y,w])+[J(x,y,z),w] \\ +[z,J(x,y,w)] = J(x,y,[z,w])+J([x,y],z,w)\notag\\+[y,J(x,z,w)]+J(y,[x,z],w)+J(y,z,[x,w]). \notag\end{align} 
\end{enumerate}
\end{definition}
\begin{rem}  In the language of Lie 2-algebras, the equation (\ref{jacobiator identity}) is the \emph{Jacobiator identity}, which is an analogue of the pentagon identity for monoidal categories.
\end{rem} 

\begin{theorem} \label{non-conn L inf} Let $V_0=\textrm{Obj}(\L_{g_{ijk}})$, with vector space structure given by 
\begin{enumerate} \item $\lambda(\xi,\{f^{\tx}_{ij}\})=(\lambda\xi,\{\lambda f^{\tx}_{ij}\})$ for each $\lambda\in \R$ and $(\xi,\{f^{\tx}_{ij}\})\in \L_{g_{ijk}}$, and \item $(\xi,\{f^{\tx}_{ij}\})+(\eta,\{f^{\te}_{ij}\})=(\xi+\eta,\{f^{\tx}_{ij}+f^{\te}_{ij}\})$ for each $(\xi,\{f^{\tx}_{ij}\}),(\eta,\{f^{\te}_{ij}\})\in \L_{g_{ijk}}$.\end{enumerate}  Let $V_1=\{\{u_i:U_i\to i\R\}\}$ with vector space structure given by addition and scalar multiplication of functions.  Then $V=V_0\oplus V_1$ has the structure of a 2-term $L_{\infty}$-algebra with 
\begin{enumerate}
\item $d:V_1\to V_0$ given by 
$ \{u_i\}\mapsto (0,\{u_i-u_j\})$ for each $\{u_i\}\in V_1$
\item $[\cdot,\cdot]:L\otimes L\to L$ given by 
\begin{enumerate} \item $[(\xi,\{f^{\tx}_{ij}\}),(\eta,\{f^{\te}_{ij}\})]=([\xi,\eta],\{\xi(f_{ij}^{\te})-\eta(f_{ij}^{\tx})\})$ for each $(\xi,\{f^{\tx}_{ij}\}),(\te,\{f^{\eta}_{ij}\})\in V_0$ \item $[(\xi,\{f^{\tx}_{ij}\}),\{u_i\}]=\{\xi(u_i)\}=-[\{u_i\},(\xi\{f^{\tx}_{ij}\})]$ for each $(\xi\{f^{\tx}_{ij}\})\in V_0$, $\{u_i\}\in V_1$ \item $[\{u_i\},\{v_i\}]=0$ for each $\{u_i\},\{v_i\}\in V_1$.\end{enumerate} \item $J=0$.
\end{enumerate}  
\end{theorem}
\begin{proof}
To check that $[\cdot,\cdot]$ is a chain map, it is sufficient to check that for each $v_0=(\xi_i,\{f^{\tx}_{ij}\})\in V_0$ and $v_1=\{u_i\}$ and $v_1'=\{u_i'\}$ in $V_1$ that 
\begin{equation} \label{chain 1}d[v_0,v_1]=[v_0,dv_1], \end{equation} and 
\begin{equation}\label{chain 2}  [dv_1,v_1']=[v_1,dv_1'].\end{equation}  To verify (\ref{chain 1}), note that 
\begin{equation} d[(\xi,f_{ij}),\{u_i\}]=d\{\xi(u_i)\}=(0,\{\xi(u_i-u_j)\}),\end{equation} whereas 
\begin{equation} [(\xi,\{f_{ij}\}),d\{u_i\}]=[(\xi,\{f_{ij}\},(0,\{u_i-u_j\})]=(0,\{\xi(u_i-u_j)\}).\end{equation}  On the other hand, by inspection both the left and right-hand sides of equation (\ref{chain 2}) are zero.

To verify that $J$ is a chain homotopy from $[[x,y],z]\to [[x,y],z]+[y,[x,z]]$, there are two conditions to verify.  First, for each $u_0,v_0,w_0\in L_0$, we must check that 
\begin{equation} dJ(u_0,v_0,w_0)=-[[u_0,v_0],w_0]+[[u_0,w_0],v_0]+[u_0,[v_0,w_0]].\end{equation} The left-hand side is 0 by definition, whereas it is easy to verify that the right-hand side vanishes using the Jacobi identity for vector fields on $M$.  The second condition is that for $v_0,w_0\in V_0$ and $v_1\in V_1$, we have 
\begin{equation} J(dv_1,v_0,w_0)=-[[v_0,w_0],v_1]+[[v_0,v_1],w_0]+[v_0,[w_0,v_1]].\end{equation}  Again, the left-hand side is zero, and a simple computation shows that the right-hand side vanishes as well.

Finally, because $J=0$, the condition (\ref{jacobiator identity}) trivial.
\end{proof}

\section{\label{conn str and curving}Connective structures and curvings}
Given a Dixmier-Douady gerbe $\CC$ over a manifold $M$, Brylinski \cite[Br1] introduced the notions of a connective structure $\A$ on $\CC$ and a curving $K$ for $\A$.  We now explain how these structures emerge very naturally from the point of view of the infinitesimal symmetries of $\CC$ introduced in the previous sections.  For simplicity we begin by working in the Cech picture.
 
 Let  $\{\{U_i\},\{Q_i\},\{s_{ij}\}\}$ be such a collection of local trivializations with corresponding Cech data $\{g_{ijk}\}$, and let $\mathcal{L}_{g_{ijk}}$ be the corresonding category of lifts described in definition (\ref{Cech lift data}).  Since by definition the category $C^{\infty}(TM)$ has only trivial morphisms, we take a linear splitting of the projection $\pi$ to be a linear splitting of the sequence of vector spaces
 \begin{equation} \xymatrix{ 0 \ar[r] & \textrm{Ker}(\pi) \ar[r] & \textrm{Obj}(\mathcal{L}_{g_{ijk}}) \ar[r]^{\pi} & C^{\infty}(TM) \ar[r] & 0.}\end{equation}  Such a splitting can be obtained from a collection of 1-forms $\{A_{ij}\in T^*U_{ij}\}$ by setting $\{f_{ij}(\xi)=\iota_{\xi}A_{ij}\}$; the condition (\ref{lift delta}) is equivalent to the condition 
\begin{equation}\label{A cech} A_{jk}-A_{ik}+A_{ij}=d\log(g_{ijk}).\end{equation}  Such a collection of 1-forms is precisely the data needed to specify a connective structure on $\CC$ in the Cech picture \cite{H}.  Thus, given such a collection of 1-forms, we may define the \emph{horizontal lift} $\tx^h$ of a vector field $\xi$ to be given by the Cech data $\{f_{ij}^{\xi^h}=\iota_{\xi}A_{ij}\}$.

Next, let $\xi,\eta$ be a pair of vector fields on $M$.  Note that in general we have no natural way to compare the lifts $\widehat{[\xi,\eta]}^h$ and $[\tx^h,\te^h]$, whose Cech data are given respectively by 
\begin{equation}\label{bracket 1}f^{\widehat{[\xi,\eta]}^h}_{ij}=\iota_{[\xi,\eta]}A_{ij}\end{equation} and 
\begin{equation}\label{bracket 2}f^{[\tx^h,\te^h]}_{ij}=\xi\cdot(\iota_{\eta}A_{ij})-\eta\cdot(\iota_{\xi}A_{ij}).\end{equation}  
On the other hand, suppose that we are given the additional structure of a \emph{curving}: in the Cech picture this is a collection of 2-forms $\{B_i\in i\Omega^2(U_i)\}$ such that on overlaps we have 
\begin{equation}\label{cech curving} B_j-B_i=dA_{ij}.\end{equation}  A simple calculation shows that
\begin{equation} f^{[\tx^h,\te^h]}_{ij}-f^{\widehat{[\xi,\eta]}^h}_{ij}=\iota_{\eta}\iota_{\xi}dA_{ij},\end{equation} and it therefore follows that
\begin{equation}\label{bracket isomorphism} \{u_i=\iota_{\eta}\iota_{\xi}B_i\} \end{equation} defines a morphism in the Cech picture between the two lifts (\ref{bracket 1}) and (\ref{bracket 2}).

 With the above discussion as motivation, we recall the definition of a connective structure on $\CC$ in the language of sheaves of categories.  The following is a restatement of definition (5.31) from \cite{Br1} in the language of \S\ref{gerbes}. \begin{definition} \label{def conn str} A \emph{connective structure} on $\CC$ is a morphism of gerbes \begin{equation}\A:\CC\to \Bone\end{equation} intertwining the homomorphism 
 \begin{equation}\label{conn str int}-d\log:\underline{\T}_M\to \underline{i\Omega}^1_M.\end{equation} \end{definition} 
 \begin{notation}  For each object $Q$ of $\CC$ over $U$, $\A(Q)$ is a sheaf over $U$.  We often write $\mu\in A(Q)$ to denote that $\mu$ is a global section of $\A(Q)$; we call such a $\mu$ a \emph{connection} on $Q$.
 \end{notation}
 \begin{notation}  Given an inclusion of open sets $i:V\hookrightarrow U$, $P\in \CC(U)$ and $\mu\in \A(P)$, we obtain an element of $\A(i^*P)$ in two steps (see definition (\ref{stack morphism})): first we restrict $\mu$ to $V$ to obtain an element of the $\ul{i\Omega}^1_V$ torsor $i^*\A(P)$.  Then using the natural isomorphism 
 \begin{equation} i^*\A(P)\tocong \A(i^*P) \end{equation} (which is given as part of the data of the 1-morphism $\A$), we obtain an element of $\A(i^*P)$.  Using a slight abuse of notation,  we will denote this element either by $\mu|_V$ or $i^*\mu$.
 \end{notation}

 \begin{rem} \label{other band conn str}  Definition (\ref{def conn str}) can easily be generalized to any gerbe with band $\ul{A}_M$ for $A$ an abelian Lie group; the sheaf $\ul{i\Omega}^1_M$ is replaced by the sheaf of $\textrm{Lie}(A)$-valued 1-forms.  We will later use the case $A=i\R$, where we make the identification $\textrm{Lie}(i\R):=i\R$.
\end{rem}

 \begin{ex}  Let $\BT$ denote the trivial DD gerbe over $M$ whose objects are $\ul{\T}_U$-torsors for open sets $U\subset M$.   We define the \emph{trivial connective structure} $\A^0_{\ul{\T}_M}$ on $\BT$ as the 1-morphism which assigns to each torsor its sheaf of connections.  We can rephrase this definition using the discussion of example (\ref{connection torsor}) as well as proposition (\ref{ass gerbe coh}).
Namely, the trivial connective structure on $\BT$ can be define as the 1-morphism of gerbes
\begin{equation} -d\log[\cdot]:\BT\to \Bone. \end{equation}

\end{ex}

We have the following definition of the horizontal lift in sheaf language.
\begin{definition} \label{horizontal} For each vector field $\xi$ on $M$, the \emph{horizontal lift} $\tx^h$ of $\xi$ to $\CC$ is the composition of 1-morphisms
\begin{equation}\xymatrix{\CC\ar[r]^-{\A} &\Bone \ar[r]^{(\iota_{\xi})[\cdot]} &\BiR\\ P \ar@{|->}[r] & \A(P) \ar@{|->}[r] & \iota_{\xi}[\A(P)].}\end{equation} 
\end{definition} In order to explain the connection between this definition and the definition of the horizontal lift in the Cech picture, let us return to the situation that we have a collection of local trivializations for $\CC$.   To express $\A$ in terms of Cech data, let us choose connections $\mu_i\in\A(Q_i)$.  If we define 1-forms $A_{ij}$ on overlaps by the equation 
\begin{equation}\label{conn str cech} \mu_j=(s_{ij})_*\mu_i+A_{ij},\end{equation} then (\ref{def conn str}) implies that on triple overlaps these 1-forms satisfy equation (\ref{A cech}).  Given a vector field $\xi$ on $M$, let $\tx^h=\iota_{\xi}d[\A]$ be the horizontal lift given in definition (\ref{horizontal}).  Then we obtain sections 
\begin{equation} r_i=\iota_{\xi}[\mu_i]\in F_{\tx^h}(Q_i),\end{equation} and a simple calculation shows that the corresponding Cech data is given by $\{\iota_{\xi}A_{ij}\}$.

Let us also recall Brylinski's definition of a \emph{curving} for $\A$.
\begin{definition}\label{def curving}\emph{(\cite{Br1}, def. 5.3.7)}  Let $\CC$ be a DD gerbe over a manifold $M$ equipped with a connective structure $\A$.  A \emph{curving} of the connective structure is a function which assigns to each object $P\in\CC(U)$ and each local section $\mu\in\A(P)(V)$ (for $V\subset U$ an open subset)  a 2-form $K(\mu)\in i\Omega^2(V)$, such that the following properties are satisfied:
\begin{enumerate}[(1)]
\item For each inclusion of open sets $i:W\hookrightarrow V$, each $P\in \CC(U)$ and $\mu\in \A(P)(V)$, we have 
\begin{equation} K(i^*\mu)=K(\mu)|_W\end{equation}
\item For each pair of objects $P,Q\in\CC(U)$, each morphism $\psi\in\ul{\Hom}(P,Q)(V)$, and each $\mu\in \A(P)(V)$, we have
\begin{equation} K(\psi_*\mu)=K(\mu)\end{equation}
\item For each object $P\in\CC(U)$, each $\mu\in \A(P)(V)$ and each 1-form $\alpha\in i\Omega^1(V)$, we have
\begin{equation} K(\mu+\alpha)=K(\mu)+d\alpha\end{equation}
\end{enumerate}
\end{definition}
\begin{rem}\label{curving restatement} Given an $\ul{i\Omega}^1_U$-torsor $A$, a map from $A$ to $\ul{i\Omega}^2_U$ intertwining $d:\ul{i\Omega}^1_U\to \ul{i\Omega}^2_U$ is equivalent to a isomorphism of $\ul{i\Omega}^2_U$-torsors $d[A]\tocong \ul{i\Omega}^2_U$.  Unravelling definition (\ref{2-morphism}), it follows that a curving for $\A$ is equivalent to a 2-morphism \begin{equation} d[\A]\tocong \ul{i\Omega}^2_M,\end{equation} where $\ul{i\Omega}^2_M$ denotes the trivial 1-morphism from $\CC$ to $\textrm{B}(\ul{i\Omega}^2_M)$ sending every object $Q\in \CC(U)$ to the trivial $\ul{i\Omega}^2_U$-torsor $\ul{i\Omega}^2_U$.
\end{rem}

\begin{ex} Let $\A^0_{\T}$ denote the trivial connective structure on the trivial DD gerbe $\ul{\textrm{B}\T}_M$.  Then the \emph{trivial curving} for $\A^0_{\T}$ assigns to each connection its curvature 2-form.
\end{ex}

To see how definition (\ref{def curving}) appears in the Cech picture, define 2-forms $B_i\in \ul{i\Omega}^2(U_i)$ by 
\begin{equation} B_i=K(\mu_i).\end{equation} It then follows from definition (\ref{def curving}) that on overlaps these satisfy equation 
(\ref{cech curving}). 

\begin{rem} Although we will not do this, it is possible to construct the natural isomorphism (\ref{bracket isomorphism}) in the sheaf language.
\end{rem}

Next, we recall the definition of the curvature 3-form associated to a curving.
\begin{definition}  Given a connective structure $\A$ on a gerbe $\CC$ and curving $K$ on a DD gerbe over a manifold $M$, the \emph{curvature} of $(\CC,\A,K)$ is the 3-form $C$ on $M$ defined locally by $C=dK(\mu)$, where $\mu$ is a section of $\A(Q)$ for $Q\in\CC(U)$.
\end{definition}

Recall that if $\Theta$ is a connection on a principal $\T$-bundle $E\to M$, then the curvature 2-form of $\Theta$ measures the failure of the splitting of the sequence (\ref{inf extension}) determined by $\Theta$ to be a homomorphism of Lie algebras.  To understand the analogous situation for gerbes, we need the notion of a homomorphism of 2-term $L_{\infty}$-algebras.

\begin{definition}\label{L-infty hom} \emph{(\cite{BC} def. 34)}  Let $V$ and $V'$ be 2-term $L_{\infty}$-algebras.  An \emph{$L_{\infty}$-homomorphism} $\phi:V\to V'$ consists of 
\begin{enumerate} \item a degree 0 chain map $\phi:V\to V'$,
\item an antisymmetric degree 1 chain map $\phi_2:V\otimes V\to V'$
\end{enumerate} such that the following equations hold:
\begin{enumerate}[(1)] \item $d(\phi_2(x,y))=\phi[x,y]-[\phi(x),\phi(y)]$ for all $x,y\in V$, and 
\item $[\phi_2(x,y),\phi(z)]+\phi_2([x,y],z)+\phi(J(x,y,z))=J(\phi(x),\phi(y),\phi(z))+[\phi(x),\phi_2(y,z)]+[\phi_2(x,z),\phi(y)]+\phi_2(x,[y,z])+\phi_2([x,z],y)$ for all $x,y,z\in V_0$.
\end{enumerate}
\end{definition}


We make $C^{\infty}(TM)$ into a 2-term $L_{\infty}$-algebra by setting $V_0=C^{\infty}(TM)$, $V_1=0$,  with bracket given by the Lie bracket.  We then can extend the linear map from $\textrm{Obj}(\L_{g_{ijk}})\to C^{\infty}(TM)$ in an obvious way to an $L_{\infty}$-homomorphism.
In the language of defintion (9.3) from \cite{R2}, we obtain a \emph{strict central extension of $L_{\infty}$-algebras} 
\begin{equation}\label{Linfty extension} \L_{g_{ijk}}(0)\to \L_{\{g_{ijk}\}}\to C^{\infty}(TM). \end{equation} One should compare this extension to (\ref{inf extension}), noting that the category $\L_{g_{ijk}}$ is naturally isomorphic to the category of principal $i\R$-bundles over $M$.

  Let us try to construct a splitting of the extension (\ref{Linfty extension}), i.e. an $L_{\infty}$-homomorphism $\phi:C^{\infty}(TM)\to \L_{g_{ijk}}$ such that $\pi\circ\phi$ is the identity on $C^{\infty}(TM)$.\footnote{More generally we would ask only for a chain homotopy from $\pi\circ \phi$ to the identity map.}  A degree zero chain map $\phi:C^{\infty}\to \L_{g_{ijk}}$ is simply a linear splitting of the projection $\textrm{Obj}(\L_{g_{ijk}})\to C^{\infty}(TM)$, which we have already seen is essentially equivalent to a connective structure $\{A_{ij}\}$ on $\CC$ in the Cech picture.  Similarly, it follows from the discussion at the beginning of this section that an antisymmetric degree 1 chain map $\phi_2:C^{\infty}(TM)\otimes C^{\infty}(TM)\to \L_{g_{ijk}}$ is essentially equivalent to a curving $\{B_i\}$.

Finally, since the Jacobiator maps $J$ vanish for both $C^{\infty}(M)$ and $L$, it follows from definition (\ref{L-infty hom}) that $(\phi,\phi_2)$ defines an $L_{\infty}$-homomorphism if and only if for each $\xi,\eta,\tau\in C^{\infty}(M)$ we have 
\begin{align} &[\phi_2(\xi,\eta),\phi(\tau)]+\phi_2([\xi,\eta],\tau)-[\phi(\xi),\phi_2(\eta,\tau)]\\&-[\phi_2(\xi,\tau),\phi(\eta)]-\phi_2(\xi,[\eta,\tau])=0.\notag
\end{align}  However, the left hand side of this equation is the element $\{u_i\}\in L^1$ given by 
\begin{align} u_i &= \tau\cdot B_i(\xi,\eta)-B_i([\xi,\eta],\tau)+\xi\cdot B_i(\eta,\tau)\\&-\eta\cdot   B_i(\xi,\tau)-B_i([\eta,\tau],\xi)+B([\xi,\tau],\eta)\notag \\ &=  C(\xi,\eta,\tau), \notag\end{align} where $C=dB_i$ is the curvature 3-form.  We therefore see that the curvature (or rather its de Rham cohomology class) is an obstruction to the existence of a splitting of (\ref{Linfty extension}) as Lie 2-algebras.

\section{\label{conn lifts}Infinitesimal symmetries of gerbes with connective structure}

In this section we will study symmetries of a gerbe $\CC$ equipped with a connective structure $\A$.  We begin by extending definition (\ref{def symmetry gerbe}) to take the connective structure into account.  The motivation for this definition can best be understood by studying the relationship to families of symmetries of gerbes with connective structure, as discussed in appendix $A$.  In \S\ref{local flows conn} we generalize the results in \S\ref{local flows} to establish an equivalence between a local version of the category of 1-parameter families of connective lifts and the category of infinitesimal connective lifts defined below.

If $(P,\Theta)$ and $(P',\Theta')$ are principal $\T$-bundles with connection over a manifold $M$, then given an isomorphism $\psi:P\to P'$ of the underlying principal bundles, we may ask whether $\psi$ is \emph{compatible} with the connections $\Theta,\Theta'$: concretely, $\Theta$ and $\Theta'$ are differential forms and we may compare $\Theta$ with $\psi^*\Theta'$.  By contrast, given Dixmier-Douady gerbes $(\CC,
A)$, $(\CC',\A')$ with connective structures and a 1-morphism $\Psi:\CC\to \CC'$, it is no longer the right question to ask whether $\A$ is \emph{equal} to $\Psi^*\A$.  Rather, we must extend $\Psi$ to a \emph{connective isomorphism} by specifying extra data in the form of an isomorphism $\A\tocong \Psi^*\A'$.  Recall from definition (\ref{def conn str}) that the connective structure $\A'$ is a 1-morphism of gerbes 
\begin{equation} \CC'\to \textrm{B}(\ul{i\Omega}^1_M); \end{equation} by definition the pull-back $\Psi^*(\A')$ is the composition 
\begin{equation} \A'\circ \Psi:\CC\to \textrm{B}(\ul{i\Omega}^1_M).\end{equation}  In this language an extension of $\Phi$ to a connective isomorphism is given by a 2-morphism 
\begin{equation} \A\toocong \A'\circ\Psi.\end{equation}  The analogous situation for infinitesimal symmetries is given in the following definition.
\begin{definition}\label{def conn lift}  Let $\CC$ be a DD gerbe over a manifold $M$ with connective structure $\A$.  Given $\xi\in C^{\infty}(TM)$, a \emph{connective lift} $\cx$ of $\xi$ to $\CC$ is a pair $\cx=(F_{\cx},\Theta_{\cx})$, where $F_{\cx}$ is a lift of $\xi$ to $\CC$ in the sense of definition (\ref{def symmetry gerbe}), and $\Theta_{\tx}:\pounds_{-\xi}[\A]\toocong \A^0_{\ul{i\R}}F_{\tx}$ is a 2-morphism.
\end{definition}
More explicitly, for every object $Q\in\CC(U)$ and every local section $\mu$ of $\A(Q)$, a connective lift determines a local connection $\Theta_{\cx}(\mu)$ on $F_{\cx}(Q)$.  Furthermore, this assignment is natural in $Q$ in the sense described in (\ref{2-morphism}), and for every 1-form $\alpha\in i\Omega^1(U)$ we have \begin{equation} \Theta_{\cx}(\mu+\alpha)=\Theta_{\cx}(\mu)-\pounds_{\xi}\alpha.\end{equation}

\begin{ex} \label{triv conn lift} Given a vector field $\xi$ on $M$, recall from example (\ref{trivial inf lift}) the trivial lift $\tx_0$ of $\tx$ to the trivial gerbe $\textrm{B}(\ul{\T}_M)$.  This lift assigns to every $\ul{\T}_U$-torsor $P$ the associated $\ul{i\R}_U$-torsor
\begin{equation} -\iota_{\xi}d\log[P].\end{equation} We wish to extend this to define a trival connective lift of $\xi$ to $\textrm{B}(\ul{\T}_M)$ equipped with the trivial connective structure $\A^0_{\ul{\T}_M}.$  Suppose $\Theta$ is a connection on a $\ul{\T}_M$-torsor $P$, then we will define a connection $-\pounds_{\xi}(\Theta)$ on $-\iota_{\xi}d\log[P]$ as follows.  Each local section of $-\iota_{\xi}d\log[P]$ is of the form 
\begin{equation} -\iota_{\xi}d\log[\sigma]+f,\end{equation} where $\sigma$ is a local section of $P$ and $f$ is a local $i\R$-valued function.  $-\pounds_{\xi}(\Theta)$ is then defined by the formula
\begin{equation} \label{der conn} A_{-\pounds_{\xi}(\Theta)}(-\iota_{\xi}d\log[\sigma]+f)=-\pounds_{\xi}A_{\Theta}(\sigma)+df.\end{equation}  Using the Cartan formula $\pounds_{\xi}=d\iota_{\xi}+\iota_{\xi}d$, one easily checks that equation (\ref{der conn}) is well-defined.

\begin{definition} The \emph{trivial connective lift} $\cx_0$ of $\xi$ to $(\textrm{B}(\ul{\T}_M),\A^0_{\ul{\T}_M})$ is the pair $(\tx_0,\Theta_{\cx_0})$, where $\tx_0$ is the trivial lift defined in (\ref{trivial inf lift}), and for each  $\ul{i\R}_U$-torsor $P$ we have 
\begin{align} (\Theta_{\xi})_P:\pounds_{-\xi}[\A^0_{\ul{\T}_M}(P)] & \to \A^0_{\ul{i\R}_M}\circ \tx_0(P) \\ \pounds_{-\xi}[\Theta] & \mapsto -\pounds_{\xi}(\Theta).\notag \end{align}
\end{definition}
\end{ex}

We next wish to define the category of connective lifts. 
\begin{definition} \label{conn equiv}Given connective lifts $\cx=(F_{\cx},\Theta_{\cx})$ and $\cx'=(F_{\cx'},\Theta_{\cx'})$ of a vector field $\xi$, an equivalence $T$ between the underlying (non-connective) lifts $F_{\cx},F_{\cx'}$ is an \emph{equivalence of connective lifts} if the following diagram commutes:
\begin{equation}\label{conn equiv cond} \xymatrix@C15mm{\A \ar@{=>}[d]_{\Theta_{\cx}} \ar@{=>}[dr]^{\Theta_{\cx'}} & \\ \A_0\circ F_{\cx} \ar@{=>}[r]_{1_{\A_0}* T}  & \A_0\circ F_{\cx'}.}\end{equation}
\end{definition}

\begin{rem} In more concrete terms, the condition (\ref{conn equiv cond}) says that for each $Q\in \CC(U)$, $\mu\in \A(Q)$ and each $\sigma\in F_{\cx}(Q)$, we have 
\begin{equation}\label{conn equiv crit} A_{\Theta_{\cx'(\mu)}}(T_Q(\sigma))=A_{\Theta_{\cx}(\mu)}(\sigma).\end{equation}  In particular, suppose $\cx=\cx'$, then by theorem (\ref{lift class}) $T$ is given by 
\begin{align} T_Q:F_{\cx}(Q) &\to  F_{\cx}(Q) \\ \sigma & \mapsto  \sigma +f|_U,\notag\end{align} for each $Q\in \CC(U)$, where $f:M\to i\R$ is a globally defined function.  Equation (\ref{conn equiv crit}) therefore implies that $T$ gives an automorphism of $\cx$ if and only if $f$ is locally constant.
\end{rem}

\begin{definition} $\check{\L}_{\CC}$ is the category whose objects are connective lifts and whose morphisms are connective equivalences of lifts.
\end{definition}
\begin{notation} For each vector field $\xi$ on $M$ $\check{L}_{\CC}(\xi)$ denotes the subcategory of connective lifts of $\xi$ to $\CC$.
\end{notation}

\begin{rem} \label{conn algebra cech} We remark that the construction of the functors $\bp,\lambda\odot$ and $[\cdot,\cdot]$ given in proposition (\ref{Lie algebra structure}) generalizes to the connective setting.  In particular, the bracket functor is extended explicitly as follows.  Let $\cx=(F_{\cx},\Theta_{\cx})$ and $\ce=(F_{\ce},\Theta_{\cx})$ be connective lifts of the vector fields $\xi,\eta$, respectively.  Then for each local section $\mu$ of $\A(P)$, $ \Theta_{[\cx,\ce]}(\mu)$ is the connection on $F_{[\cx,\ce]}(P)$ given by
\begin{equation} A_{{\Theta}_{[\tx,\te]}}:\xi[\sigma_1]\boxminus \eta[\sigma_2]\mapsto \pounds_{\xi}A_{{\Theta}_{\te}(\mu)}(\sigma_1)-\pounds_{\eta}A_{{\Theta}_{\tx}(\mu)}(\sigma_2) \end{equation} where $\sigma_1$ and $\sigma_2$ are local sections of $F_{\cx}(P)$ and $F_{\ce}$, respectively.  We will explain how the bracket appears in the Cech picture below.
\end{rem}

\subsection{\label{conn lifts cech}Connective lifts in the Cech picture}

Let us describe the category of connective lifts in the Cech picture. Let \begin{equation}\{\{U_i\},\{Q_i\},\{s_{ij}\},\{\mu_i\}\}\end{equation} be a collection of local trivializations of $(\CC,\A)$.  Recall that if we define 1-forms $A_{ij}$ on $U_{ij}$ by 
\begin{equation} \mu_j=(s_{ij})_*\mu_i-A_{ij},\end{equation} then these satisfy
\begin{equation} A_{jk}-A_{ik}+A_{ij}=d\log(g_{ijk})\end{equation} on triple overlaps.

Next, let $\check{\xi}=(F_{\cx},\Theta_{\cx})$ be a connective lift of $\xi$ to $(\CC,\A)$, and suppose we have chosen local sections $r_i\in F_{\check{\xi}}(Q_i)$.  Then for each $i$ we obtain a connection $\nu_i=\Theta_{\check{\xi}}(\mu_i)$ on $F_{\check{\xi}}(Q_i)$, and we can define 1-forms
\begin{equation} a_i=A_{\nu_i}(r_i)\in i\Omega^1_{U_i}.\end{equation} 
\begin{proposition} On overlaps $U_{ij}$, we have 
\begin{equation}\label{conn lift cech} a_j-a_i=df_{ij}+\pounds_{\xi}A_{ij}. \end{equation}
\end{proposition}
\begin{proof}
The naturality of of $\Theta_{\check{\xi}}$ implies that the diagram 
\begin{equation} \xymatrix{\A(Q_i) \ar[r]^{(\Theta_{\check{\xi}})_{Q_i}} \ar[d]_{\A(s_{ij})} & \A_0(F_{\check{\xi}}Q_i) \ar[d]^{\A_0(F_{\check{\xi}}(s_{ij}))} \\ \A(Q_j) \ar[r]_{(\Theta_{\check{\xi}})_{Q_j}} & \A_0(F_{\check{\xi}}Q_j) }\end{equation} commutes on each overlap $U_{ij}$; we therefore have 
\begin{equation} \A_0(F_{\cx}(s_{ij}))\Theta_{\cx}(\mu_i)=\Theta_{\cx}(\A(s_{ij})(\mu_i)).\end{equation}  The left-hand side of this equation is equal to 
\begin{equation} F_{\cx}(s_{ij})_*\nu_i,\end{equation} whereas the right-hand side  is 
\begin{equation} \Theta_{\cx}(\mu_j-A_{ij})=\nu_j+\pounds_{\xi}A_{ij}.\end{equation}  We therefore have 
\begin{align} a_j  = & A_{\nu_j}(r_j)  \\  = & A_{\nu_j}(F_{\cx}(s_{ij})r_i+f_{ij}) \notag\\  = & A_{\nu_j}(F_{\cx}(s_{ij})r_i)+df_{ij} \notag\\ = & A_{F_{\cx}(s_{ij})_*\nu_i-\pounds_{\xi}A_{ij}}(F_{\cx}(s_{ij})r_i)+df_{ij} \notag \\  = & A_{F_{\cx}(s_{ij})_*\nu_i}(F_{\cx}(s_{ij})r_i)-\pounds_{\xi}A_{ij}+df_{ij} \notag \\  = & A_{\nu_i}(r_i)+df_{ij}-\pounds_{\xi}A_{ij} \notag \\ 
 = & a_i+df_{ij}-\pounds_{\xi}A_{ij}.\end{align}

\end{proof} 

Next, let $\tau:(F_{\cx},\Theta_{\cx})\to (F_{\cx'},\Theta_{\cx'})$ be an equivalence of connective lifts, and let $\{r_i'\}$ be a collection of local trivializations of $F_{\cx'}$, with corresponding Cech data $\{a_i'\}$ and $\{f_{ij}'\}$.  Recall from \S\ref{Cech}, that if we define local $i\R$-valued functions $u_i$ by 
\begin{equation} r_i'=\tau_{Q_i}(r_i)+u_i,\end{equation} then these functions satisfy 
\begin{equation} f'_{ij}-f_{ij}=u_i-u_j.\end{equation}  Unravelling definition (\ref{conn equiv}), we see that 
\begin{equation} \nu_i'=\Theta_{\cx'}(\mu_i)=(\tau_{Q_i})_*\nu_i.\end{equation}  We therefore have 
\begin{align} a_i'  = & A_{\nu_i'}(r_i') \\  = & A_{(\tau_{Q_i})_*\nu_i}(\tau_{Q_i}(r_i)+u_i) \notag\\  = & A_{(\tau_{Q_i})_*\nu_i}(\tau_{Q_i}(r_i))+du_i\notag \\ = & A_{\nu_i}(r_i)+du_i \notag\\  = & a_i+du_i,\end{align} and therefore  $a'_i-a_i=du_i$ on overlaps.

\begin{definition} \label{cat conn lifts cech} Let $\{g_{ijk}:U_{ijk}\to \T\}$ be a Cech cocycle on a manifold $M$ with respect to a good open cover $\{U_i\}$, and let $\{A_{ij}\in i\Omega^1(U_{ij})\}$ be a collection of 1-forms satisfying equation (\ref{A cech}) on triple overlaps.  The category $\L_{(g_{ijk},A_{ij})}$ has as objects the set of triples $(\xi,\{f_{ij}\},\{a_i\})$, where the functions $f_{ij}$ satisfy equation (\ref{lift delta}), and $a_i\in i\Omega(U_i)$ is a collection of 1-forms satisfying equation (\ref{conn lift cech}).  A morphism from $(\xi,\{f^{\cx}_{ij}\},\{a^{\cx}_i\})\to (\xi,\{f^{\cx'}_{ij}\},\{a^{\cx'}_i\})$ is a collection of local functions $\{u_i:U_i\to i\R\}$ such that on each overlap $U_{ij}$ we have $f^{\cx'}_{ij}-f^{\cx}_{ij}=u_i-u_j$, and on each $U_i$ we have  $a_i^{\cx'}-a_i^{\cx}=du_i$.
\end{definition}

We can now generalize theorem (\ref{cech lift class}) to the connective case.
\begin{theorem}\label{cech conn lift class}  Let $\{g_{ijk}\}$ and $\{A_{ij}\}$ be as in definition (\ref{cat conn lifts cech}), and let 
\begin{equation} \pi:\L_{(g_{ijk},A_{ij})}\to \L_{g_{ijk}}\end{equation} be the obvious forgetful functor.  Then for each vector field $\xi$ on $M$ and each $\{f^{\cx}_{ij}\}\in \L_{g_{ijk}}(\xi)$, the set $\pi^{-1}(\{f^{\cx}_{ij}\})$ of connective lifts extending $\{f^{\cx}_{ij}\}$ is a torsor for global 1-forms on $M$, where the action of $\alpha\in i\Omega^1(M)$ on $\cx=(\xi,\{f^{\cx}_{ij}\},\{a_i\})$ is given by 
\begin{equation} \label{new conn lift} \cx+\alpha=(\xi,\{f^{\cx}_{ij}\},\{a_i+\alpha|_{U_i}\}).\end{equation}
\end{theorem}
\begin{proof}  Let $\tx=(\xi,\{f^{\tx}_{ij}\})$ be an object of  $\L_{g_{ijk}}(\xi)$, and let 
\begin{equation} \beta_{ij}=df^{\tx}_{ij}-\pounds_{\xi}A_{ij}.\end{equation}  Then the Cech coboundary of $\{\beta_{ij}\}$ is given by 
\begin{align} (\partial\beta)_{ijk}&=d(\partial f)_{ijk}-\pounds_{\xi}(\partial A)_{ijk} \\ &
 = d(\iota_{\xi}d\log(g_{ijk}))-\pounds_{\xi}(d\log(g_{ijk})) =0 \notag.
 \end{align}  Using partitions of unity we may therefore find $a_i\in i\Omega^1_{U_i}$ such that \begin{equation} a_j-a_i=df^{\tx}_{ij}-\pounds_{\xi}A_{ij}\end{equation} on overlaps.  Thus there exists an extension $\cx=(\xi,\{f^{\tx}_{ij}\},\{a_i\})$ of $(\xi,\{f^{\tx}_{ij}\})$ to a connective lift.  

Given a global 1-form $\alpha$, the formula (\ref{new conn lift}) clearly defines a new connective lift $\cx+\alpha$.  On the other hand, given two connective lifts $(\xi,\{f^{\cx}_{ij}\},\{a^{\cx}_i\})$, $(\xi,\{f^{\cx'}_{ij}\},\{a^{\cx'}_i\})$, if we define $\alpha_i=a^{\cx'}_i-a^{\cx}_i$, then we must have $\alpha_i|_{U_{ij}}=\alpha_j|_{U_{ij}}$, so that there is a global 1-form $\alpha$ with $\alpha_i=\alpha|_{U_i}$.  
\end{proof}

We can also generalize theorem (\ref{cech equiv}), the proof of which is in appendix B.
\begin{theorem} \label{cech equiv. conn} Let $\{\{U_i\},\{s_{ij}\},\{\mu_i\}\}$ be a collection of local trivializations of $(\CC,\A)$ with corresponding Cech data $\{g_{ijk},A_{ij}\}$.  Then there is an equivalence of categories $\L_{(g_{ijk},A_{ij})}\to \L_{(\CC,\A)}.$
\end{theorem}

For each vector field $\xi$ on $M$, let $\L_{(\CC,\A)}(\xi)$ denote the subcategory consisting of connective lifts of $\xi$ to $\CC$.  Also, let $\pi:\L_{(\CC,\A)}(\xi)\to \L_{\CC}(\xi)$ denote the obvious forgetful functor.

\begin{corollary}\label{conn lift class}  For each vector field $\xi$ on $M$ and each $\tx\in \L_{\CC}(\xi)$, the set $\pi^{-1}(\tx)$ of connective lifts extending $\xi$ is a torsor for the group of global 1-forms on $M$.  For each $\cx\in \check{\L}_{(\CC,\A)}(\xi)$, the automorphism group of $\cx$ is canonically isomorphic to the group of locally constant $i\R$-valued functions on $M$.
\end{corollary}

To end this section, we will generalize the construction in theorem (\ref{non-conn L inf}) to give $\L_{(g_{ijk},A_{ij})}$ the structure of a 2-term $L_{\infty}$ algebra.  First,  one can  checks that the bracket functor in the Cech picture is given by 
\begin{align}\label{conn bracket cech}& [(\xi,\{f^{\tx}_{ij}\},\{a^{\tx}_i\}), (\eta,\{f^{\te}_{ij}\},\{a^{\te}_i\})] \\ = &([\xi,\eta],\{\xi(f^{\te}_{ij})-\eta(f^{\tx}_{ij})\},\{\pounds_{\xi}a^{\te}_i-\pounds_{\eta}a^{\tx}_i\}).\notag\end{align} 

We then have the following proposition, whose proof is an easy extension of the proof of (\ref{non-conn L inf}).
\begin{proposition} Let $\{U_i\}$ be an open cover of a manifold $M$, and let $\{g_{ijk}\}$ be a $\T$-valued Cech cocycle.  Then there is a 2-term $L_{\infty}$ algebra $\check{V}$ with $\check{V}_0=\textrm{Ob}(\L_{(g_{ijk},A_{ij})})$ and $\check{V}_1=\{\{u_i:U_i\to i\R\}\}$, and such that
\begin{enumerate} 
\item $d:\check{V}_1\to\check{V}_0$ is given by 
\begin{equation} \{u_i\}\mapsto (0,\{u_i-u_j\},\{du_i\}), \end{equation} for each $\{u_i\}\in \check{V}_1$.
\item $[\cdot,\cdot]$ is given on elements of $\check{V}_0$ by equation (\ref{conn bracket cech}).  Given $\cx=(\xi,\{f^{\cx}_{ij}\},\{a^{\cx}_{i}\})\in \check{V}_0$ and $u=\{u_i\}\in \check{V}_1$, we define 
\begin{equation} [\cx,u]=\{\xi(u_i)\}=-[u,\cx].\end{equation} 
\item $J=0$.
\end{enumerate}
\end{proposition}

\section{\label{local flows conn}Local flows and infinitesimal connective symmetries}

Given a local flow $\Phi$ and vector field $\xi$ on $M$ as described in \S\ref{local flows}, we now wish to generalize theorem (\ref{diff equiv})  establishing a local equivalence between the category of connective lifts of $\xi$ to $\CC$ and the category of 1-parameter connective lifts of $\Phi$ to $\CC$.  Since every gerbe with connective structure is locally isomorphic to the trivial gerbe with the trivial connective structure, it is enough to consider this case.  We work with the simplicial manifold $U^{\bullet}$ introduced in \S\ref{local flows}.  As discussed in appendix A, we work with sheaves of \emph{relative 1-forms} over the manifolds $U^i$, with respect to the projections $U^i\to \R^i$. These can be described as follows: for each $i\geq 0$ and each $\vec{t}\in \R^i$, let $U_{\vec{t}}\subset U^i$ be the submanifold $U\times \{\vec{t}\}$.  Then $\ul{i\Omega}^1_{U^i,rel}$ is the quotient of the sheaf $\ul{i\Omega}^1_{U^i}$ by the subsheaf of 1-forms whose restriction to $U_{\vec{t}}$ vanishes for each $\vec{t}\in \R^i$.

\begin{terminology}  If $\alpha$ is a 1-form whose restriction to each $U_{\vec{t}}$ is zero, we say that $\alpha$ \emph{vanishes in the $M$-direction}.
\end{terminology}
 Given a principal $\T$-bundle $E\to U^i$, we have the notion of a \emph{relative connection} on $E$, which we view as a sheaf homomorphism from $\ul{E}\to \ul{i\Omega}^1_{U^i,rel}$ intertwining the homomorphism 
\begin{equation} \xymatrix{ \ul{\T}_{U^i} \ar[r]^-{d\log} & \ul{i\Omega}^1_{U^i} \ar[r] & \ul{i\Omega}^1_{U^i,rel},}\end{equation} where the second map is the quotient. The set $\A_{rel}(E)$ of all such such relative connections  is a torsor for the group of  relative 1-forms on $U^i$. Furthermore we have a quotient map $\A(E)\to \A_{rel}(E)$ intertwining the quotient map $\ul{i\Omega}^1_{U^i}\to \ul{i\Omega}^1_{U^i,rel}.$

\begin{definition}  Given $\alpha\in i\Omega^1_{M^i,rel}$, we define
\begin{equation}\label{delta a}\delta \alpha=p_0^*\alpha-p_1^*\alpha+\cdots (-1)^{i+1} p_{i+1}^*\alpha\in 
:i\Omega^1_{U^i,rel}\end{equation} 
\end{definition}  
\begin{rem}  Given $\alpha\in i\Omega^1_{M^i}$ vanishing in the $M$-direction, a simple calculation shows that 
\begin{equation} p_0^*\alpha-p_1^*\alpha+\cdots+(-1)^{i+1}p_{i+1}^*\alpha\end{equation} also vanishes in the $M$-direction.  Therefore (\ref{delta a}) is well-defined.
\end{rem}
\begin{rem} Given a $\ul{\T}_{U^i}$-torsor $S$, recall from $\S$\ref{local flows} that we obtain a $\ul{\T}_{U^{i+1}}$-torsor $\delta S$.  If $\Theta$ is a relative connection on $S$, then we obtain in a natural way a relative connection $\delta\Theta$ on $\delta S$.
\end{rem}
\begin{rem}\label{lie conn} Given a 1-form $\alpha$  on $U^1$ vanishing in the $M$ direction, note that the Lie derivative 
\begin{equation} \pounds_{\frac{d}{dt}}\alpha\end{equation} also vanishes in the $M$ direction.  Therefore we obtain a well-defined Lie derivative 
\begin{equation} \pounds_{\frac{d}{dt}}:i\Omega^1_{U^1,rel}\to i\Omega^1_{U^1,rel}.\end{equation}  Furthermore, given a $\ul{\T}_{U^1}$-torsor $E$ with relative connection $\Theta$, we can generalize the construction in example (\ref{triv conn lift}) to construct a relative connection $\pounds_{\frac{d}{dt}}(\Theta)$ on the $\ul{i\R}_{U^0}$-torsor $\iota_{\frac{d}{dt}}d\log[E]$ according to the formula
\begin{equation} A_{\pounds_{\frac{d}{dt}}(\Theta)}(\iota_{\xi}d\log[\sigma]+f)=\pounds_{\frac{d}{dt}}A_{\Theta}(\sigma)+df.\end{equation}
\end{rem}

\begin{definition}  \label{local lift cat conn}A \emph{local connective lift} of $\Phi$ to $\textrm{B}(\ul{\T}_{U})$ consists of a quadruple $(S,e,\sigma,\Theta)$, where $(S,e,\sigma)$ are as in definition (\ref{local lift cat}), and $\Theta$ is a relative connection on $S$ satisfying 
\begin{enumerate}[(i)] \item $A_{s_0^*\Theta}(e)=0$ and 
\item $A_{\delta\Theta}(\sigma)=0$.
\end{enumerate}  Given another local lift $(S',e',\sigma,\Theta')$ and isomorphism of lifts from $(S,e,\sigma,\Theta)$ to $(S',e',\sigma',\Theta')$ is an equivalence $\Psi:(S,e,\sigma)\to (S',e',\sigma')$ as in definition (\ref{local lift cat}) satisfying $\Psi^*\Theta'=\Theta.$  We denote the resulting category by $\check{\L}(\Phi)$.
\end{definition}

We next extend the differentiation functor $D$ defined in (\ref{local diff fun}) to a functor $\check{D}$ that takes connections into account.  Let $\check{\L}(\xi)$ denote the category of connective lifts of $\xi$ to the trivial gerbe with trivial connective structure over $U^0$.   It follows from (\ref{conn lift class}) that every connective lift $\cx=(F_{\cx},\Theta_{\cx})\in \check{\L}(\xi)$ of $\xi$ is determined by $\ul{i\R}_{U^0}$-torsor $F_{\cx}(\ul{\T}_{U^0})$ with the connection 
\begin{equation} \Theta_{\cx}(\Theta_0), \end{equation} where $\Theta_0$ is the trivial flat connection on $\ul{\T}_{U^0}$.   Put differently, there is an equivalence of categories from $\check{\L}(\xi)$ to the category $\ctor_{\ul{i\R}_{U^0}}$ whose objects are $\ul{i\R}_{U^0}$-torsors with connection.  For simplicity, we will therefore take $\ctor_{\ul{i\R}_{U^0}}$ to be the codomain of the differentiation functor $\check{D}$.

Given $(S,e,\sigma,\Theta)\in \L(\Phi)$, recall from remark (\ref{lie conn}) that we obtain a relative connection $\pounds_{\frac{d}{dt}}(\Theta)$ on $\iota_{\frac{d}{dt}}d\log[S]$.  Furthermore, we can restrict to obtain a connection $D(\Theta)=i^*\pounds_{\frac{d}{dt}}(\Theta)$ on the restriction of $\iota_{\frac{d}{dt}}d\log[S]$ to $U^0$.  One easily checks that the assignment $\Theta\mapsto D(\Theta)$ is functorial.

\begin{definition} The functor 
\begin{equation} \check{D}:\check{\L}(\Phi)\to \ctor_{\ul{i\R}_{U^0}}\end{equation} assigns to each element $(S,e,\sigma,\Theta)$ the $\ul{i\R}_{U^0}$-torsor with connection 
\begin{equation} (D(S),D(\Theta)).\end{equation}
\end{definition}

\begin{theorem} $\check{D}$ is an equivalence of categories.
\end{theorem}
\begin{proof} Let $\hat{\Phi}_0\in \L(\Phi)$ be the trivial lift of $\Phi$, and let $\check{\L}_0(\Phi)$ denote the subcategory of $\check{\L}(\Phi)$ whose underlying non-connective lift is $\hat{\Phi}_0$.  It follows from proposition (\ref{local lift class}) that the inclusion of $\check{\L}_0(\Phi)$ into $\check{\L}(\Phi)$ is an equivalence of categories.    Similarly, let $\check{\L}_0(\xi)$ denote the subcategory of $\ctor_{\ul{\i\R}_{U^0}}$ with underlying torsor $E_0=D(\hat{\Phi}_0)$; thus an element of $\check{\L}_0(\xi)$ is a connection $\Theta$ on $E_0$ and a morphism from $\Theta\to \Theta'$ is an automorphism of $E_0$ taking $\Theta$ to $\Theta'$.  Since every $\ul{i\R}_{U^0}$-torsor is isomorphic to $E_0$, it follows that the inclusion $\check{\L}_0(\xi)\to \ctor_{\ul{i\R}_{U^0}}$ is an equivalence of categories.  Therefore, it is sufficient to check that the restriction 
\begin{equation} \check{D}:\check{\L}_0(\Phi)\to \check{\L}_0(\xi)\end{equation} is an equivalence of categories.  

Note that we may identify the set of objects of $\check{\L}_0(\Phi)$ with the set of relative 1-forms $A$ on $U^1$ satisfying conditions (i) and (ii) in definition (\ref{local lift cat conn}).  Let us call this set $Z_{i\Omega^1}$. 
  We may view $A\in Z_{i\Omega^1}$ as a section 
\begin{equation} (x,\vec{t})\mapsto A_{(x,\vec{t})}:=A_x(\vec{t})\in T^*_xM.\end{equation}  The conditions (i) and (ii) can then be written 
\begin{equation}\label{A condition 1} \varphi^*_tA_{\varphi_t(x)}(t')-A_x(t+t')+A_x(t)=0 \end{equation} for each $x\in V$ and 
\begin{equation}\label{A condition 2} A(0)=0.\end{equation}   
We have a linear differentiation map 
\begin{equation} \Delta_{i\Omega^1}:Z_{i\Omega^1}\to i\Omega^1(U^0) \end{equation} mapping $A\in Z_{i\Omega^1}$ to  
\begin{equation} \frac{d}{ds}|_{s=0}A_x(s)\end{equation}  

\begin{lemma}\label{diff lemma} $\Delta_{i\Omega^1}$ is an isomorphism of vector spaces.
\end{lemma}
\begin{proof}  Denote $\Delta_{i\Omega^1}A$ by $\alpha\in i\Omega^1(U^0)$.  Equation (\ref{A condition 1}) implies that 
\begin{equation}  \frac{A_x(t+t')-A_x(t)}{t'}=\varphi_t^*(\frac{A(t')}{t'}),\end{equation} and therefore 
\begin{equation} \frac{d}{ds}|_{s=t}A_x(s)=\varphi_t^*\frac{d}{ds}|_{s=0}A_{\varphi_t(x)}(s)=\varphi_t^*\alpha_{\varphi_t}.\end{equation}  By the fundamental theorem of calculus and condition (\ref{A condition 2}) we therefore have 
\begin{equation}\label{A fund thm} A_x(t)=\int_0^t\varphi_s^*\alpha_{\varphi_s(x)}ds.\end{equation} Therefore $A$ is completely determined by $\alpha$ and $\Delta_{i\Omega^1}$ is injective.  Conversely, given an arbitrary $\alpha\in i\Omega^1(U^0)$, it is easly verified that the relative 1-form on $U^1$ defined by equation (\ref{A fund thm}) is an element of $Z$, i.e. it satisfies conditions (\ref{A condition 1}) and (\ref{A condition 2}).

\end{proof}

Similarly, we my identify the set of objects of $\check{\L}_0(\xi)$ with the set $i\Omega^1_{U^0}$ of 1-forms on $U^0$.  Moreover, we have a commutative diagram 
\begin{equation} \xymatrix{ Z_{i\Omega^1} \ar[r]^-{\cong} \ar[d]_{\Delta_{i\Omega^1}} & \textrm{Ob}(\check{\L}_0(\Phi)) \ar[d]^{\check{D}} \\ i\Omega^1_{U^0} \ar[r]_-{\cong} & \textrm{Ob}(\check{\L}_0(\xi)).}\end{equation}

 Thus, lemma (\ref{diff lemma}) implies that $D$ is a bijection on the level of objects.  Since both categories are groupoids, it only remains to show that for every object $x\in \check{\L}_0(\Phi)$ that $D$ gives an isomorphism from $\Aut(x)\tocong \Aut(\check{D}(x))$.

Recall the set $Z_{\T}$ of functions $U^1\to \T$ defined in the proof of theorem (\ref{diff equiv}), as well as the map $\Delta:Z_{\T}\to iC^{\infty}(U^0)$.  One easily checks that an element $g\in Z_{\T}$ defines an automorphism of $x$ if and only if the relative 1-form on $U^1$ determined by $d\log g$ is equal to zero; call this subset $Z_0\subset Z_{\T}$. On the other hand, a function $f:U^0\to i\R$ defines an  automorphism of $D(x)$ if and only if $f$ is constant.   Equation (\ref{diff inverse})  then implies that  $\Delta_{\T}:Z\to iC^{\infty}(U^0)$ restricts to a bijection between $Z_0$ and the constant functions on $U^0$. 
\end{proof}

\section{\label{courant}The Courant Algebroid Associated to $\check{\L}_{\CC}$}

  In \cite{H}, starting with Cech data for a DD gerbe with connective structure over a manifold $M$, Hitchin constructs an extension of vector bundles
  \begin{equation}\xymatrix{ 0 \ar[r] &T^*M \ar[r] & E \ar[r] & TM \ar[r] & 0},\end{equation} which he calls the ``generalized tangent bundle".\footnote{For an alternative construction due originally to Severa in terms of ``conducting bundles", see \cite{BCh}.}  Moreover, the vector bundle $E$ is equipped with a symmetric, non-degenerate pairing and a skew-symmetric bracket, giving $E$ the structure of a \emph{Courant Algebroid} \cite{Gu}. In this section we give an alternative construction of this algebroid in terms of the category of infinitesimal connective symmetries of $(\CC,\A)$.\footnote{More precisely, for simplicity we give a construction of the $C^{\infty}_M$-module of global sections of $E$.  Since all our constructions in this paper are local, however, we could easily generalize to construct the entire sheaf of sections of $E$.}     In particular, our construction does not depend on a choice of local trivializations for ($\CC,\A$); given such a choice, however, we explain how to compare our construction with that of Hitchin.  It follows from the results of \cite{RW} than any  Courant algebroid $E$ gives rise to a 2-term $L_{\infty}$-algebra $L_E$.  We recall this construction and prove that there is an isomorphism of $L_{\infty}$-algebras $L_E\cong \L_{(g_{ijk},A_{ij})}$. Since the latter is a \emph{strict} 2-term $L_{\infty}$-algebra (i.e. $J=0$),  we obtain a useful perspective on $L_E$.  
  
  \subsection{\label{Hitchin}The Courant Algebroid associated to $(\{g_{ijk}\},\{A_{ij}\})$.}
On any smooth manifold $M$, the Courant bracket is defined on sections of $TM\oplus T^*M$ by the formula \cite{Gu} \begin{equation} [\xi+a,\eta+b]_c=[\xi,\eta]+\pounds_{\xi}b-\pounds_{\eta}a-\frac{1}{2}d\iota_{\xi}b+\frac{1}{2}d\iota_{\eta}a.\end{equation}  Here $\xi$ and $\eta$ are vector fields, $a$ and $b$ are 1-forms, and $[\xi,\eta]$ is the regular Lie bracket of vector fields.  This bracket has several interesting features.  First, although it is skew symmetric, it does not satisfy the Jacobi identity.  However, the defect in the Jacobi identity is easily expressed in terms of the non-degenerate pairing on $C^{\infty}(TM\oplus T^*M)$ given by
\begin{equation} \label{courant pairing}\langle \xi+a,\eta+b\rangle=\frac{1}{2}(\iota_{\xi}b+\iota_{\eta}a).\end{equation}  If we define \begin{equation}\textrm{Jac}(A,B,C)=[[A,B]_c,C]_c+[[B,C]_c,A]_c+[[C,A]_c,B]_c\end{equation} for $A,B,C\in C^{\infty}(TM\oplus T^*M)$, then one can show (see for example  \cite{Gu})
\begin{equation} \textrm{Jac}(A,B,C)=d(\textrm{Nij}(A,B,C)),\end{equation} where the \emph{Nijenhuis tensor} $\textrm{Nij}$ is defined 
\begin{equation} Nij(A,B,C)=\frac{1}{3}(\langle[A,B]_c,C\rangle+\langle[B,C]_c,A\rangle+\langle[C,A]_c,B\rangle).\end{equation}  The properties of the Courant bracket and pairing on $TM\oplus T^*M$ motivate the definition of a \emph{Courant algebroid}\cite{Gu}\cite{BCh}.

Another interesting feature of the Courant bracket is that it admits ``B-field" transformations as symmetries.  Namely, if $B$ is a closed 2-form on $M$, then a simple computation shows
\begin{equation} [\xi+a+\iota_{\xi}B,\eta+b+\iota_{\eta}B]_c=[\xi+a,\eta+b]_c+\iota_{[\xi,\eta]}B\end{equation} so that the transformation 
\begin{equation}\label{b-field transform} \xi+a\mapsto \xi+a+\iota_{\xi}B\end{equation} is compatible with the bracket.  In addition, the skew-symmetry of B implies that the transformation (\ref{b-field transform}) preserves the pairing (\ref{courant pairing}).  The existence of these symmetries allows one to ``twist" the Courant bracket on $TM\oplus T^*M$ to obtain more general Courant algebroids.



In particular, suppose we are given a DD gerbe with connective structure $(\CC,\A)$ over $M$, then we recall the following construction from \cite{H}.     Let $\{\{U_i\},\{Q_i\},\{s_{ij}\},\{\mu_i\}\}$ be a collection of local trivializations for $(\CC,\A)$, and let $\{\{g_{ijk}\},\{A_{ij}\}\}$ be the corresponding Cech data.   From the relation\begin{equation} A_{jk}-A_{ik}+A_{ij}=d\log(g_{ijk}), \end{equation}  it follows that the linear transformation 
\begin{equation} \left(\begin{array}{cc} 1 & 0 \\ -dA_{ij} & 1\end{array} \right):TU_{ij}\oplus T^*U_{ij}\to TU_{ij}\oplus T^*U_{ij} \end{equation} defines an extension $E$ of $TM$ by $T^*M$, where $dA_{ij}$ acts on $TU_{ij}$ by contraction $\xi\mapsto \iota_{\xi}dA_{ij}$.  In other words, a (global) section of $E$ is described by a pair $(\xi,\{a_i\})$ with $a_i\in i\Omega^1(U_i)$ satisfying 
\begin{equation} a_j-a_i=-\iota_{\xi}dA_{ij}\end{equation}on the overlaps $U_{ij}$.  Because the 2-forms $dA_{ij}$ are closed, we obtain a well-defined Courant bracket $[\cdot,\cdot]_E$ and bilinear pairing $\langle\cdot,\cdot\rangle_E$ on sections of $E$.

Suppose that in addition we are given a curving $K$ for $\A$ with corresponding 2-forms 
\begin{equation} B_i=K(\mu_i)\end{equation} satisfying
\begin{equation} B_j-B_i=dA_{ij}\end{equation} as discussed in \S\ref{conn str and curving}.  We can then define a splitting $s:TM\oplus T^*M\to E$ by 
\begin{equation} s(\xi+a)=(\xi,\{a-\iota_{\xi}B_i\}).\end{equation}  We then have 
\begin{equation} \langle s(\xi+a),s(\eta+b)\rangle_E=\langle \xi+a,\eta+b\rangle \end{equation} and 
\begin{equation} [s(\xi+a),s(\eta+b)]_E=s([\xi+a,\eta+b]_c)+\iota_{\xi}\iota_{\eta}C, \end{equation} where $C=dB_i$ is the curvature form.


\subsection{The Courant algebroid in terms of infinitesimal connective lifts}
We next explain a construction of $E$ in terms of our category of infinitesimal connective lifts.  Let 
\begin{equation} \tilde{\pi}:\L_{(\CC,\A)}\to \L_{\CC} \end{equation} be the obvious forgetful functor.  Recall from proposition (\ref{conn lift class}) that for each object $\tx\in \L_{\CC}$, the set $\tilde{\pi}^{-1}(\tx)$ of connective lifts extending $\tx$ is a torsor for the group of 1-forms $i\Omega^1$.  Furthermore, since we have specified a connective structure $\A$ on $\CC$, for each vector field $\xi$ we have the horizontal lift $\tx^h\in \L_{\CC}(\xi)$ defined in (\ref{horizontal}).  We may then define
\begin{definition} 
\begin{equation} E=\coprod_{\xi\in C^{\infty}(TM)}\tilde{\pi}^{-1}(\tx^h). \end{equation} \end{definition}

Let $\pi:E\to C^{\infty}(TM)$ denote the projection map, and for each $\xi\in C^{\infty}(TM)$ let $E_{\xi}=\pi^{-1}(\xi)$ denote the set of connective lifts extending $\xi^h$.  Let us describe the sets $E_{\xi}$ more concretely.  Given $\cx=(\xi^h,\Theta_{\cx})\in E_{\xi}$,  for each object $Q\in \CC(U)$ and each local section $\mu$ of $\A(Q)$, define a 1-form $a_{\cx}(\mu)$ by 
\begin{equation}\label{ext map} a_{\cx}(\mu)=A_{\Theta_{\cx}(\mu)}(\iota_{\xi}[\mu]).\end{equation}  The relations
\begin{equation} \Theta_{\cx}(\mu+\alpha)=\Theta_{\cx}(\mu)-\pounds_{\xi}\alpha \end{equation} together with the Cartan formula for the Lie derivative imply that for each $1$-form $\alpha$ we have 
\begin{equation} \label{ext 1}a_{\cx}(\mu+\alpha)=a_{\cx}(\mu)-\iota_{\xi}d\alpha.\end{equation}  Equivalently, we have an isomorphism of $\ul{i\Omega}^1_U$-torsors 
\begin{equation}-\iota_{\xi}d[\A(Q)]\tocong \ul{i\Omega}^1_U.\end{equation}
In addition, one can check directly that given $\psi:Q\to R$ we have 
\begin{equation} \label{ext 2}a_{\cx}(\mu)=a_{\cx}(\psi_*\mu),\end{equation}  and that given an inclusion $i:V\to U$, we have 
\begin{equation}\label{ext 3} a_{\cx}(i^*\mu)=a_{\cx}(\mu)|_{V}.\end{equation} 
Unwinding definition (\ref{2-morphism}) we see that the conditions (\ref{ext 2}) and (\ref{ext 3}) imply that the assignment (\ref{ext map}) is equivalent to a 2-morphism from $\iota_{\xi}d[\A]$ to the constant 1-morphism $\CC\to \textrm{B}(\ul{i\Omega}^1_M)$ taking every $Q\in \CC(U)$ to the trivial $\ul{i\Omega}^1_U$-torsor.\footnote{For comparison, see remark (\ref{curving restatement})}.

\begin{proposition} \label{hor ext char}The correspondence $\cx\mapsto a_{\cx}$ defines a bijection between $E_{\xi}$ and the set of 2-morphisms  
\begin{equation} -\iota_{\xi}d[\A]\toocong \ul{i\Omega}^1,\end{equation} where $ \ul{i\Omega}^1$ denotes the constant 1-morphism taking every $Q\in\CC(U)$ to $\ul{i\Omega}^1_U.$
\end{proposition}
\begin{proof}  By definition, an element of $E_{\xi}$ is an equivalence\begin{equation} \Theta_{\cx}:\pounds_{-\xi}[\A]\toocong \A_0(F_{\tx^h})=-d[F_{\tx^h}],\end{equation} where by definition 
\begin{equation} F_{\cx^h}=\iota_{\xi}[\A].\end{equation}  Suppose we are given an equivalence \begin{equation} a_{\cx}:-\iota_{\xi}d[\A]\toocong \ul{i\Omega}^1_M\end{equation} as in the proposition.  Then applying proposition (\ref{ass gerbe coh}) repeatedly we can construct an equivalence 
\begin{align} \pounds_{-\xi}[\A] & \toocong -\iota_{\xi}d[\A]\bp(-d\iota_{\xi}[\A]) \toocong \ul{i\Omega}^1_M\bp (-d\iota_{\xi}[\A])  \\ & \toocong -d\iota_{\xi}[\A]\toocong \A_0(F_{\cx^h}).\notag\end{align}

By corollary (\ref{conn lift class}) the set $E_{\xi}$ is a torsor for the group of global 1-forms on $M$; on the other hand, it is not hard to see that the set of 2-morphisms described in the proposition is also a  torsor for 1-forms and that the map taking $a_{\cx}$ to $\Theta_{\cx}$ is by construction a morphism of torsors.  This map is therefore both one-to-one and onto.  
\end{proof} 

Using the characterization of sections of $E$ given in proposition (\ref{hor ext char}), we can now endow $E$ with several interesting algebraic structures. Given an element $\cx$ of $E_{\xi}$, let $a_{\cx}$ denote the corresponding function satisfying (\ref{ext 1}), (\ref{ext 2}) and (\ref{ext 3}).  Given $\cx,\ce\in E$, we define $\cx+\ce$ by 
\begin{equation} \label{E sum} a_{\cx+\ce}=a_{\cx}+a_{\ce}, \end{equation} and given $f\in C^{\infty}(M)$, we define $f\cx$ by 
\begin{equation} \label{E times f}  a_{f\cx}=fa_{\cx}.\end{equation}  It is then readily verified that $a_{\cx+\ce}$ and $a_{f\cx}$ satisfy (\ref{ext 1}), (\ref{ext 2}) and (\ref{ext 3}).  Furthermore, we define $0\in E_0$ to be the function taking every local section $\mu$ of $\A(Q)$ to $0$.  Altogether $E$ obtains the structure of a $C^{\infty}(M)$-module.  Furthermore, it is clear that the projection map $\pi:E\to C^{\infty}(TM)$ is a map of modules.  Therefore we obtain an extension of modules \begin{equation}\label{atiyah 3}\xymatrix{ 0 \ar[r] & i\Omega^1(M) \ar[r] & E \ar[r] & C^{\infty}(TM)\ar[r] & 0.}\end{equation}
  
Suppose we are given a curving $K$.  If we then define 
\begin{equation} a_{s(\xi)}(\mu)=-\iota_{\xi}K(\mu), \end{equation} then it follows from definition (\ref{def curving}) that $a_{s(\xi)}$ satisfies condition (\ref{ext 1}).  Therefore a curving determines a splitting of the sequence (\ref{atiyah 3}).

We next construct a non-degenerate pairing $\langle\cdot,\cdot\rangle_E$ on $E$.  Given $a_{\cx},a_{\ce}\in E$ locally over an open set $U\subset M$ by 
\begin{equation}\label{pairing def} \langle a_{\cx},a_{\cx}\rangle_E|_U=\frac{1}{2}(\iota_{\xi}a_{\ce}(\mu)+\iota_{\eta}a_{\cx}(\mu)),\end{equation} for $\mu$ is a global section of $\A(Q)$ for some $Q\in \CC(U)$.  By (\ref{ext 1}) and (\ref{ext 2}), we see that (\ref{pairing def}) is independent of the choice of $Q$ and $\mu$.  Furthermore, if we cover $M$ by open sets $\{U_i\}$ such that $\CC(U_i)$ is non-empty for each $i$, then condition (\ref{ext 3}) implies that the local the formula (\ref{pairing def}) is consistent with restrictions and we therefore obtain a global function on $M$.  Clearly $\langle\cdot,\cdot\rangle_E$ is bilinear over functions and symmetric.  To see that it is also non-degenerate, suppose the for some $\cx\in E$ we have 
\begin{equation} \langle\cx,\ce\rangle_E=0\end{equation} for each $\ce\in E$.  In particular, for each 1-form $\alpha$ on $M$, if we let $\eta=0$ and $\ce$ equal the image of $\alpha$ in $E_0$, we have 
\begin{equation} \frac{1}{2}\iota_{\xi}\alpha=0.\end{equation}  Since this holds for arbitrary $\alpha$ we must have $\xi=0$.  Therefore $\cx\in E_0$, say $\cx=\beta$ for $\beta$ a 1-form.  But equation (\ref{pairing def}) implies that $\beta$ must pair with each vector field trivially, and must therefore vanish.

There is also a natural bracket on $E$.  Given $\cx,\ce\in E$, define $[\cx,\ce]_E$ by 
\begin{equation} a_{[\cx,\ce]}=\pounds_{\xi}a_{\ce}-\pounds_{\eta}a_{\cx}-\frac{1}{2}d\iota_{\xi}a_{\ce}+\frac{1}{2}d\iota_{\eta}a_{\cx}.
\end{equation}
To verify that this does in fact define an element of $E_{[\xi,\eta]}$, note that 
\begin{align} & a_{[\cx,\ce]_E}(\mu+\alpha)-a_{[\cx,\ce]_E}(\mu) \\
 =  & -(\pounds_{\xi}\iota_{\eta}d\alpha-\pounds_{\eta}\iota_{\xi}d\alpha+\frac{1}{2}d\iota_{\xi}\iota_{\eta}d\alpha-\frac{1}{2}d\iota_{\eta}\iota_{\xi}d\alpha \notag) \\
 =  &-(\pounds_{\xi}\iota_{\eta}d\alpha-\iota_{\eta}d\iota_{\xi}d\alpha+d\iota_{\eta}\iota_{\xi}d\alpha-d\iota_{\eta}\iota_{\xi}d\alpha \notag) \\
 =  &-(\pounds_{\xi}\iota_{\eta}d\alpha-\iota_{\eta}\pounds_{\xi}d\alpha \notag )\\
 =  &-[\pounds_{\xi},\iota_{\eta}]d\alpha \notag \\
 =  &-\iota_{[\xi,\eta]}d\alpha.\notag \end{align}

Although it is possible to show directly that these structures satisfy the compatibility conditions possessed by the sections of a Courant algebroid, we will instead proceed by comparing our construction to that in \cite{H}.  We introduce local trivializations $\{\{U_i\},\{Q_i\},\{s_{ij}\},\{\mu_i\}$ for $(\CC,\A)$ with corresponding Cech data $(\{g_{ijk}\},\{A_{ij}\})$.  Recall from \S\ref{conn str and curving} that we obtain Cech data for the horizontal lift of a vector field $\xi$ 
\begin{equation} f_{ij}^{\tx^h}=\iota_{\xi}A_{ij}.\end{equation}  A simple calculation then shows that the collection\footnote{To avoid notational clutter, we will surpress the brackets around $f^{\cx}_{ij}$ and $a^{\cx}_{ij}$.}
\begin{equation} (\xi,\iota_{\xi}A_{ij},a^{\cx}_i) \end{equation} is Cech data for a connective lift if and only if 
\begin{equation} a^{\cx}_j-a^{\cx}_i=-\iota_{\xi}dA_{ij}.\end{equation}   In terms of the functions $a_{\cx}$ given above, the 1-forms $a_i$ are given by 
\begin{equation} a^{\cx}_i=a_{\cx}(\mu_i).\end{equation}   Therefore in the Cech picture the pairing is given by 
\begin{equation} \langle (\xi,\iota_{\xi}A_{ij},a^{\cx}_i),(\eta,\iota_{\eta}A_{ij},a^{\ce})\rangle_E =\frac{1}{2} (\iota_{\xi}a^{\ce}+\iota_{\eta}a^{\cx}), \end{equation} and the bracket by 
$ [(\xi,\iota_{\xi}A_{ij},a^{\cx}_i),(\eta,\iota_{\eta}A_{ij},a^{\ce}]=$
\begin{equation} ([\xi,\eta],\iota_{[\xi,\eta]}A_{ij},\pounds_{\xi}a^{\ce}_i-\pounds_{\eta}a^{\cx}_i-\frac{1}{2}d\iota_{\xi}a^{\ce}_i+\frac{1}{2}\iota_{\eta}a^{\ce}_i).\end{equation}  We see that these are exactly the formulas used to define the Courant algebroid structure in \S\ref{Hitchin}.

In \cite{RW} it was shown that any Courant algebroid $E$ gives rise to an $L_{\infty}$-algebra.  Moreover, as pointed out in \cite{R1}, this $L_{\infty}$-algebra can be restricted to one with 2-terms, which we call $L_E$.   \begin{definition} In terms of local trivializations $\{\{U_i\},\{Q_i\},\{s_{ij}\},\{\mu_i\}\}$,  the 2-term $L_{\infty}$-algebra $L_E=\{W_0\oplus W_1,d,[\cdot,\cdot],Jac\}$ is given by 
\begin{enumerate} \item $W_0=C^{\infty}(E)$, $W_1=C^{\infty}(\ul{i\R}_M)$.
\item $d:W_1\to W_0$ is given by 
\begin{equation} f \mapsto (0,a_i=df|_{U_i}).\end{equation} 
\item $[\cdot,\cdot]:W_0\times W_0\to W_0$ is given by the Courant bracket.  Given $(\xi,a_i)\in W_0$ and $f\in W_1$, we define 
\begin{equation} [(\xi,a_i),f]=-[f,(\xi,a_i)]=\frac{1}{2}\xi\cdot f.\end{equation}
\item $\textrm{Jac}(A,B,C)=-\frac{1}{3}\textrm{Nij}(A,B,C). $ 
\end{enumerate} 
\end{definition} 

Suppose that we are given two sections $\tx,\te$ of the Courant algebroid $E$, as described by Cech data.  In particular, $\tx$ and $\te$ are elements of the category $\L_{(g_{ijk},A_{ij})}$, and we can take their bracket $[\tx,\te]$ as described in equation (\ref{conn bracket cech}).  This yields another connective lift which is not in general a section of $E$.  On the other hand, as discussed above the Courant bracket of $\tx$ and $\te$ does define another section of $E$.  Although these two brackets are not equal, they are related by a natural isomorphism.  The following theorem explains this relationship in the language of $L_{\infty}$-algebras.

\begin{theorem} Let $E$ be the Courant algebroid constructed above.  Then there is an isomorphism of 2-term $L_{\infty}$ algebras
\begin{equation} \Phi:L_{E} \tocong L_{{(g_{ijk},A_{ij})}}.\end{equation}
\end{theorem}
\begin{proof}  First, we define a degree 0 chain map $\phi_0:W_0 \to V_0$.  In degree 0 we define 
\begin{equation} \phi_0:(\xi,\{a_{j}\})\mapsto (\xi,\{\iota_{\xi}A_{ij}\},\{a_i\});\end{equation} by our earlier discussion right-hand side defines an object of $\L_{(g_{ijk},A_{ij})}$.  In degree 1 we define $\phi_0$ to be the identity map on $C^{\infty}(\ul{i\R}_M)$.  The verification that $\phi_0$ is a chain map is trivial.  

Given $x=(\xi,\{a_i\})$ and $y=(\eta,\{b_i\})$, note that 
\begin{align} \phi([x,y]) & = \phi([\xi,\eta],\{\pounds_{\eta}b_i-\pounds_{\xi}a_i-\frac{1}{2}d\iota_{\eta}b_i+\frac{1}{2}d\iota_{\eta}a_i\}) \\ & = ([\xi,\eta],\{\iota_{[\xi,\eta]}A_{ij}\},\{\pounds_{\eta}b_i-\pounds_{\xi}a_i-\frac{1}{2}d\iota_{\eta}b_i+\frac{1}{2}d\iota_{\eta}a_i\}).\notag
\end{align}  On the other hand we have 
\begin{equation} [\phi(x),\phi(y)]=([\xi,\eta],\{\xi\cdot(\iota_{\eta}A_{ij})-\eta\cdot(\iota_{\xi}A_{ij})\},\{\pounds_{\xi}b_i-\pounds_{\eta}a_i\}).\end{equation}  Using the formula
\begin{equation} \iota_{\eta}\iota_{\xi}A_{ij}=\xi\cdot(\iota_{\eta}A_{ij})-\eta\cdot(\iota_{\xi}A_{ij})-\iota_{[\xi,\eta]}A_{ij},\end{equation} we therefore have
\begin{equation} \phi([x,y])-[\phi(x),\phi(y)]=(0,\{-\iota_{\eta}\iota_{\xi}A_{ij}\},d(-\frac{1}{2}\iota_{\xi}b_i+\frac{1}{2}\iota_{\eta}a_i\}).\end{equation}  If we then define 
\begin{align} \phi_2: C^{\infty}(E)\times C^{\infty}(E) & \to C^{\infty}(\ul{i\R}_M)
\\ (\xi,\{a_i\}),(\eta,\{b_i\}) & \mapsto \{-\frac{1}{2}\iota_{\xi}b_i+\frac{1}{2}\iota_{\eta}a_i\},
\end{align} a simple computation shows that 
\begin{equation} d\phi_2(x,y)=\phi([x,y])-[\phi(x),\phi(y)]\end{equation} for all $x,y\in C^{\infty}(E)$.  To finish the proof that $\Phi=(\phi,\phi_2)$ defines a homomorphism of $L_{\infty}$-algebras, we need to verify that equation (2) in definition (\ref{L-infty hom}) holds.  We omit this computation, which is somewhat tedious but straightforward.

\end{proof} 
\appendix
\section{Appendix}

In this appendix we discuss smooth families of symmetries of gerbes and their relationship to infinitesimal symmetries.  This gives a direct connection between equivariant gerbes and the infinitesimal symmetries explored in this paper.  We consider both gerbes with and without connective structures.
\subsection{Gerbes without connective structures}

Let $\Phi$ be a 1-parameter family of diffeomorphisms of $M$ generated by a vector field $\xi$.  We now explain how a family of symmetries of $\CC$ lifting $\Phi$ gives rise to an infinitesimal symmetry of $\CC$ lifting $\xi$ via a process analogous to differentiation.  The local version of this functor was discussed in \S\ref{local flows}

Let us briefly return to the case where $P$ is a principal circle bundle over $M$.  Given any diffeomorphism $\varphi$ of $M$, we can form the pull-back $\varphi^*E$.  By definition, for each $m\in M$ the fiber $(\varphi^*E)_m$ can be identified with the fiber $E_{\varphi(m)}$, and a symmetry of $E$ lifting $\varphi$ is equivalent to an isomorphism $E\tocong \varphi^*E$.  More generally, given a smooth family of diffeomorphisms \begin{equation} \Phi:M\times \R\to M,\end{equation} a smooth family of symmetries of $E$ lifting $\Phi$ is equivalent to an isomorphism 
\begin{equation}\hat{\Phi}:p_1^*E\to p_0^*E\end{equation} of bundles over $M^1=M\times \R$, where as discussed in \S\ref{VF} $p_0=\Phi$ and $p_1$ is projection onto $M$.  Equivalently, the isomorphism $\hat{\Phi}$ may be encoded as a global section of the bundle
\begin{equation} \delta E=p_0^*E\otimes p_1^*E^{\vee}.\end{equation}




To generalize to gerbes, recall from \S\ref{gerbes}, that in addition to forming the inverse image of a gerbe, we may define the tensor product of two gerbes and the dual $\CC^{\vee}$ of a gerbe $\CC$ using the associated gerbe construction.  Thus we may define a gerbe 
\begin{equation} \delta\CC=p_0^*\CC\otimes p_1^*\CC^{\vee}\end{equation} over $M^1$.  Recall also from \S\ref{VF} that for each open set $U\in M$ we define an open set 
\begin{equation} \nu(U)=p_0^{-1}(U)\cap p_1^{-1}(U)\subset M^1.\end{equation} Given any sheaf $\F$ over $M^1$, we can define a sheaf $\nu_*\F$ over $M$ by 
\begin{equation} \nu_*\F(U)=F(\nu(U))\end{equation} for any open set $U\subset M$; note that this definition makes sense both for sheaves taking values in a category (e.g. the category of groups) and for gerbes.  We can also define a homomorphism of sheaves of groups over $M$ 
\begin{equation} \delta_{\T}:\ul{\T}_{M}\to \nu_*(\ul{\T}_{M^1})\end{equation} by 
\begin{equation} \label{delta 1 g} \delta_{\T} g=(p_0^*g)(p_1^*g^{-1}).\end{equation} 
\begin{rem}\label{delta canonical} By construction we have a canonical morphism of gerbes 
\begin{equation} \delta_{\CC}:\CC\to \nu_*\delta\CC\end{equation} intertwining $\delta_{\T}$.
\end{rem}

\begin{definition}\label{def 1 par symmet gerbe cat} The category of smooth familes of symmetries of $\CC$ lifting $\Phi$ is the category $\L_{\CC}(\Phi)$ of global sections of $\delta \CC$ over $M$.
\end{definition}  
\begin{rem} Definition (\ref{def 1 par symmet gerbe cat}) is similar to that of an $\R$-equivariant gerbe over $M$; see \cite{Br2}, \cite{G2}, \cite{Mein} for general discussions of equivariant gerbes.  In particular, an extension of $\CC$ to an $\R$-equivariant gerbe gives an example of a symmetry of $\CC$ lifting $\Phi$ in the sense of definition (\ref{def 1 par symmet gerbe cat}).  
\end{rem}

\subsection{The differentiation functor}

We now explain the construction of a functor $\L_{\CC}(\Phi)$ to $\L_{\CC}(\xi)$, the category of global infinitesimal symmetries of $\CC$ lifting the vector field $\xi$.  Our approach parallels the discussion in \S{\ref{VF}}; in particular the reader may find it useful to compare the constructions in this section to the derivation of equation (\ref{f dependence}).  
  
Given $\hat{\Phi}\in\L_{\CC}(\Phi)$ will construct $D(\hat{\Phi})=F_{\tx}$ as a composition of 1-morphisms which are described in the following lemma.  The existence of these 1-morphisms is a straightforward consequence of the definitions in  appendix B, and the proof will be omitted.  

 \begin{lemma}\label{some 1-morphisms}\begin{enumerate}[(1)] \item  Let $S$ be a global section of a DD gerbe $\CC$ over a manifold $M$.  Then there is a 1-morphism of gerbes (intertwining the identity on $\ul{\T}_M$) 
 \begin{equation} \ul{\Hom}(\cdot,S):\CC^{op}\to \BT\end{equation} sending 
 \begin{equation} Q\mapsto \ul{\Hom}(Q,S|_U) \end{equation} for each local section $Q\in \CC(U)$.  Moreover, the assignment 
 \begin{equation} S\mapsto \ul{\Hom}(\cdot,S)\end{equation} defines an equivalence of categories
 \begin{equation} \Gamma(\CC)\tocong \Hom(\CC^{op},\textrm{B}(\ul{\T}_M)),\end{equation} where $\Gamma(\CC)$ denotes the category of global sections of $\CC$.
 \item Let $i:N\to M$ be the inclusion of a submanifold.  Then there is a \emph{restriction} 1-morphism 
 \begin{equation} i^*[\cdot]:\BiR\to i_*\textrm{B}(\ul{i\R}_N) \end{equation} intertwining the restriction homomorphism 
 \begin{equation} \ul{i\R}_M\to i_*\ul{i\R}_N \end{equation} sending $f\in \ul{i\R}_M(U)$ to the restriction of $f$ to $U\cap N$.  
 \end{enumerate}
 \end{lemma}  
 
 \begin{notation} For convenience, let $\frac{d}{dt}\log:\ul{\T}_{M^1}\to \ul{\T}_{M^1}$ denote the homomorphism given by 
 \begin{equation} g \mapsto \iota_{\frac{d}{dt}}d\log g.\end{equation}
 \end{notation}  
 
 \begin{rem}Recall from the discussion before definition (\ref{def symmetry gerbe}) that there is a canonical equivalence 
 \begin{equation} \Hom_{\iota_{\xi}d\log}(\CC^{op},\textrm{B}(\ul{i\R}_M))\tocong \Hom_{-\iota_{\xi}d\log}(\CC,\textrm{B}(\ul{i\R}_M))\end{equation} which is actually a bijection on the level of both objects and morphisms.  For convenience we will work here with $\Hom_{\iota_{\xi}d\log}(\CC^{op},\textrm{B}(\ul{i\R}_M))$, which we will call $\tilde{\L}_{\CC}(\xi)$.
 \end{rem}
 
 \begin{definition}\label{def diff fun} $D:\L_{\CC}(\Phi)^{op}\to \tilde{\L}_{\CC}(\xi)$ is the functor
 \begin{equation} S\mapsto i^*[\frac{d}{dt}\log[\ul{\Hom}(\delta(\cdot),S)]]\end{equation} for $S\in \Gamma(\delta\CC)=\L_{\CC}(\Phi)$.
 \end{definition}
  
\begin{rem}  More formally, for each $S\in \Gamma(\delta \CC)$, $D(S)$ is the composition of 1-morphisms
 \begin{equation}\xymatrix{ \CC^{op} \ar[r]^-{\Phi_1}& \nu_*(\delta\CC^{op})\ar[r]^{\Phi_2} & \nu_*\textrm{B}(\ul{\T}_{M^1}) \ar[r]^{\Phi_3} & \nu_*\textrm{B}(\ul{i\R}_{M^1}) \ar[r]^{\Phi_4} & \BiR. }\end{equation}

Here $\Phi_1$ is the 1-morphism described in remark (\ref{delta canonical}), $\Phi_2$ is $\nu_*$ applied to the 1-morphism $\ul{\Hom}(\cdot,S)$ described in part (1) lemma (\ref{some 1-morphisms}), $\Phi_3$ is $\nu_*(\frac{d}{dt}\log)[\cdot]$. $\Phi_4$ is given by applying $\nu_*$ to the restriction 1-morphism described in part (2) of \ref{some 1-morphisms}; note that $\Phi_4$ is a 1-morphism from  
\begin{equation} \nu_*\textrm{B}(\ul{\T}_{M^1})\to \nu_*i_*\textrm{B}(\ul{i\R}_M)=\textrm{B}(\ul{i\R}_M).\end{equation}  

 We must  check that $D(S)$ does in fact intertwine the homomorphism $\iota_{\xi}d\log$.  The composition $\Phi_2\circ\Phi_1$ intertwines 
\begin{equation} \delta_{\T}:\ul{\T}_M\to \nu_*\ul{\T}_{M^1}.\end{equation} 
$\Phi_3$ intertwines $\nu_*\frac{d}{dt}\log$, and $\Phi_4$ intertwines the restriction homomorphism $\nu_*\ul{\T}_{M^1}\to \ul{\T}_M$.  Let $\zeta$ denote the composition 
\begin{equation} i^*\circ \nu_*(\frac{d}{dt}\log)\circ \delta_{\T}:\ul{\T}_M\to \ul{i\R}_M.\end{equation} Given $g:U\to \T$ we have
\begin{equation} \zeta(g)=\frac{d}{dt}\log(g(p_0)g^{-1}(p_1))|_{U}.\end{equation}  More concretely, the value of $\zeta(g)$ on $x\in U$ is given by 
\begin{align} & \frac{d}{dt}|_{t=0}\log(g(\varphi_t(x))g^{-1}(x)) \\ = & \frac{d}{dt}|_{t=0}\log(g(\varphi_t(x)) \notag \\ = & \iota_{\xi}d\log(g)(x).\notag\end{align}
\end{rem}

 
 \subsection{Gerbes with connective structures}
 
In section \ref{local flows conn} we discussed families of connective symmetries of a gerbe $\CC$ with connective structure $\A$ using relative 1-forms.  To explain why, let us briefly turn to the case of a $\T$-bundle $E\to M$ with connection $\Theta$.   Given a symmetry $\Phi:E\to E$ covering $\varphi:M\to M$, $\Phi$ is a symmetry of the pair  $(E,\Theta)$ if 
\begin{equation} \label{compatible connection} \Phi^*\Theta=\Theta,\end{equation} where $\Theta$ is viewed as a $1$-form on the total space of $E$.  Given a family of symmetries $\{\Phi_t\}$ covering $\{\varphi_t\}$ (or more generally any Lie group of symmetries $G$), there are two possible ways in which we might wish $\{\Phi_t\}$ to be compatible with $\Theta$.  On the one hand, we might ask that for each $t$, $\Phi_t$ is a symmetry of $(E,\Theta)$.  In terms of the vector field $\tx$ corresponding to $\{\Phi_t\}$, this is equivalent to requiring
\begin{equation}\label{connection preserving}\pounds_{\tx}\Theta=0.\end{equation}  On the other hand, we might want a definition such that, in the good case that the quotient of $M$ by the symmetry is a manifold, a lift of $\{\varphi_t\}$ to $(E,\Theta)$ is equivalent to a bundle with connection over the quotient $M/\R$.  In this case, in addition to requiring (\ref{connection preserving}), we must also  have 
\begin{equation} \label{quotient connection} \iota_{\tx}\Theta=0.\end{equation}  

This distinction can be understood in the simplicial language introduced in the previous section.  Recall that a 1-parameter family of symmetries $\{\Phi_t\}$ of $E$ covering $\{\varphi_t\}$ is equivalent to a global section $S$ of $\delta E=(p_1^{-1}E)^*\otimes p_0^{-1}E$ over $M^1$.  The connection $\Theta$ on $E$ determines a connection $\delta\Theta$ over $\delta E$, and conditions (\ref{connection preserving}) and (\ref{quotient connection}) are satisfied if and only if \begin{equation} \alpha=A_{\delta\Theta}S=0,\end{equation} i.e. if and only if the section $S$ is flat.  On the other hand, the condition (\ref{connection preserving}) by itself holds if and only if the restriction of $\alpha$ to $M_t=M\times\{t\}$ vanishes for each $t$, i.e. if and only if the relative 1-form determined by $\alpha$ is zero.

In this paper we will be interested only in the condition (\ref{connection preserving}) and the analogous notion for gerbes\footnote{For a discussion of the relationship between equivariant gerbes and gerbes over quotients, see \cite{G2}.}. We can define a \emph{relative connective structure} on a gerbe over $M\times \R$ by replacing 1-forms with relative 1-forms in definition (\ref{def conn str}).

  Let us generalize definition (\ref{def 1 par symmet gerbe cat}) to the situation where $\CC$ has a connective structure $\A$.  We first observe that the gerbe $\delta \CC$ over $M^1$ naturally inherits a relative connective structure, which we denote $\delta \A$.  We will sketch the construction, which is similar to that on page 211 of \cite{Br1}.  For the related notion of an equivariant gerbe with connective structure, we refer the reader to \cite{Br1} and \cite{G2}.   By construction, given local sections $Q_0\in \CC(U_0)$ and $Q_1\in \CC(U_1)$, we have an object 
  \begin{equation} \label{delta object}p_0^*Q_0\otimes p_1^*Q^{\vee}_1\end{equation} of $\delta\CC$.  We then define 
$ \delta\A(p_0^*Q_0\otimes p_1^*Q^{\vee}_1)$ to be 
\begin{equation} q[p_0^*\A(Q_0)\boxminus p_1^*(Q_1)],\end{equation} where $q:\ul{i\Omega}^1_{M^1}\to \ul{i\Omega}^1_{M^1,rel}$ is the quotient map.  We then complete the construction of $\delta \A$ using the fact that every object of $\delta \CC$ is locally isomorphic to one of the form (\ref{delta object}).  Note that for every object $Q\in \CC(U)$ and every connection $\mu\in \A(Q)$, we obtain a relative connection on $\delta Q$, which we denote by $\delta\mu$.  Furthermore, given any 1-form $\alpha\in i\Omega^1(U)$ we have 
\begin{equation} \delta^i(\mu+\alpha)=\delta^i(\mu)+p_0^*\alpha-p_1^*\alpha,\end{equation}where the relative 1-forms on the right hand side are implicitly restricted to the appropriate open set.  More formally, we have the following.
\begin{lemma} \label{delta 2 morphism} There is a 2-morphism
\begin{equation}\label{delta conn 2-mor} \delta_{\ul{i\Omega}^1_{rel}}[\A ]\Rightarrow \nu_*(\delta\A). \end{equation} \end{lemma}


\begin{definition} \label{1 par conn lift} A \emph{connective lift} of $\Phi$ to $\CC$ is a pair $(S,\mu_S)$, where $S$ is a global section of $\delta\CC$, and $\mu_S$ is a global section of $\delta\A(S)$. Given another connective lift $(S',\mu_{S'})$, an \emph{equivalence of connective lifts} from $(S,\mu)\to (S',\mu_{S'})$ is an isomorphism  $\tau:S\to S'$ such that $\tau_*(\mu_S)=\mu_{S'}$.  We will denote the corresponding category by $\L_{(\CC,\A)}(\Phi)$.
\end{definition}

In the previous section we constructed a functor from the category of 1-parameter lifts of $\Phi$ to $\CC$ to the category of global lifts of the corresponding vector field $\xi$ to $\CC$, or rather to the equivalent category 
\begin{equation} \tilde{\L}_{\CC}=\Hom_{\iota_{\xi}d\log}(\CC^{op},\textrm{B}(\ul{i\R}_M).\end{equation}
Suppose $\CC$ has a connective structure $\A$, which we recall is a 1-morphism from $\CC$ to $\textrm{B}(\ul{i\Omega}^1_M)$ intertwining $-dlog$.  Equivalently, $\A$ determines a 1-morphism 
\begin{equation} \tilde{A}\in \Hom_{d\log}(\CC^{op},\textrm{B}(\ul{i\Omega}^1_M)).\end{equation}  We may then define the category $\tilde{\L}_{(\CC,\A)}$ which is equivalent to the category of connective lifts $\L_{(\CC,\A)}$.  An object of $\tilde{\L}_{(\CC,\A)}$ is an element $F_{\tx}\in\tilde{\L}_{\CC}$ together with a 2-morphism $\Theta_{\tx}:-\pounds_{\xi}[\tilde{\A}]\toocong \tilde{A}^0_{\ul{i\R}_M}\circ F_{\tx}$.

Recall from part (1) of lemma (\ref{some 1-morphisms}) that for each global section $S$ of $\delta \CC$ over $M\times \R$, we have a 1-morphism of gerbes 
   \begin{equation} \ul{\Hom}(\cdot,S):\delta\CC^{op}\to \textrm{B}(\ul{\T}_{M^1}).\end{equation} 
Furthermore, if we fix $\mu_S\in \delta \A(S)$, then for every object $P\in \delta \CC(U)$ and every relative connection $\mu\in \delta\A(P)$, we obtain a relative connection on $\ul{\Hom}(P,S)$.  This connection, which we denote by $\mu_S-\mu$, is described explicitly as follows:  given a local section $\psi\in\ul{\Hom}(P,S)$, define
\begin{equation}\label{mu_S-mu} A_{\mu_S-\mu}(\psi)=\mu_S-\psi_*\mu,\end{equation} where by definition the right hand side is the unique 1-form $\alpha$ such that $\mu_S=\psi_*\mu+\alpha.$

Next, given a relative connection $\Theta$ on a $\T$-bundle $E\to U\subset M^1$, recall from remark (\ref{lie conn}) that we can construct a (relative) connection on $\iota_{\frac{d}{dt}}d\log[E]$, which we call $\pounds_{\frac{d}{dt}}(\Theta)$.  This is characterized by the formula \begin{equation} \label{der conn}A_{\pounds_{\frac{d}{dt}}(\Theta)}(\iota_{\frac{d}{dt}}d\log[\sigma]+f)=\pounds_{\frac{d}{dt}}A_{\Theta}(\sigma)+df.\end{equation}  
Finally, note that given a relative connection $\Theta$ on a principal $\T$-bundle $P$ over some open set $U\subset M\times \R$, we can restrict $\Theta$ to a connection on the restriction of $P$ to $U\cap (M\times \{0\})$.  The following lemma summarizes the above discussion in more formal language.

\begin{lemma}\label{conn str diff}
\begin{enumerate}[(1)]\item Let $S$ be a global section of $\delta \CC$ over $M^1$, and let $\ul{\Hom}(\cdot,S):\delta \CC^{op}\to \textrm{B}(\ul{\T}_{M^1})$ be the 1-morphism described in (\ref{some 1-morphisms}). For fixed $\mu_S\in \delta\A(S)$, there is a 2-morphism \begin{equation} -1[\tilde{\A}] \toocong  \tilde{\A}^0_{\ul{\T}_{M^1},rel}\circ \ul{\Hom}(\cdot,S) \end{equation} given by formula (\ref{mu_S-mu}).
\item There is a 2-morphism 
\begin{equation} \pounds_{\frac{d}{dt}}(\cdot):\pounds_{\frac{d}{dt}}[\tilde{\A}^0_{\ul{\T}_{M^1},rel}]\toocong \tilde{\A}^0_{\ul{i\R}_{M^1},rel}\circ \frac{d}{dt}\log[\cdot].\end{equation}
\item There is a 2-morphism 
\begin{equation} i^*[\tilde{\A}^0_{\ul{\T}_{M^1},rel}]\toocong i_*(\tilde{\A}^0_{\ul{\T}_M}\circ i^*[\cdot]).\end{equation}
\end{enumerate}
\end{lemma} 

Recall   the construction of the functor $D:\L_{\CC}(\Phi)\to \L_{\CC}(\xi)$ in definition (\ref{def diff fun}): for each global section $S\in \L_{\CC}(\Phi)$ of $\delta \CC$, we constructed the corresponding infinitesimal lift $F_{\tx}=D(S)$ as a composition of 1-morphisms 
\begin{equation}\xymatrix{ \CC^{op} \ar[r]^-{\Phi_1}& \nu_*(\delta\CC)\ar[r]^{\Phi_2} & \nu_*\textrm{B}(\ul{\T}_{M^1}) \ar[r]^{\Phi_3} & \nu_*\textrm{B}(\ul{i\R}_{M^1}) \ar[r]^{\Phi_4} & \BiR. }\end{equation}  Now, given $(S,\mu_S)\in \L_{(\CC,\A)}(\Phi)$, we wish to construct a 2-morphism 
\begin{equation} \Theta_{\cx}:-\pounds_{\xi}[\tilde{\A}]\to \tilde{\A}^0_{\ul{i\R}_M}\circ F_{\tx} \end{equation}  To do so, consider the diagram 
\begin{equation}
\xymatrix{ \CC^{op}  \ar[r]^{\tilde{\A}} \ar[d]_{\Phi_1} \ar@{=>}[dr] &  \textrm{B}(i\Omega^1_M) \ar[d]^{\delta_{i\Omega^1}}  \\ 
\nu_*(\delta\CC^{op}) \ar[r]^{\nu_*\delta\tilde{\A}} \ar[d]_{\Phi_2} \ar@{=>}[dr] & \nu_*\textrm{B}(i\Omega^1_{M^1,rel}) \ar[d]^{id} \\ 
\nu_*\textrm{B}(\ul{\T}_{M^1}) \ar[r]^{\nu_*\tilde{\A}^0_{\ul{\T}_{M^1},rel}} \ar[d]_{\Phi_3} \ar@{=>}[dr] & \nu_*\textrm{B}(i\Omega^1_{M^1,rel}) \ar[d]^{\pounds_{\frac{d}{dt}[\cdot]}}\\
\nu_*\textrm{B}(\ul{i\R}_{M^1}) \ar[r]^{\nu_*\tilde{\A}^0_{\ul{\i\R}_{M^1},rel}} \ar[d]_{\Phi_4} \ar@{=>}[dr] & \nu_*\textrm{B}(i\Omega^1_{M^1,rel})\ar[d]^{i^*[\cdot]} \\
\textrm{B}(\ul{i\R}_M) \ar[r]^{\tilde{\A}^0_{\ul{i\R}_M}}& B(i\Omega^1_M). }\end{equation}  In this diagram double arrows denote the 2-morphisms from lemma (\ref{delta 2 morphism}) and lemma (\ref{conn str diff}).  The counter-clockwise composition of the outer arrows from $\CC^{op}$ to $B(\ul{i\Omega}^1_M)$ is by definition 
\begin{equation} \tilde{\A}^0_{\ul{i\R}_M}\circ F_{\tx}.\end{equation} On the other-hand, consider the vertical composition of 1-morphisms from $B(\ul{i\Omega}^1_M)$ in the upper right-hand corner to $B(\ul{i\Omega}^1_M)$ in the lower right-hand corner, which we call $\Psi$.  By composing 2-morphisms we obtain a 2-morphism 
\begin{equation}\label{comp 2 morphism} \Psi\circ\tilde{\A}\toocong \tilde{\A}^0_{\ul{i\R}_M}\circ F_{\tx}.\end{equation}
Note that $\Psi$ intertwines
\begin{equation} -i^*\circ\pounds_{\frac{d}{dt}}\circ \delta_{i\Omega^1}:i\Omega^1_M\to i\Omega^1_M.\end{equation}  By definition, this takes a 1-form $\alpha$ on an open set $U\subset M$ to 
\begin{equation} -(\pounds_{\frac{d}{dt}}(p_0^*\alpha-p_1^*{\alpha}))|_U=-\pounds_{\xi}\alpha.\end{equation}  Therefore, by proposition (\ref{ass gerbe coh}) we have a canonical 2-morphism from \begin{equation}-\pounds_{\xi}[\cdot]\toocong \Psi.\end{equation}  Combining this 2-morphism with the 2-morphism (\ref{comp 2 morphism}), we obtain the desired 2-morphism
\begin{equation} -\pounds_{\xi}[\A]\toocong \A^0_{\ul{i\R}_M}\circ F_{\tx}.\end{equation}  If we denote this 2-morphism by $\Theta_{\cx}$, it is straightforward to check that the assignment $D:(S,\mu_S)\mapsto (F_{\tx},\Theta_{\cx})$ is functorial.

\section{Appendix}  In this section we recall the precise definition of the 2-category of stacks over a manifold $M$.  We then use these definitions to prove theorems (\ref{cech equiv}) and (\ref{cech equiv.  conn}).
\subsection{Gerbes as sheaves of groupoids}

\begin{definition}\label{def groupoid presheaf} Let $M$ be a smooth manifold.  A \emph{presheaf of groupoids} $\CC$ over $M$ consists of the following.
\begin{enumerate}  \item For every open set $U\subset M$, a groupoid $\CC(U)$.
\item For every inclusion of open sets $i:V\to U$, a functor $i^*:\CC(U)\to \CC(V)$. \item For every sequence of inclusions of open sets
\begin{equation} \xymatrix{W \ar[r]^j &  V \ar[r]^i &  U } \end{equation} 
 a natural transformation $\alpha_{i,j}:j^*i^*\Rightarrow (ij)^*,$ such that for every sequence of inclusions 
\begin{equation} \xymatrix{ T \ar[r]^k & W \ar[r]^j & V \ar[r]^i & U} 
 \end{equation} 
the diagram \begin{equation}\label{presheaf coh} \xymatrix@C15mm{ k^*j^*i^* \ar@{=>}[r]^{1_{k^*}*\alpha_{i,j}} \ar@{=>}[d]_{\alpha_{j,k}*1_{i^*}} & k^*(ij)^* \ar@{=>}[d]^{1_{k^*}\alpha_{ij,k}} \\ (jk)^*i^* \ar@{=>}[r]_{\alpha_{i,jk}} & (ijk)^*}  \end{equation} commutes.
\end{enumerate}
\end{definition}
\begin{notation} The notation used in diagram (\ref{presheaf coh}) is the following: given categories $A,B,C$, functors $F,G:A\to B$, $H,I:B\to C$ and natural transformations $\alpha:F\Rightarrow G$, $\beta:H\Rightarrow I$, we denote by $\beta*\alpha:HF\Rightarrow IG$  the horizontal composition of $\beta$ with $\alpha$.
\end{notation}

\begin{notation} When we wish to be more explicit, we will sometimes label the functors $i^*$ and the natural transformations $\alpha_{i,j}$ by $\CC$; thus a presheaf $\CC$ consists of a triple $\{\{\CC(U)\},\{i^*_{\CC}\},\{\alpha_{\CC,i,j}\}\}$ subject to the conditions given in (\ref{def groupoid presheaf}).
\end{notation}
\begin{rem} Let $\textbf{Open}(M)$ denote the category whose objects are open subsets of $M$, and whose morphisms are inclusions. Then we can state definition (\ref{def groupoid presheaf}) more succinctly by saying that a presheaf of groupoids over $M$ is a pseudo-functor from $\textbf{Open}(M)^{op}$ to the 2-category $\textbf{Gpd}$ of groupoids, where the former is thought of as a 2-category with only identity 2-morphisms (for the relevant definitions see \cite{Bo}).
\end{rem} 

Given an open subset $U\subset M$ and objects $P,Q\in \CC(U)$, it is an easy consequence of definition (\ref{def groupoid presheaf}) that there is a presheaf (of sets) $\ul{\Hom}(P,Q)$ over $U$ which assigns to each inclusion $i:V\hookrightarrow U$ the set $\ul{\Hom}(i^*P,i^*Q)$.  The presheaf$\CC$ is said to be a \emph{prestack} if $\ul{\Hom}(P,Q)$ is a sheaf for all open sets $U$ and all objects $P$ and $Q$.   We are interested in prestacks that satisfy an additional gluing property.  Let $\{U_i\}$ be an open cover of an open subset $V\subset M$, and denote the n-fold intersection $U_{i_1}\cap\cdots\cap U_{i_n}$ by $U_{i_1\cdots i_n}$.  Let $\{Q_i\}\in \CC(U_i)$ and $s_{ij}:Q_i|_{U_{ij}}\to Q|_{U_{ij}}$.  Note that on triple intersection $U_{ijk}$, we have two morphisms from $Q_i|_{U_{ijk}}\to Q_j|_{U_{ijk}}$.  On the one hand, we can compose 
\begin{equation} \xymatrix{ Q_i|_{U_{ijk}} \ar[r]^{\cong} & (Q_i|_{U_{ik}})|_{U_{ijk}} \ar[r]^{s_{ik}|_{U_{ijk}}} &  (Q_k|_{U_{ik}})|_{U_{ijk}}  \ar[r]^{\cong} & Q_k|_{U_{ijk}},}\end{equation} where the first an last isomorphisms are constructed using the natural transformations $\alpha$ described in definition (\ref{def groupoid presheaf}). Similarly, we can first use $s_{ij}$ to obtain a morphism from $Q_i|_{U_{ijk}}$ to $Q_j|_{U_{ijk}}$, and then compose with the morphism from $Q_j|_{U_{ijk}}$ to $Q_k|_{U_{ijk}}$ determined by $s_{jk}$.  
We say that the collection $\{Q_i\}$ $\{s_{ij}\}$ satisfies the \emph{descent condition} if these two morphisms are equal for each $i,j,k$.  From now on, for notational simplicity we suppress the restriction maps and natural transformations (see remark 1.2 in \cite{He}) and simply write the descent condition as 
\begin{equation}\label{descent condition objects} s_{jk}\circ s_{ij}=s_{ik}.\end{equation}  Similarly, given another collection $\{\{Q_i'\}\{s'_{ij}\}\}$ satisfying the descent condition, we say that a collection of morphisms $\{\psi_i:Q_i\to Q'_i\}$ satisfy the descent condition if we have 
\begin{equation} \label{descent condition morphisms}\psi_js_{ij}=s'_{ij}\psi_i \end{equation} on double overlaps.  Overall we obtain a \emph{descent category} $\textrm{Desc}(\CC,\{U_i\})$.  It is not hard to construct a restriction functor from $\CC(U)\to \textrm{Desc}(\CC,\{U_i\})$ using the structure described in definition (\ref{def groupoid presheaf}).
\begin{definition} A prestack $\CC$ is a \emph{stack} if for each open set $V\subset M$ and each open cover $\{U_i\}$ of $V$, the restriction functor $\CC(V)\to \textrm{Desc}(\CC,\{U_i\})$ is an equivalence of categories.
\end{definition}

\begin{definition}\label{stack morphism} Let $\CC$ and $\CC'$ be presheaves of groupoids over $M$.  A \emph{1-morphism} $\Phi:\CC\to \CC'$ consists of the following.
\begin{enumerate} \item For every open set $U\subset M$, a functor $\Phi_U:\CC(U)\to \CC'(U)$. 
\item For every inclusion of open sets $i:V\to U$, a natural transformation 
\begin{equation} \label{morph nat tran}\Phi_i:i^*_{\CC'}\Phi_U\Rightarrow\Phi_V i^*_{\CC}, \end{equation} such that for every sequence of inclusions
\begin{equation} \xymatrix{ W\ar[r]^j & V \ar[r]^i &  U }\end{equation} the diagram 
\begin{equation} \label{mor nat tran ass} \xymatrix@C15mm{ j_{\CC'}^*i_{\CC'}^*\Phi_U \ar@{=>}[r]^{1_{j^*_{\CC'}}*\Phi_i} \ar@{=>}[d]_{\alpha_{\CC',i,j}*1_{\Phi_U}} & j^*_{\CC'}\Phi_Vi^*_{\CC} \ar@{=>}[r]^{\Phi_{j}*1_{i^*_{\CC}}} & \Phi_W j^*_{\CC}i^*_{\CC} \ar@{=>}[d]^{1_{\Phi_W}*\alpha_{\CC,i,j}} \\ (ij)^*_{\CC'}\Phi_U \ar@{=>}[rr]_{\Phi_{ij}} & & \Phi_W(ij)^*_{\CC} } \end{equation} commutes.
\end{enumerate}
\end{definition}
\begin{rem} In the language of 2-categories, definition (\ref{stack morphism}) says that a 1-morphism from $\CC$ to $\CC'$ is a pseudo-natural transformation between the pseudo-functors corresponding to $\CC$ and $\CC'$(again, see \cite{Bo}).
\end{rem}
\begin{definition} \label{2-morphism} Let $\CC$, $\CC'$ be presheaves of groupoids over $M$, and let $\Phi,\Psi:\CC\to \CC'$ be 1-morphisms.  A \emph{2-morphism} $\tau:\Phi\Rightarrow \Psi$ is a collection of natural transformations $\tau(U):\Phi_U\Rightarrow \Psi_U$ for every open set $U\subset M$, such that for each inclusion of open sets $i:V\hookrightarrow U$ the diagram
\begin{equation}\label{2-morphisms ass} \xymatrix@C15mm{ i^*_{\CC'}\Phi_U \ar@{=>}[d]_{1_{i^*_{\CC'}}*\tau(U)} \ar@{=>}[r]^{\Phi_i} & \Phi_V i^*_{\CC} \ar@{=>}[d]^{\tau(V)*1_{i^*_{\CC}}} \\ i^*_{\CC'}\Psi_U \ar@{=>}[r]_{\Psi_i} & \Psi_V i^*_{\CC} }\end{equation} commutes.
\end{definition}

\begin{rem} Again in the language of 2-categories, a 2-morphism from $\Phi$ to $\Psi$ is a modification of pseudo-natural transformations.\cite{Bo}
\end{rem}

\subsection{Proofs of theorems (\ref{cech equiv}) and (\ref{cech equiv.  conn}).}  

\textit{Proof of (\ref{cech equiv})}:  It will be more convenient to construct equivalences 
\begin{equation} \Gamma:\L_{g_{ijk}}(\xi)\to \tilde{\L}_{\CC}(\xi)\end{equation} for each vector field $\xi$,  where we recall from appendix A that 
\begin{equation}\tilde{\L}_{\CC}(\xi)=\ul{\Hom}_{\iota_{\xi}d\log}(\CC^{op},\textrm{B}(\ul{i\R}_M))\end{equation} is the category of contravariant functors from $\CC$ to $\textrm{B}(\ul{i\R}_M)$ intertwining $\iota_{\xi}d\log$.  We may then use the canonical equivalence between $\tilde{\L}_{\CC}(\xi)$ and $\L_{\CC}(\xi)$.

 Given an object $\{f^{\tx}_{ij}\}\in \L_{g_{ijk}}(\xi)$, let us first construct an object $F_{\tx}=\Gamma(\{f^{\cx}_{ij}\})\in \tilde{\L}_{\CC}(\xi)$.   Given an open set $V\subset M$ and $P\in \CC(V)$, let $V_i=V\cap U_i$,  $P_i=P|_{V_i}$, $E_i(P)=\ul{\Hom}(P_i,Q_i|_{V_i})$, and
\begin{equation} F_i(P)=\iota_{\xi}d\log[E_i(P)].\end{equation}   Using the natural isomorphism $P_i|_{V_i\cap V_j} \tocong P_j|_{V_i\cap V_j}$ and the morphisms $s_{ij}:Q_i\to Q_j$ we obtain isomorphisms\footnote{for notational simplicity, from now one we will supress writing restriction maps explicitly.}
\begin{equation} (s_{ij})_*:E_i(P)\to E_j(P). \end{equation}  If we then define 
\begin{equation} \zeta_{ij}=\iota_{\xi}d\log[(s_{ij})_*]-f^{\tx}_{ij}:F_i(P)\to F_j(P),\end{equation}  it follows from (\ref{cocycle definition}) and (\ref{lift delta}) that the the morphisms $\{\zeta_{ij}\}$ satisfy the descent condition, and we therefore define $F_{\tx}(P)\in \tor_{\ul{i\R}_V}$ to be the torsor obtained by gluing.  Similarly, given a morphism $\psi: P\to P'$ in $\CC(U)$, the morphisms
\begin{equation} \iota_{\xi}d\log[\psi^*]:F_i(P')\to F_i(P)\end{equation} glue to give a morphism $F_{\tx}(\psi):F_{\tx}(P')\to F_{\tx}(P)$, and for each $\T$-valued function $g$ we have $F_{\tx}(\psi\cdot g)=F_{\tx}(\psi)+\iota_{\xi}d\log(g)$.  It is then straightforward to construct the restriction natural isomorphisms $F_{\tx,i}$ described as part of definition (\ref{stack morphism}).  Thus, given $\{f^{\tx}_{ij}\}\in \L_{g_{ijk}}(\xi)$, we have produced a lift $F_{\tx}=\Gamma(\{f^{\tx}_{ij}\})\in \tilde{\L}_{\CC}(\xi)$.

Next, given  an isomorphism $\{u_i\}$ from $\{f^{\tx}_{ij}\}$ to $\{f^{\tx'}_{ij}\}$ in $\L_{g_{ijk}}$, we wish to produce an isomorphism $\Gamma(\{u_i\}):F_{\tx}\tocong F_{\tx'}$.  For each $i$, let $(\tau_P)_i$ be the automorphism of the $\ul{i\R}_{V_i}$-torsor $\iota_{\xi}d\log[E_i(P)]$ corresponding to the function $u_i|_{V_i}.$  It then follows from equation (\ref{cech morphism}) that the diagram 
\begin{equation} \xymatrix{ \iota_{\xi}d\log[E_i(P)] \ar[r]^{\zeta_{ij}} \ar[d]_{(\tau_P)_i} & \iota_{\xi}d\log[E_j(P)] \ar[d]^{(\tau_P)_j} \\ \iota_{\xi}d\log[E_i(P)] \ar[r]_{\zeta'_{ij}} & \iota_{\xi}d\log[E_j(P)] } \end{equation} commutes, so that the morphisms $(\tau_P)_i$ satisfy the descent condition (\ref{descent condition objects}) and thus glue to define an isomorphism $(\tau)_P:F_{\tx}(P)\tocong F_{\tx'}(P)$.  It is easily checked that this isomorphism is natural in $P$, so that the condition (\ref{2-morphisms ass}) is satisfied, and that $\Gamma$ is functorial. 

To verify that $\Gamma$ is an equivalence of categories, we must check that it is both essentially surjective and fully faithful.  Given $F_{\tx}\in \L_{\CC}(\xi)$, choose local trivializations $\{r_i\}$ with  corresponding Cech data idefined by 
\begin{equation} F_{\tx}(s_{ij})(r_j)=r_i+f^{\tx}_{ij}.\end{equation} Let $F'_{\tx}=\Gamma(\{f^{\tx}_{ij}\})$.    For each object $P\in \CC(U)$, we wish to define isomorphisms
\begin{equation}  \tau_{P,i}:F_{\tx}(P)\to F_i(P) \end{equation} such that on overlaps $U_{ij}$ we have 
\begin{equation}\label{triangle} \xymatrix{ F_{\tx}(P) \ar[dr]_{\tau_{P,j}} \ar[r]^{\tau_{P,i}} & F_i(P) \ar[d]^{\zeta_{ij}}  \\  & F_j(P).}\end{equation}  To do so, for each $i$ choose $\psi_i:P\to Q_i$, and let $\tau_{P,i}$ be the unique morphism
\begin{equation} F_{\tx}(P)\to F_i(P)=\iota_{\xi}d\log[\ul{\Hom}(P,Q_i)] \end{equation} taking 
\begin{equation} F_{\tx}(\psi)(r_i)\mapsto \iota_{\xi}d\log[\psi].\end{equation}  To see that this is independent of the choice of $\psi$, a simple computation shows that for each $\T$-valued function $h$ we have 
\begin{equation} \tau_{P,i}(F_{\tx}(\psi\cdot h)(r_i))= \iota_{\xi}d\log[\psi\cdot h].\end{equation} Similarly, it is easily checked that diagram (\ref{triangle}) commutes.  This shows essentially surjectivity.  Since any two objects in 
 $\L_{g_{ijk}}$ are isomorphic, to show that $\Gamma$ is fully faithful it is sufficient to check that for every object $\{f^{\tx}_{ij}\}\in \L_{g_{ijk}}$ that $\Gamma$ induces a bijection from $\Aut(\{f^{\tx}_{ij}\})$ to $\Aut(\Gamma(\{f^{\tx}_{ij}\})$.  Let $\tau$ be an automorphism of $F_{\tx}=\Gamma(\{f^{\tx}_{ij}\})$.  Then for each object $P\in \CC(U)$, $\tau_P$ corresponds to a function $f_P:U\to i\R$ via the isomorphism $\alpha_P$ in definition (\ref{def gerbe}).  Furthermore, given another object $Q\in \CC(V)$, definition (\ref{stack morphism}) implies that $f_P$ and $f_Q$ agree on their common domain, and thus we obtain a global function $f_{\tau}:M\to i\R$; conversely, given such a function, we can construct an automorphism of $F_{\tx}$.  Finally, by construction it is clear that $\Gamma:\Aut(\{f^{\tx}_{ij}\})\to \Aut(F_{\tx})$ is consistent with the identification of both groups with $C^{\infty}_M(i\R)$, and is therefore a bijection.

\textit{Proof of (\ref{cech equiv.  conn}):}  We will use the same notation as that in the proof of (\ref{cech equiv}).  Given $(\xi\{f^{\cx}_{ij}\},\{a_i\})\in \L_{(g_{ijk},A_{ij})}$, let $F_{\cx}=\Gamma((\xi,\{f^{\cx}_{ij}\})$ be the (non-connective lift) constructed in the proof of (\ref{cech equiv})) from $(\xi,\{f^{\cx}_{ij}\})$  In order to extend $F_{\cx}$ to a connective lift, for each $P\in \CC(V)$ and each $\mu\in \A(P)$, we must construct a connection $\Theta_{\tx}(\mu)$ on $F_{\cx}(P)$.  Recall that $F_{\cx}(P)$ is constructed by gluing the $\ul{i\R}_{V_i}$-torsors 
\begin{equation} F_i(P)=\iota_{\xi}d\log[\ul{\Hom}(P,Q_i)] \end{equation} via maps $\zeta_{ij}:F_i(P)\to F_j(P)$.  Thus to specify a connection on $F_{\tx}(P)$, we must specify connections $\nu_i$ on $F_i(P)$ such that on overlaps we have $\nu_j=(\zeta_{ij})_*\nu_i$.  Using lemma (\ref{conn str diff}), we obtain  a connection $\mu_i-\mu_P$ on $\ul{\Hom}(Q_i,P)$, and by lemma (\ref{triv conn lift})  we may define a connection $\nu_i$  on $F_i(P)$ as 
\begin{equation} \nu_i=\pounds_{\xi}(\mu_i-\mu_P)+a_i.\end{equation}  We then have 
\begin{align} (\zeta_{ij})_*(\nu_i) & = (\iota_{\xi}d\log[(s_{ij})_*-f^{\cx}_{ij})_*(\pounds_{\xi}[\mu_i-\mu_P]+a_i) \\  = & \A^0_{\ul{i\R}}(\iota_{\xi}d\log[(s_{ij})_*])\pounds_{\xi}(\mu_i-\mu_P)+df^{\cx}_{ij}+a_i \notag \\
= & \pounds_{\xi}((s_{ij})_*(\mu_i)-\mu_P)+df^{\cx}_{ij}+a_i \notag \\
= & \pounds_{\xi}[\mu_j-A_{ij}-\mu_P]+df^{\cx}_{ij}+a_i \notag \\ 
= &  \pounds_{\xi}(\mu_j-\mu_P)+a_j-(a_j-a_i)+df^{\cx}_{ij}-\pounds_{\xi}A_{ij} \notag \\= & \nu_j.\notag\end{align}  We then let $\Theta_{\cx}(\mu)$ be the connection obtained by gluing; by construction we have $\Theta_{\cx}(\mu+\alpha)=\Theta_{\cx}(\mu)-\pounds_{\xi}\alpha$.  It is straightforward to check that $\Theta_{\cx}$ is compatible with restrictions and is suitably natural in $P$.

Next, suppose are given an isomorphism $\{u_i\}$ from $(\{f^{\cx}_{ij}\},\{a^{\cx}_i\})$ to $(\{f^{\cx'}_{ij}\},\{a^{\cx'}_i\})$.  Recall from the proof of (\ref{cech equiv}) that we obtain a morphism $\Gamma(\{u_i\}):F_{\tx}(P)\to F_{\tx'}(P)$ by letting each $u_i$ act as an automorphism of $F_i(P)$.  The condition $a_i^{\cx'}=a_i^{\cx}+du_i$ implies that $(u_i)_*\nu_i=\nu_i'$ for each $i$.  Therefore $\Gamma(\{u_i\})_*\Theta_{\cx}(\mu)=\Theta_{\cx'}(\mu)$, so that $\Gamma(\{u_i\})$ is compatible with the connections and thus defines a connective equivalence.  Mimicking the arguments used in the proof of (ref{cech equiv}) above, it is then easily checked that $\Gamma$ is both essentially surjective and fully faithful.

\end{document}